\documentclass[12pt]{article}

\usepackage{amsmath} 
\usepackage{amsthm} 
\usepackage{amsfonts} 
\usepackage{amssymb} 
\usepackage{stmaryrd} 
\usepackage{mathtools} 

\usepackage{bm} 
\usepackage{bbm} 

\usepackage{array} 
\usepackage[inline]{enumitem} 
\usepackage{microtype}

\usepackage{tikz} 
\usetikzlibrary{cd} 
\usetikzlibrary{decorations.markings} 

\newtheoremstyle{thmstyle}
  {\medskipamount}
  {\smallskipamount}
  {\slshape}
  {0pt}
  {\bfseries}
  {.}
  { }
  {\thmname{#1}\thmnumber{ #2}{\normalfont\thmnote{ (#3)}}}

\newtheoremstyle{plainstyle}
  {\medskipamount}
  {\smallskipamount}
  {\rmfamily}
  {0pt}
  {\bfseries}
  {.}
  { }
  {\thmname{#1}\thmnumber{ #2}{\normalfont\thmnote{ (#3)}}}

\theoremstyle{thmstyle}
\newtheorem{theorem}{Theorem}[section]
\newtheorem{lemma}[theorem]{Lemma}

\newtheorem{proposition}[theorem]{Proposition}

\newtheorem{claim}[theorem]{Claim}

\theoremstyle{plainstyle}
\newtheorem{definition}[theorem]{Definition}
\newtheorem{remark}{Remark}

\newtheorem{example}{Example}

\newenvironment{proofof}[1]{\begin{proof}[Proof of #1.]}{\end{proof}}

\setlist[enumerate]{label={\roman*.}, ref={(\roman*)}}

\newcommand{\rn}{\bm}
\newcommand{\df}{\stackrel{\text{def}}{=}}
\newcommand{\place}{\mathord{-}}
\newcommand{\cat}{\textsc}
\newcommand{\disjcup}{\mathbin{\stackrel\cdot\cup}}

\def\symdiff{\mathbin{\triangle}}
\newcommand{\comp}{\mathbin{\circ}}
\newcommand{\rest}{\mathord{\vert}}

\newcommand{\function}[2]{\colon #1 \rightarrow #2}
\newcommand{\injection}[2]{\colon #1 \rightarrowtail #2}
\newcommand{\interpret}[2]{\colon #1 \leadsto #2}

\newcommand{\set}[2]{\left\{\hspace{0.2ex} #1 \:\middle\vert\: #2\right\}}
\newcommand{\rk}{\text{rk}}

\DeclareMathOperator{\Homop}{Hom}

\newcommand{\HomT}[1]{\Homop^+(\cA[#1],\RR)}

\DeclareMathOperator{\Aut}{Aut}
\def\Tind{T_{\operatorname{ind}}}
\def\tind{t_{\operatorname{ind}}}
\def\Dopen{D_{\operatorname{open}}}
\DeclareMathOperator{\im}{im}

\DeclareMathOperator{\sgn}{sgn}
\DeclareMathOperator{\id}{id}
\DeclareMathOperator{\Th}{Th}

\newcommand{\NN}{\mathbb{N}}
\newcommand{\PP}{\mathbb{P}}
\newcommand{\RR}{\mathbb{R}}

\newcommand{\One}{\mathbbm{1}}

\newcommand{\cA}{\mathcal{A}}

\newcommand{\cC}{\mathcal{C}}
\newcommand{\cE}{\mathcal{E}}

\newcommand{\cG}{\mathcal{G}}
\newcommand{\cH}{\mathcal{H}}
\newcommand{\cK}{\mathcal{K}}
\newcommand{\cL}{\mathcal{L}}
\newcommand{\cM}{\mathcal{M}}
\newcommand{\cN}{\mathcal{N}}

\def\prop{\texttt}
\def\CliqueDisc{\prop{CliqueDisc}}
\def\Disc{\prop{Disc}}
\def\Dev{\prop{Dev}}
\def\Independence{\prop{Independence}}
\def\UCouple{\prop{UCouple}}
\def\UInduce{\prop{UInduce}}

\newcommand{\TcColoring}[1][c]{T_{#1\operatorname{-Coloring}}}

\newcommand{\TLinOrder}{T_{\operatorname{LinOrder}}}

\newcommand{\TGraph}{T_{\operatorname{Graph}}}
\newcommand{\TkHypergraph}[1][k]{T_{#1\operatorname{-Hypergraph}}}
\newcommand{\explicitTcColkHyp}[2]{T_{#1,#2}}
\newcommand{\TcColkHyp}[1][{c}{k}]{\explicitTcColkHyp#1}

\newcommand{\TkTournament}[1][k]{T_{#1\operatorname{-Tournament}}}

\title{Natural quasirandomness properties}
\author{Leonardo N.~Coregliano\thanks{University of Chicago, {\tt lenacore@uchicago.edu}} \and Alexander A.~Razborov\thanks{University of Chicago,
{\tt razborov@math.uchicago.edu}, and Steklov Mathematical Institute, Moscow}
}
\begin{document}
\maketitle

\begin{abstract}
  The theory of quasirandomness has greatly expanded from its inaugural graph theoretical setting to several
  different combinatorial objects such as hypergraphs, tournaments, permutations, etc. However, these
  quasirandomness variants have been done in an ad-hoc case-by-case manner.
In this paper, we propose three new hierarchies of quasirandomness
  properties that can be naturally defined for arbitrary combinatorial
  objects. Our properties are also ``natural'' in more formal sense: they
  are preserved by local combinatorial constructions (encoded by open
  interpretations). We show that our quasirandomness properties have
several different but equivalent characterizations that are similar to
  hypergraph quasirandomness properties. We also prove several implications
  and separations comparing them to each other and to what has been known
  for hypergraphs.

  The main notion explored by our statements and proofs is that of unique
  coupleability: two limit objects are uniquely coupleable if there is a
  unique limit object in the combined theory that is an alignment (i.e., a
  coupling) of these two objects.
\end{abstract}

\section{Introduction}

The theory of graph quasirandomness introduced by Thomason~\cite{Tho85} and
Chung--Graham--Wilson~\cite{CGW89} studies deterministic graphs that look
random. The main discovery of this theory is that several properties that
hold asymptotically almost surely for the sequence of Erd\H{o}s--R\'{e}nyi
random graphs $(\rn{G_{n,p}})_{n\in\NN}$ are equivalent when rephrased as properties of a
deterministic graph sequence $(G_n)_{n\in\NN}$. Since then, the theory of
quasirandomness has expanded not only within graph
theory~\cite{CG92,SS97,SS03,Sha08,Yus10,Jan11,CFS18} but also towards studying
quasirandomness for other combinatorial objects such as
tournaments~\cite{CG91,KS13,CR17}, permutations~\cite{Coo04,Coo05,KP13,CKNPSV20} and
hypergraphs~\cite{CG90,Chu90,KRS02,KNRS10,DR11,LM15b,LM17,Tow17,ACHP18}.

The theory of quasirandomness was one of the motivations and driving forces behind
the theory of dense limits of combinatorial objects (we
refer the reader to~\cite{Lov12} for the case of graphs and to~\cite{Aus08,AC14,CR19}
for the general case). The starting point of the latter theory is that if
$(N_n)_{n\in\NN}$ is a sequence of combinatorial objects such that for every
fixed combinatorial object $M$, the normalized number of (unlabeled induced)
copies $p(M,N_n)$ of $M$ in $N_n$ converges to some limit $\phi(M)$, then the
sequence $(N_n)_{n\in\NN}$ can alternatively be represented as a limit object
that captures all these limit values. But as the theory of graph (and other)
limits has been maturing, and in particular after the uniqueness theorem was
proved in~\cite{BCL10} (see~\cite[Theorem~13.10]{Lov12} for graphs
and~\cite[Theorem~3.9]{CR19} for the general case), it has turned out that in
a sense this theory transcends counting. Namely, limit objects can be
described, up to an appropriate notion of isomorphism (or, as Lov\'{a}sz
dubbed it, {\em cryptomorphism}), using different languages and quite
different kinds of mathematics and statistics (\cite[Theorem~11.52]{Lov12}
and~\cite[Theorem~6.3]{CR19}) and only one of those descriptions is based on
sampling statistics $p(M,\place)$ per se~\cite{Raz07}. Arguably, it is this
versatility that is largely responsible for the wide spread of graph limits
and their connections to many other areas.

\smallskip
The situation with quasirandomness remains somewhat different, and we are
aware only of a few attempts to study it intrinsically, that is, based on
principles other than counting. One of the equivalent properties in the
seminal paper~\cite{CGW89} ($P_3$) was of spectral nature, namely it
requested the second largest eigenvalue of $G_n$ to be $o(\lvert G_n\rvert)$. This spectral
theme was further continued for (linear) quasirandom hypergraphs
in~\cite{LM15a,LM17}.

Even though most other quasirandomness properties in the literature are
stated in terms of counting, it is still possible to extract from them
something intrinsic. For example, the property $P_4$ in~\cite{CGW89} (see
also~\cite[Theorem~2.4]{SS97}) implies that quasirandom limits $W$ are the
only graphons with the following unique inducibility property: if
$(G_n)_{n\in\NN}$ converges to $W$ then the sequence of induced graphs
$(G_n\rest_{U_n})_{n\in\NN}$ also converges to $W$ as long as $\lvert
U_n\rvert\geq\Omega(\lvert G_n\rvert)$. As another example, using graphon language~\cite{LS06}, we can
extract a trivial intrinsic characterization of quasirandom limits in terms
of an independence property: a graphon $W\function{[0,1]^2}{[0,1]}$ is
quasirandom if and only if $W$ a.e.\ does not depend on its variables, that
is, it is a.e.\ constant.

\bigskip
In this paper we attempt to initiate a more systematic study of quasirandom
properties that can be reasonably identified as ``intrinsic'' (for reasons
that will become clear very shortly, we will also use in this context the
word ``natural''), and let us first explain what we roughly mean by this. Our
explanation will be deliberately informal and open-ended; instead of trying
to give a rigorous definition, we present a set of tests that in our view
have to be passed and then describe some concrete properties we will be
studying in this paper that pass these tests.

First and foremost, we view this paper as a continuation of~\cite{Raz07,CR19}, which in particular implies that we require qualifying properties to be formulated in an uniform way for arbitrary universal theories in a finite relational language. For examples of what can be expressed in that language see~\cite[Sct.~2.1 and Sct.~7]{CR19}.

The next two requirements are somewhat derivative of the first.

We require that the property should not refer to densities of concrete models
and their explicit values (thus, this is more about the \emph{formulation} of
the property than the class of objects defined by it.) The reason is that any
such definition is necessarily somewhat arbitrary. For example, there is no
such thing as ``edge densities'' in the theories of tournaments and
permutations so their ad hoc analogues had to be found when defining
quasirandom objects in those contexts. Of the quasirandom graph properties
mentioned above, the description as a constant graphon definitely satisfies
this criterion, and so does the inducibility property (the tweak of $P_4$ in~\cite{CGW89}). Spectral properties also pass the test but unfortunately they
fail (given our current state of knowledge) the previous universality test.

The next requirement is that we want the property to be preserved under open
interpretations, and this is where the word ``natural'' (like in ``natural
transformations'' -- open interpretations do form a category~\cite[Sct.~2.2]{CR19}) comes in. In plain words, everything that can be syntactically
defined in a quasirandom object must display proportionally strong quasirandom
properties. Again, in an implicit form this requirement was exploited in the
previous literature both in positive and negative manner. For example, the
proofs of the implications $P_{10}\Longrightarrow P_{11}\Longrightarrow P_1(s)$ in the
seminal paper on quasirandom tournaments~\cite{CG91} can be viewed as
divided into two parts. First one proves that all ``couplings'' of a
quasirandom graph with a linear ordering are the same and hence completely
determined by the random coupling. Then the tournament obtained from the
resulting quasirandom ordered graph via the ``arc-orientation''
interpretation must be quasirandom. This example is paradigmatic for many
parts of our paper. As for ``negative'' use, let us note that most
separations in the hierarchy of quasirandom hypergraphs~\cite{ACHP18,LM15b,Tow17} can be viewed as coming from the fact that these
properties are {\em not} preserved under open interpretations between the
theories of hypergraphs of possibly different arity. We will elaborate on
this in Section~\ref{sec:separations} (see Theorem~\ref{thm:separationDevUInduce}).

Our final requirement is more ``traditional'', and it is well-rooted in the
previous literature. Namely, we require that the property should be satisfied
asymptotically almost surely for some ``natural'' random model of some
``natural'' theory $T$. Examples of ``natural'' random models include, of
course, the Erd\H{o}s--R\'{e}nyi model and its generalization to hypergraphs,
the random tournament, the random permutation, etc.

\bigskip
This list of requirements may appear to be rather restrictive, so let us describe quasirandom properties we are studying;
they are essentially far-reaching generalizations of what we already
discussed above. Several more remarks are in place before we begin.
\begin{enumerate}[label={\arabic*.}]
\item We have deliberately decided against attempting to state our
    properties in the language of finite combinatorial objects and their
    asymptotic behavior -- it is probably possible but the result might be
    rather ugly and disappointing. Instead, we use the language of
    graphons~\cite{LS06}, hypergraphons~\cite{ES12} and theons~\cite{CR19}
    for the geometric view of our objects and that of flag
    algebras~\cite{Raz07} for a concise algebraic description. We remark
    that we are not the first authors to make this election, and the
    advantages of using the continuous setting are illustrated by the fact
    that such proofs are often more elegant and less technical than their
    finite world counterparts~\cite{Jan11,KP13,Tow17}. This view is more
    instructive, too: for example, by looking back through the lenses of
    graphons, we can extract an elegant graphon proof of quasirandomness of
    property $P_2(4)$ of~\cite{CGW89} based on the Lebesgue Density Theorem
    from a paper as early as~\cite[Theorem~3.10]{DF81}.

However, for the benefit of more combinatorially-oriented reader we try to inject as much of ``finite intuition'' as possible in appropriate places.

\item Our properties are not equivalent with those previously studied in
    the literature even for hypergraphs (see Figure~\ref{fig:hypergraph}).
    Hence the reader interested only in this case can safely assume that
    our base theory is $\TkHypergraph$ for some $k\geq 3$, and the objects
    are just hypergraphons. But let us mention that more complicated
    objects like colorings, orderings, couplings, etc.\ will pop up in the
    statements and the proofs anyway.

\item Finally, the description below is loose and sweeps under the rug some
    important technicalities. Proper definitions are deferred to Section~\ref{sec:quasirandomness}.
\end{enumerate}

\begin{description}
\item[\protect{$\Independence[\ell]$}] If we want to realize the
    quasirandom (that is, constant) graphon of density $p$ as a
    2-hypergraphon $\mathcal G\subseteq [0,1]^3$, one way of doing it is by
    \begin{equation}\label{eq:W}
    \mathcal G\df \set{(x_{\{1\}}, x_{\{2\}}, x_{\{1,2\}})}{x_{\{1,2\}}\leq p}.
    \end{equation}
    This 2-hypergraphon has one peculiar property: it does not depend on
    first-order coordinates $x_{\{1\}}, x_{\{2\}}$, and this property is
    perfectly generalizable. Namely, we call a combinatorial object $\phi$
    {\em $\ell$-independent} if it has a representation similar
    to~\eqref{eq:W} that does not depend on the coordinates $x_A$ with
    $\lvert A\rvert\leq \ell$. This is the strongest in the hierarchy of
    our properties, and it relatively easily implies all the others, with
    the same value of the parameter $\ell$. Let us also remark that if the
    object is given in an implicit form, say as a positive homomorphism
    $\phi\in \HomT T$ from the flag algebra, then $\Independence[\ell]$
    only talks about the {\em existence} of the required geometric
    realization or, equivalently, about the {\em possibility} of
    straightening up any geometric realization\footnote{In the case of a
    $T$-on $\mathcal N$, we require that this transformation be uniform
    over all $P$-ons $\mathcal N_P$ forming $\mathcal N$.} using specific
    families of measure-preserving functions~\cite[Ch.~7.3]{Lov12},
    \cite[Sct.~4.1]{ES12}, \cite[Sct.~3]{CR19}. As an example of a
    non-straight representation, the 2-hypergraphon
    \begin{align}\label{eq:Wskew}
      \mathcal G' &\df \set{(x_{\{1\}}, x_{\{2\}}, x_{\{1,2\}})}{(x_{\{1\}} + x_{\{2\}} + x_{\{1,2\}})\bmod 1\leq p}
    \end{align}
    represents the same limit as the one in~\eqref{eq:W} but does depend on first-order coordinates.

\item[\protect{$\UCouple[\ell]$} (Unique $\ell$-coupleability)]
    Roughly speaking (the exact definition in the language of open
    interpretations will be given in Section~\ref{sec:quasirandomness}),
    two combinatorial objects $\phi$ and $\psi$ are uniquely coupleable if
    any two alignments of these objects on the same ground set (a coupling)
    give the same object in the combined theory. In that case, this unique
    coupling can be described by the random alignment, called
    \emph{independent coupling}, and this allows us to compute the combined
    object (represented as a flag-algebraic homomorphism) by a very simple
    formula. For example, quasirandom graphs of density $p\in[0,1]$ are
    uniquely coupleable with any 2-coloring of the vertices as well as with
    the linear ordering. They are {\em not} uniquely coupleable with
    themselves, except for the trivial case $p\in \{0,1\}$. Now, to every
    combinatorial object $\phi$ we associate its rank dually to the notion
    of \Independence: $\rk(\phi)\leq\ell$ if and only if $\phi$ has a
    representation as a $T$-on $\mathcal N$ in which all $P$-ons $\mathcal
    N_P$ depend {\em only} on the coordinates $x_A$ with $\lvert
    A\rvert\leq\ell$. We call an object $\phi$ \emph{uniquely
    $\ell$-coupleable} if it is uniquely coupleable with all objects $\psi$
    such that $\rk(\psi)\leq\ell$.

\item[\protect{$\UInduce[\ell]$} (Unique $\ell$-inducibility)] One
    equivalent way to view the induced subgraph $G\rest_V$ is this: we
    first color the vertices into two colors, say, green (corresponding to
    $V$) and red. Then instead of removing red vertices, we remove all
    edges adjacent to at least one red vertex. In this form, it has a
    perfect generalization in higher dimensions. Namely, we start as in the
    previous paragraph and consider couplings of a combinatorial object
    $\phi$ with an $\ell$-hypergraphon $\psi$ (note that $\rk(\psi)\leq
    \ell$). The unique coupleability requires that for any two such
    couplings $\xi_1$ and $\xi_2$, we have $\xi_1(M)=\xi_2(M)$ for any
    model $M$ of the combined theory. Unique inducibility by $\psi$ relaxes
    this property by requiring that $\xi_1(M)=\xi_2(M)$ holds only for
    those $M$ that are based on a clique in the hypergraphon $\psi$. The
    object $\phi$ is \emph{uniquely $\ell$-inducible} if it is uniquely
    inducible by any $\ell$-hypergraphon $\psi$.
\end{description}

From the loose formulation of the properties above, one can already see that
the first two ``naturality'' requirements are satisfied: the formulations are
made for arbitrary theories and do not refer to densities of concrete models
and their explicit values. As for the third ``naturality'' requirement (Theorem~\ref{thm:naturality}), while
the fact that $\Independence[\ell]$ is preserved under open interpretations
follows easily from the general theory, for $\UCouple[\ell]$ and
$\UInduce[\ell]$, this will follow from an amalgamation property (Theorem~\ref{thm:amalgamation}) that roughly says that couplings can be lifted
through open interpretations (Proposition~\ref{prop:lifting}).

As we mentioned before, the quasirandom $k$-hypergraph satisfies $\Independence[k-1]$. The situation for asymmetric combinatorial objects is more diverse. For
example, the quasirandom tournament satisfies $\UCouple[1]$ but not
$\Independence[1]$ and this example can be generalized to higher values of
$\ell$. One interesting example for unique inducibility is the linear order
as it satisfies $\UInduce[\ell]$ for every $\ell$ without being a trivial
object.

\bigskip
All our properties are anti-monotone in $\ell$ in the sense that for any of
the above, we have the implications $P[\ell]\implies P[\ell-1]$ (see
Theorem~\ref{thm:anti-monotone}) and as for relations between the properties
(Theorem~\ref{thm:inter-properties}), we show that $\Independence[\ell]$
implies $\UCouple[\ell]$ and that $\UCouple[\ell]$ implies\footnote{This
implication is obvious from the definition.} $\UInduce[\ell]$ (see
Figure~\ref{fig:implications}).

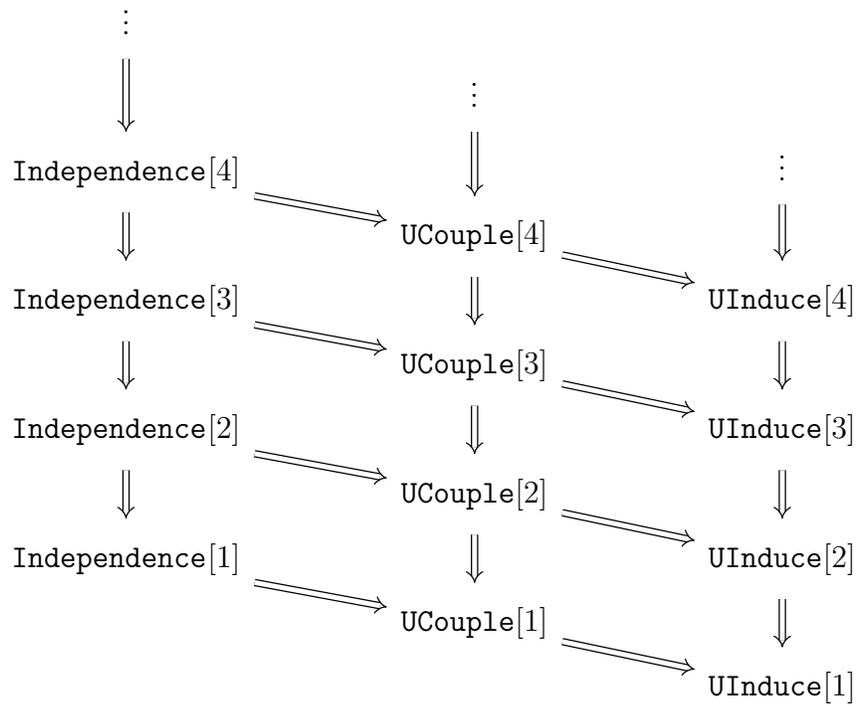
\begin{figure}[htbp]
  \begin{center}
    \begingroup
\newcommand{\indepsample}{\Independence[n]}
\newcommand{\coupsample}{\UCouple[n]}
\newcommand{\indsample}{\UInduce[n]}
\begin{tikzcd}[column sep={1.75cm}, row sep={-0.2cm}, arrows={Rightarrow}, cells={inner ysep={1.75ex}}]
  \makebox[\widthof{\indepsample}][c]{\vdots}
  \arrow[dd] & &
  \\
  &
  \makebox[\widthof{\coupsample}][c]{\vdots}
  \arrow[dd]
  &
  \\
  \Independence[4]
  \arrow[dd]\arrow[dr]
  & &
  \makebox[\widthof{\indsample}][c]{\vdots}
  \arrow[dd]
  \\
  &
  \UCouple[4]
  \arrow[dd]\arrow[dr]
  &
  \\
  \Independence[3]
  \arrow[dd]\arrow[dr]
  & &
  \UInduce[4]
  \arrow[dd]
  \\
  &
  \UCouple[3]
  \arrow[dd]\arrow[dr]
  &
  \\
  \Independence[2]
  \arrow[dd]\arrow[dr]
  & &
  \UInduce[3]
  \arrow[dd]
  \\
  &
  \UCouple[2]
  \arrow[dd]\arrow[dr]
  &
  \\
  \Independence[1]
  \arrow[dr]
  & &
  \UInduce[2]
  \arrow[dd]
  \\
  &
  \UCouple[1]
  \arrow[dr]
  &
  \\
  & &
  \UInduce[1]
\end{tikzcd}
\endgroup    

    \caption{Implications between quasirandomness properties. This is almost a Hasse diagram: only separations
      of the form $\UCouple[\ell]\nRightarrow\Independence[\ell']$ for $\ell' < \ell$ are left open.}
    \label{fig:implications}
  \end{center}
\end{figure}

In terms of separations, we show that no upward implication holds, that is,
none of the studied quasirandomness properties with parameter $\ell$ can
imply the same, or for that matter any other, property with parameter
$\ell+1$ (Theorem~\ref{thm:separationupward}). As for separations between
different families of properties, we show that $\UCouple[\ell]$ does not
imply $\Independence[\ell]$ (Theorem~\ref{thm:separationUCoupleIndependence})
and $\UInduce[\ell]$ does not imply even $\UCouple[1]$
(Theorem~\ref{thm:separationUInduceUCouple}). We have not been able to extend
the latter result to $\UCouple$ vs.\ $\Independence$, that is to
show that $\UCouple[\ell]$ does not imply $\Independence[\ell']$ for a single
pair $\ell'<\ell$; in fact these are the only separations involving the three
families of properties that we leave open. All these separations are
relatively easy when we are working with arbitrary theories, but we show that
they still hold even if we restrict ourselves to the theory of
$k$-hypergraphs, for $k\geq \ell+2$
(Theorems~\ref{thm:separationUCoupleIndependencehypergraph}
and~\ref{thm:separationUInduceUCouplehypergraph}).

Next, we provide the following alternate characterizations (summarized in
Theorems~\ref{thm:UCouple} and~\ref{thm:UInduce}) of these classes.

\begin{description}
\item[Weak $\ell$-independence] Every combinatorial object $\phi$ can be
    represented, in a canonical way, by an infinite countable random model
    $\rn{K}$ defined from a collection of independent random variables
    $(\rn{\theta}_A)_A$ indexed by finite non-empty subsets of $\NN_+$ (see
    e.g.~\cite[proof of Theorem~3.4]{CR19}). We say that $\phi$ is
    \emph{weakly $\ell$-independent} if $\rn{K}$ is independent from
    $(\rn{\theta_A} \mid \lvert A\rvert\leq\ell)$ {\bf as a random
    variable} (full $\Independence[\ell]$ requires this to happen
    ``pointwise''). This weak version of independence turns out to be
    equivalent to $\UCouple[\ell]$
    (Theorem~\ref{thm:UCouple}\ref{thm:UCouple:weakindep}).

\item[$\ell$-Locality] One of the defining properties of the countable
    random model $\rn{K}$ is {\em locality}: the marginals
    $(\rn{K}\rest_{V_i} \mid i\in I)$ are mutually independent whenever the
    collection of finite sets $(V_i)_{i\in I}$ is pairwise disjoint. The
    notion of \emph{$\ell$-locality} strengthens this property to require
    mutual independence of $(\rn{K}\rest_{V_i} \mid i\in I)$
    whenever the collection of finite sets $(V_i)_{i\in I}$ have pairwise
    intersections of size at most $\ell$. It is clear that weak
    $\ell$-independence implies $\ell$-locality, but we prove that the
    converse also holds, hence $\ell$-locality is also equivalent to $\UCouple[\ell]$ (Theorem~\ref{thm:UCouple}\ref{thm:UCouple:local}).

\item[Symmetric $\ell$-locality] The notion of \emph{symmetric $\ell$-locality} relaxes the notion of $\ell$-locality by requiring only
    mutual independence of the events $(\rn{K}\rest_{V_i}\cong M_i \mid i\in I)$
    for all choices of $(V_i)_{i\in I}$ with pairwise intersections of size
    at most $\ell$ and all choices of models $M_i$, i.e., we only care
    about the submodels $\rn{K}\rest_{V_i}$ up to isomorphism. We show that
    symmetric $\ell$-locality is equivalent to $\UInduce[\ell]$
    (Theorem~\ref{thm:UInduce}\ref{thm:UInduce:symlocal}).
\end{description}

The right way to view the definitions of unique coupleability and unique
inducibility is that each $\psi$ of rank $\leq\ell$ generates a test for the
respective property that $\phi$ has to pass. It is natural to ask for a
smaller and more explicit set of universal tests that guarantees each
property. We show (Theorem~\ref{thm:UCouple}\ref{thm:UCouple:QRhyper}) that
$\phi\in\UCouple[\ell]$ is equivalent to $\phi$ being uniquely coupleable
with a non-degenerate {\bf quasirandom} $\ell'$-hypergraphon $\psi_{\ell',p}$
in every dimension $\ell'\leq\ell$. We further prove
(Theorem~\ref{thm:UCouple}\ref{thm:UCouple:QRhyperindepcoup}) that it is also
equivalent to $\phi$ being uniquely coupleable with their independent
coupling $\psi_{1,p_1}\otimes\ldots\otimes \psi_{\ell,p_\ell}$; for the
reasons explained right after the statement of the theorem, it does not
immediately follow from the previous item~\ref{thm:UCouple:QRhyper}.
In the particular case $\ell=1$, this means that the fact that $\phi$ is uniquely coupleable with a single non-trivial vertex-coloring implies it must also be uniquely coupleable with any rank $1$ limit object, such as linear orders, permutations, etc.

Our findings for unique inducibility are by far less conclusive but at
least we can show that it is sufficient to consider only hypergraphons $\psi$
with any fixed non-trivial edge density $p\in(0,1)$ (Theorem~\ref{thm:UInduce}\ref{thm:UInduce:singledensity}).

\medskip
Of all choices of parameters, arguably the most interesting one is when
$\ell$ is exactly one less than the maximum arity $k$ of a predicate of the
language. In the theory of $k$-hypergraphs the three classes with $\ell =
k-1$ become the same and are satisfied only by the full quasirandom
hypergraph, that is, the almost sure limit of the generalization of the
Erd\H{o}s--R\'{e}nyi model. If we consider general theories of arity at most
$k$, it is not hard to see (Theorem~\ref{thm:fullIndependence}) that
$(k-1)$-independent objects are (essentially) quasirandom colored
$k$-hypergraphs. The property $\UCouple[k-1]$ in arity at most $k$
corresponds to independent couplings of quasirandom colored $k$-hypergraphs with
generalizations of quasirandom tournaments (Theorem~\ref{thm:fullUCouple}).
The case of unique inducibility is (again) considerably more complicated:
$\UInduce[1]$ in arity $2$ corresponds to (essentially) independent couplings
of quasirandom colored graphs with an aligned coupling of several biased
quasirandom tournaments; this latter aligned coupling is so that all biases
are in the same direction. But since this latter proof is very technical and
does not seem to easily generalize to arbitrary arities $k$, we do not
include it in the paper.

Finally, let us compare our properties to the known hypergraph
quasirandomness properties (Figure~\ref{fig:hypergraph}). In~\cite{Tow17},
Towsner defined $k$-hypergraph quasirandomness properties $\Disc[\cA]$ for
every antichain $\cA$ of non-empty subsets of $[k]\df\{1,\ldots,k\}$ and
showed that $\Disc[\binom{[k]}{\ell}]$ and $\Disc[\cA_\ell]$ are equivalent
to $\CliqueDisc[\ell]$ and $\Dev[\ell]$ of~\cite{LM15b}, respectively,
where $\cA_\ell\df \{A\in\binom{[k]}{k-1} \mid [k-\ell]\subseteq A\}$. It is
immediate from definitions that $\UInduce[\ell]$ implies $\CliqueDisc[\ell]$
(Theorem~\ref{thm:UInduce->CliqueDisc}). In terms of separations between our
properties and the ones from the literature, we show the strongest separation
possible. The strongest $\Disc[\cA]$ property that is not equivalent to full
quasirandomness is $\Dev[k-1]$ and this does not imply even $\UInduce[1]$
(Theorem~\ref{thm:separationDevUInduce}). In the other direction, the weakest
$\Disc[\cA]$ property that is not implied by $\CliqueDisc[\ell]$ is
$\Disc[\{[\ell+1]\}]$ and this is not implied by $\Independence[\ell]$
(Theorem~\ref{thm:separationIndependenceDisc}).

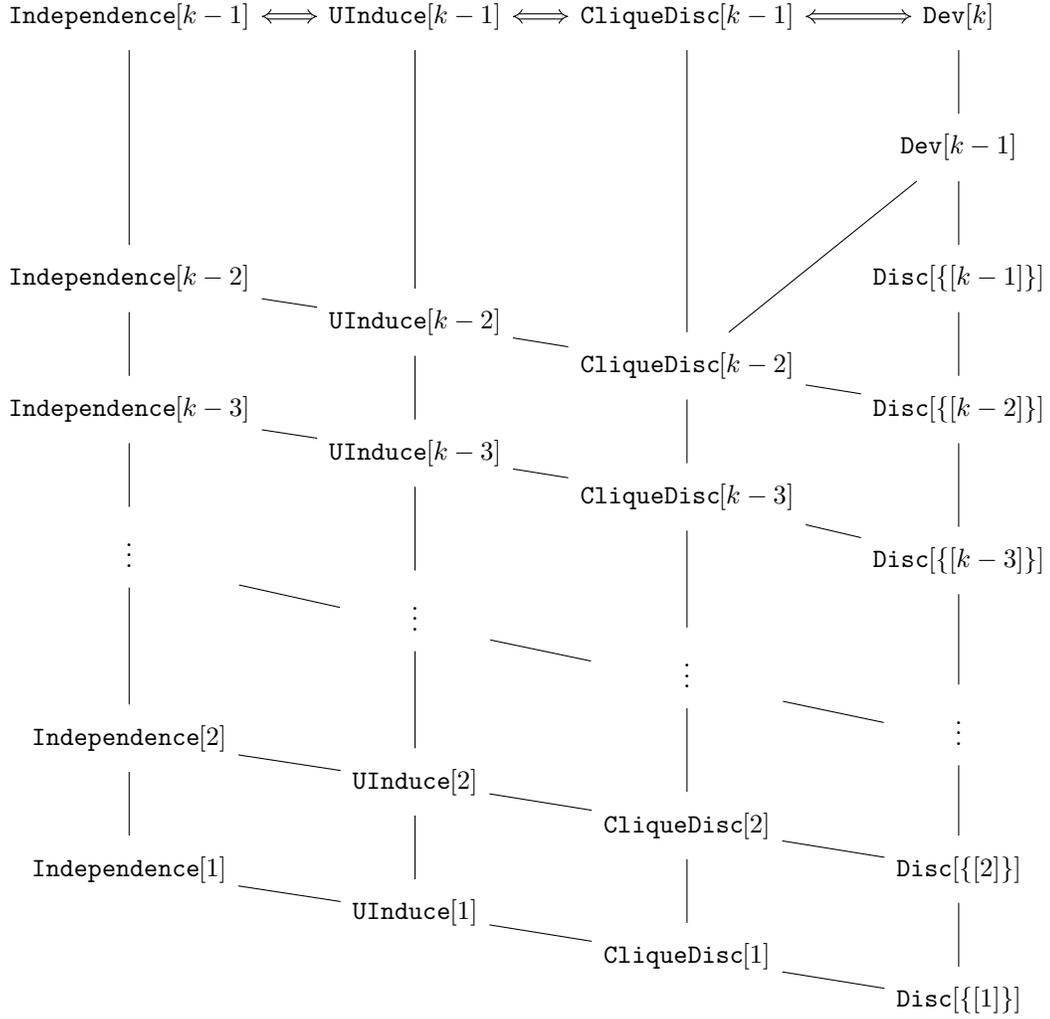
\begin{figure}[htbp]
  \begin{center}
    \begingroup
\newcommand{\indepsample}{\Independence[n]}
\newcommand{\indsample}{\UInduce[n]}
\newcommand{\CDsample}{\CliqueDisc[n]}
\newcommand{\Discsample}{\Disc[\{[n]\}]}
\begin{footnotesize}
  \begin{tikzcd}[column sep={0.75cm}, row sep={-0.3cm}, arrows={dash}, cells={inner ysep={1.75ex}}]
    \Independence[k-1]
    \arrow[r, Leftrightarrow]\arrow[dddddd]
    &
    \UInduce[k-1]
    \arrow[r, Leftrightarrow]
    \arrow[ddddddd]
    &
    \CliqueDisc[k-1]
    \arrow[r, Leftrightarrow]
    \arrow[dddddddd]
    &
    \Dev[k]
    \arrow[ddd]
    \\
    \vphantom{\indepsample} & & &
    \\
    \vphantom{\indepsample} & & &
    \\
    & & &
    \Dev[k-1]
    \arrow[ddd]\arrow[dddddl]
    \\
    \vphantom{\indepsample} & & &
    \\
    \vphantom{\indepsample} & & &
    \\
    \Independence[k-2]
    \arrow[dr]\arrow[ddd]
    & & &
    \Disc[\{[k-1]\}]
    \arrow[ddd]
    \\
    &
    \UInduce[k-2]
    \arrow[dr]\arrow[ddd]
    & &
    \\
    & &
    \CliqueDisc[k-2]
    \arrow[dr]\arrow[ddd]
    &
    \\
    \Independence[k-3]
    \arrow[dr]\arrow[ddd]
    & & &
    \Disc[\{[k-2]\}]
    \arrow[ddd]
    \\
    &
    \UInduce[k-3]
    \arrow[dr]\arrow[ddd]
    & &
    \\
    & &
    \CliqueDisc[k-3]
    \arrow[dr]\arrow[ddd]
    &
    \\
    \makebox[\widthof{\indepsample}][c]{\vdots}
    \arrow[dr]\arrow[ddd]
    & & &
    \Disc[\{[k-3]\}]
    \arrow[ddd]
    \\
    &
    \makebox[\widthof{\indsample}][c]{\vdots}
    \arrow[dr]\arrow[ddd]
    & &
    \\
    & &
    \makebox[\widthof{\CDsample}][c]{\vdots}
    \arrow[dr]\arrow[ddd]
    &
    \\
    \Independence[2]
    \arrow[dr]\arrow[ddd]
    & & &
    \makebox[\widthof{\Discsample}][c]{\vdots}
    \arrow[ddd]
    \\
    &
    \UInduce[2]
    \arrow[dr]\arrow[ddd]
    & &
    \\
    & &
    \CliqueDisc[2]
    \arrow[dr]\arrow[ddd]
    &
    \\
    \Independence[1]
    \arrow[dr]
    & & &
    \Disc[\{[2]\}]
    \arrow[ddd]
    \\
    &
    \UInduce[1]
    \arrow[dr]
    & &
    \\
    & &
    \CliqueDisc[1]
    \arrow[dr]
    &
    \\
    & & &
    \Disc[\{[1]\}]
  \end{tikzcd}
\end{footnotesize}
\endgroup    

    \caption{Hasse diagram of quasirandomness hypergraph properties in arity $k$.
    The top four equivalent properties represent full quasirandomness.}
    \label{fig:hypergraph}
  \end{center}
\end{figure}

\bigskip

The paper is organized as follows. In Section~\ref{sec:prelim} we give necessary preliminaries. In
Section~\ref{sec:mainresults} we formally state our main results. In Section~\ref{sec:basic}, we prove
some basic facts that will be used throughout the text. In
Section~\ref{sec:naturality}, we show that our properties are natural, that is, that they are preserved under
open interpretations. In Section~\ref{sec:UInduce} we prove the alternative formulations of $\UInduce$, and in
Section~\ref{sec:UCouple} we prove the alternative formulations of $\UCouple$. The proofs are done in this
slightly reversed order because they are simpler for the unique inducibility; besides, some auxiliary
statements we need for that part are later re-used for the unique coupleability. In
Section~\ref{sec:separations}, we show separations between different classes of properties. In
Section~\ref{sec:fullQR}, we completely classify the properties $\Independence[k-1]$ and $\UCouple[k-1]$ when
all arities are at most $k$. The paper is concluded with a few remarks and open problems in
Section~\ref{sec:concl}.

\section{Preliminaries and notation}
\label{sec:prelim}

Throughout the text, we will use the notation $\NN\df\{0,1,\ldots\}$ for the
non-negative integers and $\NN_+\df\NN\setminus\{0\}$ for the positive
integers. We also let $[n]\df\{1,\ldots,n\}$ and $(n)_m\df
n(n-1)\cdots(n-m+1)$. The usage of the arrow $\rightarrowtail$ for a function
will always presume the function to be injective. For a set $V$, we let
$(V)_\ell$ be the set of all injective functions
$\alpha\injection{[\ell]}{V}$ and for such an $\alpha$, we may use the
notation $\alpha_i$ for $\alpha(i)$ when convenient. We let
$2^V\df\{A\subseteq V\}$ be the set of all the subsets of $V$, let
$\binom{V}{\ell}\df\{A\subseteq V\mid\lvert A\rvert=\ell\}$ and let
$\binom{V}{> \ell}\df\{A\subseteq V\mid\lvert A\rvert > \ell\}$. For $V\subseteq\NN_+$ and $A\in
\binom{V}{\ell}$, we let $\iota_A\injection{[\ell]}{V}$ be the function
  enumerating the set $A$ in the increasing order (so $\im(\iota_A)=A$).
We let
$r(V)$ be the set of all \emph{finite non-empty} subsets of $V$ and
$r(V,\ell)\df\{A\in r(V)\mid\lvert A\rvert\leq \ell\}$ be the set of all
non-empty subsets of $V$ of size at most $\ell$. We will be frequently
abusing notation by identifying $[n]$ with $n$, e.g., we will use $r(n,\ell)$ as a shorthand for
$r([n],\ell)$. Random variables will always be typed in $\rn{math\ bold\
face}$. We denote by $S_V$ the group of bijections $V\rightarrowtail V$ so
that $S_n$ is the group of permutations on $n$ elements.

\subsection{Model theory and limit theory}
\label{sec:notation}

We will be working in the framework of~\cite{CR19}, in which combinatorial objects are encoded as models of a
canonical theory. We will also be using the same notation as in~\cite{CR19} with some small additions.

For a finite relational language $\cL$, we let $T_\cL$ be the pure canonical theory on $\cL$, that is, the
theory whose axioms are
\begin{align}\label{eq:canonical}
  \forall\vec{x}, \left(\bigvee_{1\leq i < j\leq k(P)} x_i = x_j\right) \to \neg P(x_1,\ldots,x_{k(P)}).
\end{align}
for every $P\in\cL$. For a canonical theory $T$ and a set $V$, let $\cK_V[T]$ be the set of (labeled) models of $T$ with
vertex set $V$. We use the symbol $\cong$ to indicate that two models are isomorphic. For $n\in \mathbb N_+$, we let
$\mathcal M_n[T]\df\mathcal K_n[T]/\cong$ be the set of $n$-element (unlabeled) models up to isomorphism;
we also let $\mathcal M[T]\df\bigcup_{n\in\mathbb N_+} \mathcal M_n[T]$. For $n=\lvert V\rvert$
and $K\in\mathcal K_V[T]$, we denote by $[K]\in \mathcal M_n[T]$ the isomorphism
type of $K$.

Other important examples of canonical theories include the
theory of $k$-hypergraphs $\TkHypergraph$, whose language contains a single
predicate $E$ of arity $k(E)\df k$ and whose axioms are~\eqref{eq:canonical} 
for $P=E$ and
\begin{align}\label{eq:hypergraphsymmetry}
  \forall\vec{x}, (E(x_1,\ldots,x_k)\to E(x_{\sigma(1)},\ldots,x_{\sigma(k)})) & & (\sigma\in S_k);
\end{align}
the theory of (simple) graphs $\TGraph\df\TkHypergraph[2]$; the theory of
(strict) linear orders $\TLinOrder$, whose language contains a single binary
predicate $\prec$ with the axioms
\begin{gather*}
  \forall x, \neg(x \prec x);\\
  \forall\vec{x}, (x_1\neq x_2\to (x_1\prec x_2 \lor x_2\prec x_1));\\
  \forall\vec{x}, (x_1\prec x_2\land x_2\prec x_3 \to x_1\prec x_3);
\end{gather*}
and the theory of $c$-colorings $\TcColoring$, whose language contains $c$ unary predicates $\chi_1,\ldots,\chi_c$
and that has axioms
\begin{gather*}
  \forall x, \neg\chi_i(x)\lor\neg\chi_j(x) \qquad (1\leq i < j\leq c);\\
  \forall x, \bigvee_{i\in[c]}\chi_i(x).
\end{gather*}
Note that $\TcColoring[2]$ and $\TkHypergraph[1]$ are isomorphic in the category \cat{Int} (see~\cite[Sct.~2.2]{CR19}).

Given an atomless complete\footnote{\label{foot:complete}
In~\cite[Sct.~7]{CR19} we carefully considered incomplete spaces as well and
drew finer distinctions between various assumptions on them, cf.~the
discussion in~\cite[page~218]{Lov12}. It was needed to differentiate between
weak theons (satisfying the axioms a.e.) and strong ones (satisfying them
everywhere off-diagonal), as well as for removal lemmas. As we prefer to
avoid dwelling into these issues in this paper, we make the simplifying
assumption of completeness once and for all.} probability space $\Omega =
(X,\cA,\mu)$, a set $V$ and $\ell\in\NN$, we let $\cE_{V,\ell}(\Omega)\df
X^{r(V,\ell)}$, equipping it with the completion of the product measure of
$\lvert r(V,\ell)\rvert$ copies of $\mu$, which by abuse of notation we also denote by
$\mu$ (cf.~\cite[Definition~7.3]{CR19}). Likewise,
$\cE_V(\Omega)\df X^{r(V)}$.  Given an injective function
$\alpha\injection{V_1}{V_2}$, we define the projection
$\alpha^*\function{\cE_{V_2,\ell}(\Omega)}{\cE_{V_1,\ell}(\Omega)}$ by
$\alpha^*(x)_A\df x_{\alpha(A)}$ (this is consistent with the notation for the
projection $\alpha^*\function{\cE_{V_2}(\Omega)}{\cE_{V_1}(\Omega)}$ defined
analogously in~\cite[Definition~2.19]{CR19}, which we will also use). The spaces
used in this paper most often are $([0,1]^t,\cL^t,\lambda^t)$, where $\cL^t$ is
the $\sigma$-algebra of Lebesgue measurable subsets of $[0,1]^t$ and $\lambda^t$
is the ($t$-dimensional) Lebesgue measure; these will be denoted simply by
$[0,1]^t$. When $\Omega=[0,1]$, we will omit $\Omega$ from the notation
(e.g., a $P$-on without reference to any space $\Omega$ is assumed to be a
measurable subset of $\cE_{k(P)}\df\cE_{k(P)}([0,1])$). For spaces
$\Omega$ and $\Omega'$, we let $\Omega\times\Omega'$ be the \emph{completion}
of the product space. Finally, we will often abuse
notation by identifying the spaces $\cE_V(\Omega\times\Omega')$ and
$\cE_V(\Omega)\times\cE_V(\Omega')$ via the correspondence
$\cE_V(\Omega\times\Omega')\ni x\leftrightarrow
(y,z)\in\cE_V(\Omega)\times\cE_V(\Omega')$ given by $y_A \df (x_A)_1$ and
$z_A\df (x_A)_2$ for every $A\in r(V)$. An analogous identification will be
done for products of finitely many spaces.

We also adopt the same conventions as in~\cite{CR19}: unless we
explicitly say otherwise, all our languages are assumed to be finite first-order relational languages, all our
theories are assumed to be canonical (in particular, also universal and we will typically omit universal
quantifiers from their axioms) and all our structures are assumed to be canonical (i.e., models of $T_\cL$, or
equivalently, structures $K$ such that $R_P(K)\df\{\alpha\in V(K)^{k(P)}\mid K\vDash
P(\alpha_1,\ldots,\alpha_{k(P)})\}$ is contained in $(V(K))_{k(P)}$ for every $P\in\cL$).

Recall that a sequence of finite (unlabeled) models $(N_n)_{n\in\NN}$ is
\emph{convergent} if $\lvert N_n\rvert < \lvert N_{n+1}\rvert$ and for every
fixed finite model $M$, the limit $\lim_{n\to\infty} p(M,N_n)$ exists, where
$p(M,N)$ denotes the normalized number of unlabeled induced copies of $M$ in $N$.  We
will be using three cryptomorphic ways of representing convergent
sequences: flag-algebraic homomorphisms~\cite{Raz07}, theons~\cite[Scts.~3
and~7]{CR19} and exchangeable arrays~\cite[Definition~5.7]{CR19}. In this
language, a hypergraphon of~\cite{ES12} is, up to zero-measure change, a
$\TkHypergraph$-on and there is a (not one-to-one) correspondence between
graphons $W$ of~\cite{LS06} and $\TGraph$-ons $\cN$ that preserves densities
given by
\begin{align*}
  W & \mapsto \{x\in\cE_2 \mid x_{\{1,2\}} < W(x_{\{1\}},x_{\{2\}})\}\\
  W_\cN & \mapsfrom \cN,
\end{align*}
where
\begin{align}\label{eq:theonsVSgraphons}
  W_\cN(x_{\{1\}},x_{\{2\}}) & \df \lambda(\{x_{\{1,2\}} \mid (x_{\{1\}},x_{\{2\}},x_{\{1,2\}})\in\cN\}).
\end{align}

Furthermore, for $M\in\cM[T]$ we let
\begin{align}\label{eq:labeledflagalgebra}
  \langle M\rangle &\df \frac{\lvert\Aut(M)\rvert}{\lvert M\rvert!} M
\end{align}
denote\footnote{Note that if we think of $M$ as a flag algebra type, then this notation is compatible
  with~\cite[Definition~8]{Raz07}. But in this paper, like in~\cite{CR19}, we try to avoid
  flag algebras in non-trivial types.} the element of the flag algebra $\cA[T]$ encoding the \emph{labeled} (induced) density of
$M$.

The main theorem of dense limit theory says that positive homomorphisms, theons and local exchangeable arrays all encode convergent sequences.
\begin{theorem}[\protect{\cite{LS06,Raz07}}, \protect{\cite[Theorem~6.3]{CR19}} see also~\protect{\cite[Sct.~7]{CR19}}]\label{thm:cryptomorphism}
  Fix an atomless complete probability space $\Omega$ and consider the following objects for a theory $T$.
  \begin{enumerate}
  \item A convergent sequence $(N_n)_{n\in\NN}$ of models of $T$.
  \item A positive homomorphism $\phi\in\HomT{T}$.
  \item A $T$-on $\cN$ over $\Omega$.
  \item A local exchangeable array $\rn{K}$ supported on models of $T$.
  \end{enumerate}

  The objects above are cryptomorphic in the sense that given an instance of one of them, one can
  ``explicitly'' construct instances of the others that satisfy the following for every $M\in\cM[T]$:
  \begin{align*}
    \lim_{n\to\infty} p(M,N_n) & = \phi(M) = \phi_\cN(M) = \PP[\rn{K}\rest_{[\lvert M\rvert]}\cong M].
  \end{align*}
\end{theorem}

One of the (easy) directions of the cryptomorphism above will be of
particular importance to us, namely, how to construct a local exchangeable
array $\rn{K}$ from a given $T$-on $\cN$ over $\Omega=(X,\cA,\mu)$.
Intuitively, the only thing we have to do is to independently sample
countably many points from our theon. Formally, let $\rn{\theta} =
(\rn{\theta}_A)_{A\in r(\NN_+)}$ be picked in $\cE_{\NN_+}(\Omega)$ according to
[the product measure] $\mu$, that is, each $\rn{\theta}_A$ picked in $X$
according to $\mu$ independently of all other
coordinates. The \emph{exchangeable array $\rn{K}$ corresponding to $\cN$
with respect to $\rn{\theta}$} is defined by
\begin{align}\label{eq:arraycryptomorphism}
  V(\rn{K}) & \df \NN_+, &
  R_P(\rn{K}) & \df \{\alpha\in(\NN_+)_{k(P)} \mid \alpha^*(\rn{\theta}) \in \cN_P\}
\end{align}
and we have $\phi_\cN(M) = \PP[\rn{K}\rest_{[\lvert M\rvert]}\cong M]$ for every $M\in\cM[T]$ (see~\cite[proof of Theorem~3.4]{CR19}).

Once we capture combinatorial objects as models of canonical theories, local combinatorial constructions are
then captured by open interpretations (see~\cite[Sct.~2.2]{CR19}) in the sense that if $I\interpret{T_1}{T_2}$
is an open interpretation and $K$ is a model of $T_2$, there is a naturally defined model $I(K)$ of $T_1$
given by $V(I(K))\df V(K)$ and $R_P(I(K))\df\{\alpha\in(V(K))_{k(P)}\mid K\vDash
I(P)(\alpha_1,\ldots,\alpha_{k(P)})\}$. The simplest but most important type of open interpretations are the
\emph{structure-erasing} interpretations, which are open interpretations of the form $I\interpret{T_1}{T_1\cup
  T_2}$, where $T_1\cup T_2$ is the disjoint union of the theories $T_1$ and $T_2$. They act identically on the
  language of $T_1$, and the corresponding combinatorial
construction corresponds to erasing all information of $T_2$. Convergent sequences behave very well
with respect to open interpretations, namely, if $(N_n)_{n\in\NN}$ is a convergent sequence of models of
$T_2$, then $(I(N_n))_{n\in\NN}$ is a convergent sequence of models of $T_1$. This behavior is translated to
operations on the limit objects of Theorem~\ref{thm:cryptomorphism}. Namely, if $\phi\in\HomT{T_2}$, the
$T_2$-on $\cN$ and the array $\rn{K}$ correspond to a convergent sequence $(N_n)_{n\in\NN}$ of models of $T_2$
under Theorem~\ref{thm:cryptomorphism}, then $\phi^I\df\phi\comp\pi^I$ (where $\pi^I$ is $\pi^{(U,I)}$
in~\cite[Definition~4 and Theorem~2.6]{Raz07} when $U(x)$ is $x=x$; see also~\cite[Theorem~2.14]{CR19}), $I(\cN)$ given by $I(\cN)_P\df T(I(P),\cN)$
(see~\cite[Definition~3.5]{CR19}) and $I(\rn{K})$ are limit objects corresponding to $(I(N_n))_{n\in\NN}$ for
Theorem~\ref{thm:cryptomorphism} (see~\cite[Remark~6]{CR19}).

Finally, let us denote the identity interpretation of a theory $T$ by
$\id_T\interpret{T}{T}$ and for interpretations $I\interpret{T_1}{T_3}$ and
$J\interpret{T_2}{T_4}$, we denote by $I\cup J\interpret{T_1\cup T_2}{T_3\cup
T_4}$ the amalgamation interpretation that acts as $I$ on $T_1$ and acts as
$J$ on $T_2$.

\subsection{Quasirandomness properties}
\label{sec:quasirandomness}

In this section we formalize all notions of quasirandomness presented in the introduction.

\begin{definition}[rank and independence]
  The \emph{rank} of a peon $\cN\subseteq\cE_k(\Omega)$ over $\Omega=(X,\cA,\mu)$, denoted $\rk(\cN)$, is the minimum $r\in\NN$ such that $\cN$ can be written as $\cN = \cH\times X^{\binom{[k]}{> r}}$ for some $\cH\subseteq\cE_{k,r}(\Omega)$. The \emph{rank} of an Euclidean structure $\cN$ is the maximum rank $\rk(\cN)$ of its peons.

  Dually, for $\ell\in\NN$, a peon $\cN\subseteq\cE_k(\Omega)$ is called \emph{$\ell$-independent} if
  it can be written as $\cN = \cE_{k,\ell}(\Omega)\times\cH$ for some $\cH\subseteq X^{\binom{[k]}{> \ell}}$ and an Euclidean structure is called \emph{$\ell$-independent} if all of its peons are
  $\ell$-independent.

  For $\ell\in\NN$, an Euclidean structure $\cN$ on $\cL$ over $\Omega$ is \emph{weakly $\ell$-independent} if the exchangeable array $\rn{K}$ corresponding to $\cN$ with respect to $\rn{\theta}$ picked in $\cE_{\NN_+}(\Omega)$ according to $\mu$ (see~\eqref{eq:arraycryptomorphism}) is independent from $(\rn{\theta}_A \mid A\in r(\NN_+,\ell))$ as a random variable.

  Given $\phi\in\HomT{T}$, the \emph{rank} of $\phi$, denoted $\rk(\phi)$, is
  the minimum rank of a $T$-on $\cN$ such that $\phi_\cN=\phi$. Dually, we
  say $\phi\in\HomT{T}$ is \emph{$\ell$-independent} (resp., \emph{weakly $\ell$-independent})
  if there exists an $\ell$-independent (resp., weakly $\ell$-independent) $T$-on $\cN$ such that
  $\phi_\cN=\phi$. We will refer to the former property as $\Independence[\ell]$ but we do not introduce any special notation for weak
  independence as it will be shown to be equivalent to another property below.
\end{definition}

\begin{definition}[couplings]
  Given canonical theories $T_1,\ldots,T_t$ and $\phi_i\in\HomT{T_i}$ ($i\in[t]$), a \emph{coupling} of
  $\phi_1,\ldots,\phi_t$ is a positive homomorphism $\xi\in\HomT{\bigcup_{i\in[t]} T_i}$ such that $\xi^{I_i}
  = \phi_i$ for every $i\in[t]$, where $I_i\interpret{T_i}{\bigcup_{j\in[t]} T_j}$ is the structure-erasing
  interpretation.
\end{definition}

The most important coupling is the independent coupling defined below. In the finite world, the independent
coupling of limits of sequences $(N_n^i)_{n\in\NN}$ with $V(N_n^i)=V(N_n^j)$ corresponds to the almost sure
limit of the random sequence $(\rn{N_n})_{n\in\NN}$ where $\rn{N_n}$ is obtained by first randomly permuting
the vertices of each $N_n^i$ uniformly and independently and coupling the result.

\begin{definition}[independent coupling, semantic version]\label{def:indcoup}
  For $i\in[t]$, let $\cN^i$ be a $T_i$-on over $\Omega_i$. The \emph{independent coupling} of $\cN^1,\ldots,\cN^t$ is the $(\bigcup_{i\in[t]}T_i)$-on $\cN^1\otimes\cdots\otimes\cN^t$ over $\prod_{i\in[t]}\Omega_i$ defined by
  \begin{align*}
    (\cN^1\otimes\cdots\otimes\cN^t)_P & \df \set{x\in\prod_{j\in[t]}\cE_{k(P)}(\Omega_j)}{\pi_i(x)\in\cN^i_P},
  \end{align*}
  whenever $P$ is in the language of $T_i$ and where $\pi_i$ denotes the natural projection on the $i$-th coordinate.
  \end{definition}

\begin{definition}[independent coupling, syntactic version]\label{def:indcoup_syntactic}
  For $i\in[t]$, let $\phi_i\in\HomT{T_i}$. The \emph{independent coupling} of $\phi_1,\ldots,\phi_t$ is defined
  by
  \begin{align*}
    (\phi_1\otimes\cdots\otimes\phi_t)(\langle M\rangle) & \df \prod_{i\in[t]} \phi_i(\langle I_i(M)\rangle),
  \end{align*}
  for every $M\in\cM[\bigcup_{i\in[t]}T_i]$, where $I_i\interpret{T_i}{\bigcup_{j\in[t]}T_j}$ is the
  structure-erasing interpretation.
\end{definition}

These two definitions are obviously consistent: if $\cN^i$ is a $T_i$-on over $\Omega_i$ such that
$\phi_{\cN^i} = \phi_i$ ($i\in[t]$), then $(\phi_1\otimes\cdots\otimes\phi_t) =
\phi_{\cN^1\otimes\cdots\otimes\cN^t}$. In particular, this implies that $\phi_1\otimes\cdots\otimes\phi_t\in
\HomT{\bigcup_{i\in[t]} T_i}$ (which can be also verified by a direct computation).

\begin{definition}[unique coupleability and inducibility]\label{def:UCoupleUInduce}
  We say that $\phi_1,\ldots,\phi_t$ are \emph{uniquely coupleable} if the
  independent coupling is their only coupling.
  For $\ell\in\NN$, we say that $\phi\in\HomT{T}$ is \emph{uniquely $\ell$-coupleable} if for every theory
  $T'$ and every $\psi\in\HomT{T'}$ with $\rk(\psi)\leq\ell$, $\phi$ and $\psi$ are uniquely coupleable.
  We will be using the abbreviation $\UCouple[\ell]$ for this property.

  Given $\ell\in\NN_+$, $\phi\in\HomT{T}$ and $\psi\in\HomT{\TkHypergraph[\ell]}$, we say that $\phi$ is
  \emph{uniquely inducible} by $\psi$ if for any coupling $\xi$ of $\phi$ and $\psi$ and for
  every $M\in\cM[T\cup\TkHypergraph[\ell]]$ such that $I(M)$ is a complete $\ell$-hypergraph,
  we have $\xi(M) = (\phi\otimes\psi)(M)$,
  where $I\interpret{\TkHypergraph[\ell]}{T\cup\TkHypergraph[\ell]}$ is the structure-erasing
  interpretation. We say that $\phi$ is \emph{uniquely $\ell$-inducible} if it is uniquely
inducible by every $\psi\in\HomT{\TkHypergraph[\ell]}$, and we will be using
the abbreviation $\UInduce[\ell]$. For completeness, we declare every $\phi$ to satisfy $\UInduce[0]$.
\end{definition}

\begin{remark}\label{rmk:UInduce1}
  Since $\TkHypergraph[1]\cong\TcColoring[2]$, for $\ell=1$ we prefer to work with the following equivalent
  formulation of $\UInduce[1]$ that can be deduced from this isomorphism. $\phi\in\HomT{T}$ is uniquely
  inducible by $\psi\in\HomT{\TcColoring[2]}$ if for any coupling $\xi$ of $\phi$ and $\psi$ and for every
  $M\in\cM[T\cup\TcColoring[2]]$ such that $R_{\chi_1}(M) = V(M)$, we have $\xi(M) =
  (\phi\otimes\psi)(M)$. Then
  $\phi$ is uniquely $1$-inducible if it is uniquely inducible by every $\psi\in\HomT{\TcColoring[2]}$.

  Also, as we will see below (Theorem~\ref{thm:anti-monotone}), $\UInduce[\ell]$
  implies $\UInduce[\ell']$ for any $\ell'\leq\ell$. Hence, we could have
  equivalently required in this definition unique inducibility by every
  $\psi\in\HomT{\TkHypergraph[\ell']}$ with $\ell'\leq\ell$.
\end{remark}

These three properties are central to our paper. If $P$ is any of them, we
will say interchangeably that $\phi$ satisfies $P[\ell]$ or that $\phi\in
P[\ell]$.

\begin{definition}[locality]
  Let $\cN$ be a $T$-on over $\Omega=(X,\cA,\mu)$ and let $\rn{K}$ be the exchangeable array corresponding to $\cN$
  with respect to $\rn{\theta}$ picked in $\cE_{\NN_+}(\Omega)$ according to $\mu$
  (see~\eqref{eq:arraycryptomorphism}).

  We say that $\cN$ is \emph{$\ell$-local} if for every collection $(V_i)_{i\in I}$ of finite subsets of
  $\NN_+$ with pairwise intersections of size at most $\ell$, the marginals $(\rn{K}\rest_{V_i}\mid i\in I)$
  are mutually independent.

  We say that $\cN$ is \emph{symmetrically $\ell$-local} if for every collection $(V_i)_{i\in I}$ of finite
  subsets of $\NN_+$ with pairwise intersections of size at most $\ell$, the random variables $(\left[\rn{K}\rest_{V_i}\right]\mid i\in I)$ (recall that $[K]$ is the isomorphism
  type of $K$) are mutually independent.

  We say that $\phi\in\HomT{T}$ is \emph{$\ell$-local} (resp., \emph{symmetrically $\ell$-local}) if there exists an $\ell$-local (resp., symmetrically $\ell$-local) $T$-on $\cN$ such that $\phi = \phi_\cN$.
\end{definition}

Note that both the notions of $0$-locality and symmetric $0$-locality
coincide with the notion of locality for $\rn{K}$ (see~\cite[Definition~5.12]{CR19}). Besides, it is very easy
to give an explicit purely syntactic description of both locality and symmetric
locality in the style of Definition~\ref{def:indcoup_syntactic}; this in particular implies that for an $\ell$-local (resp., symmetrically $\ell$-local) $\phi$, every $T$-on $\cN$ with $\phi = \phi_\cN$ must necessarily be $\ell$-local (resp., symmetrically $\ell$-local).

Finally, let us state the properties $\CliqueDisc[\ell]$ and $\Disc[\cA]$ in
the limit language.

\begin{definition}\label{def:CliqueDisc}
  Let $K^{(t)}_n\in\cM_n[\TkHypergraph[t]]$ be the complete $t$-uniform hypergraph on $n$ vertices and let $\rho_t\df
  K^{(t)}_t$. Let $\phi\in\HomT{\TkHypergraph}$ and $\ell\in[k]$.

  We say that $\phi$ satisfies $\CliqueDisc[\ell]$ (\cite{LM15b}) if for every
  $\psi\in\HomT{\TkHypergraph[\ell]}$ and every coupling $\xi$ of $\phi$ and $\psi$, we have
  \begin{align*}
    \xi(K^{(k,\ell)}_k) & = \phi(\rho_k)\psi(K^{(\ell)}_k),
  \end{align*}
  where $K^{(k,\ell)}_k\in\cM_k[\TkHypergraph\cup\TkHypergraph[\ell]]$ is the model obtained by aligning
  $\rho_k$ and $K^{(\ell)}_k$ (i.e., the model of size $k$ that is a complete hypergraph in both theories).

  Given an antichain $\cA\subseteq r(k)$, let $\cL_\cA$ be the
  language containing one predicate symbol $P_A$ of arity $k(P_A)\df\lvert A\rvert$
  for every $A\in\cA$.
  We say that $\phi$ satisfies $\Disc[\cA]$ (\cite{Tow17,ACHP18}) if for every $\psi\in\HomT{T_{\cL_\cA}}$ and every coupling $\xi$ of $\phi$ and $\psi$, if $\rn{K}$ is the exchangeable array in $\cK_{\NN_+}[\TkHypergraph\cup T_{\cL_\cA}]$ associated with $\xi$, then we
  have
  \begin{align*}
    \begin{multlined}
      \PP[(1,\ldots,k)\in R_E(\rn{K})\land \forall A\in\cA, \iota_A\in R_{P_A}(\rn{K})]
      \\
      =
      \phi(\rho_k)\cdot\PP[\forall A\in\cA, \iota_A\in R_{P_A}(\rn{K})],
    \end{multlined}
  \end{align*}
  that is, the events $(1,2,\ldots,k)\in R_E(\rn K)$ and $\forall A\in \cA,\iota_A\in R_{P_A}(\rn K)$ are
  independent.
\end{definition}

\smallskip
In~\cite{Tow17}, the definition of $\Disc[\cA]$ further requires symmetry of the predicate symbols $P_A$, but
it was shown in~\cite{ACHP18} that this condition can be dropped.

\subsection{Useful theories and objects} \label{sec:useful}

In this final preliminary subsection, we define some theories and limit objects that are necessary to formally
state some of our main results. We start with a very general definition (that nonetheless
will be used in full generality in Theorem~\ref{thm:fullUCouple}) and then derive
all others as special cases.

For $c\geq 2$, let $\Pi_c\df\{p=(p_i)_{i=1}^c \in (0,1)^c \mid \sum_{i=1}^c p_i = 1\}$ be the interior of the
standard $(c-1)$-dimensional simplex. Also, given $x\in\cE_n$, let $\sigma_x\in S_n$ be the unique permutation
such that $x_{\{\sigma_x^{-1}(1)\}} < \cdots < x_{\{\sigma_x^{-1}(n)\}}$ when the coordinates $(x_{\{i\}} \mid
i\in[n])$ are distinct, and define it arbitrarily otherwise.

  \begin{definition}[$S_k$-action theories]\label{def:actiontheories}
  Let $k\in\NN_+$, let $\cL$ be a language containing only predicate symbols of arity exactly $k$, let
  $\Theta\function{S_k\times \cL}{\cL}$ be a (left) action of $S_k$ on $\cL$ and write
  $\sigma\cdot P\df\Theta(\sigma,P)$. The canonical theory $T_\Theta$ is defined as the theory over $\cL$ with
  axioms
  \begin{align}
    \left(\bigwedge_{1\leq i < j\leq k} x_i\neq x_j\right) & \equiv \left(\bigvee_{P\in\cL} P(x_1,\ldots,x_k)\right);
    \label{eq:action:existsone}
    \\
    P(x_{\sigma(1)},\ldots,x_{\sigma(k)}) & \equiv (\sigma\cdot P)(x_1,\ldots,x_k) & (P\in\cL,\sigma\in S_k);
    \label{eq:action:action}
    \\
    \neg P(x_1,\ldots,x_k) & \lor \neg P'(x_1,\ldots,x_k) & (P,P'\in\cL, P\neq P').
    \label{eq:action:unique}
  \end{align}

  Given a $\Theta$-invariant $p = (p_P)_{P\in\cL}\in[0,1]^\cL$ with $\sum_{P\in\cL} p_P = 1$, the
  \emph{$(\Theta,p)$-quasirandom homomorphism} is the homomorphism $\psi_{\Theta,p}\in\HomT{T_\Theta}$
  corresponding to picking at random for each $k$-set $A$, independently of other $k$-sets, an orbit
  $O\subseteq\cL$ of the action $\Theta$ with probability $\sum_{P\in O}p_P$ then uniformly at random choosing
  an $S_k$-equivariant assignment of the $k$-tuples with image $A$ to the elements of $O$. A $T_\Theta$-on
  $\cN^Z$ representing $\psi_{\Theta,p}$ is given by\footnote{We will check that all axioms of
  $T_\Theta$ are satisfied and provide an alternate syntactic description as part
  of Proposition~\ref{prop:ThetapQR}.}
  \begin{align}\label{eq:ThetapQRcNZ}
    \cN^Z_P
    & \df
    \{x\in\cE_k \mid x_{[k]}\in Z_{\sigma_x\cdot P}\}
    \qquad (P\in\cL),
  \end{align}
  where $Z = (Z_P)_{P\in\cL}$ is a measurable partition of $[0,1]$ with $\lambda(Z_P) = p_P$ ($P\in\cL$).
\end{definition}

Let us now note a few special cases that will play an active role in our paper.

\begin{definition}[$c$-colored $k$-hypergraphs]
Let $\mathcal L=\{E_1,\ldots,E_c\}$ and assume that the action
$\Theta$ is trivial. In that case we will denote the theory $T_\Theta$
by $T_{c,k}$ and call it the \emph{theory of $c$-colored $k$-hypergraphs}.
The $(\Theta,p)$-quasirandom homomorphism will be called
\emph{quasirandom $c$-colored $k$-hypergraphon with densities $p$} and denoted
by $\psi_{k,p}$.
\end{definition}

\begin{definition}[quasirandom $k$-hypergraphons]\label{def:hypergraphons}
  Let us further specify $c=2$ in the previous definition. Since $E_2$ is the
  negation of $E_1$ and hence can be safely removed, the theory
  $T_\Theta$ is isomorphic to $\TkHypergraph$. For $p\in (0,1)$, the $(\Theta,(p,1-p))$-quasirandom homomorphism
  is called  the \emph{quasirandom $k$-hypergraphon of density $p$}; it will
  also be denoted by $\psi_{k,p}$.
\end{definition}

\begin{definition}[Colorings]\label{def:colorings}
Letting in Definition~\ref{def:actiontheories} $k=1$ instead, and keeping the action
$\Theta$ trivial, we see that $T_\Theta$
is naturally isomorphic to the theory $\TcColoring$. The quasirandom object will
be called \emph{$c$-coloring with densities $p$}, $p\in\Pi_c$, and denoted by
$\psi_p\in\HomT{\TcColoring}$.
\end{definition}

\begin{definition}[$k$-tournaments]
  Let now $\mathcal L=\{E_1,E_2\}$ and $k\geq 2$, but this time the action $\Theta$
  is not trivial but instead given by the sign homomorphism $\sgn\function{S_k}{S_2}$.
  Then the only $\Theta$-invariant $p$ is $p_1=p_2=1/2$ and, as in the case of
  hypergraphons, we can exclude $E_2$ from the theory. We call it
  the theory of \emph{$k$-tournaments} and denote by $\TkTournament$; intuitively,
  this theory corresponds to choosing one of the two possible orientations for every
  $k$-set. The quasi-random object $\psi_{\Theta, (1/2, 1/2)}$ will then be called
  the {\em quasirandom $k$-tournamon} and denoted by $\psi_k$; thus, $\psi_k\in
  \HomT{\TkTournament}$, and $\psi_2$ is the ordinary quasi-random tournamon.
\end{definition}

\section{Main results}
\label{sec:mainresults}

In this section we present the main results. We remark that some of these
results follow trivially from definitions and we will point these out as we
go along.

\begin{theorem}\label{thm:anti-monotone}
  The properties $\Independence$, $\UCouple$ and $\UInduce$ are anti-monotone in the sense that
  $P[\ell]\implies P[\ell-1]$.
\end{theorem}

For $\Independence$ and $\UCouple$, this theorem trivially follows from definitions. Even though it is
possible to give an ad hoc proof that $\UInduce$ is also anti-monotone, this follows trivially from its
equivalence with symmetric locality (Theorem~\ref{thm:UInduce} below) and the fact that symmetric locality is
trivially anti-monotone.

\begin{theorem}\label{thm:inter-properties}
For any $\ell\in\NN$, $\Independence[\ell]\implies \UCouple[\ell]\implies
\UInduce[\ell]$.
\end{theorem}
The second implication follows trivially from the definitions.

\medskip
The next theorem concerns preservation of properties under open
interpretations.

\begin{theorem}[Naturality]\label{thm:naturality}
  Let $I\interpret{T_1}{T_2}$ be an open interpretation and let $\ell\in\NN$. The following hold for
  any $\phi\in\HomT{T_2}$.
  \begin{enumerate}
  \item If $\phi$ is uniquely coupleable with some $\psi\in\HomT{T'}$, then
      $\phi^I$ is uniquely coupleable with
      $\psi$.\label{thm:naturality:uniquecouplings}
  \item If $\phi\in\Independence[\ell]$, then
      $\phi^I\in\Independence[\ell]$.\label{thm:naturality:Independence}
  \item If $\phi\in\UCouple[\ell]$, then
      $\phi^I\in\UCouple[\ell]$.\label{thm:naturality:UCouple}
  \item If $\phi\in\UInduce[\ell]$, then
      $\phi^I\in\UInduce[\ell]$.\label{thm:naturality:UInduce}
  \end{enumerate}
\end{theorem}

Item~\ref{thm:naturality:Independence} follows trivially from the definition
of $I(\cN)$ applied to an $\ell$-independent $T_2$-on $\cN$ such that
$\phi=\phi_\cN$. Note also that item~\ref{thm:naturality:UCouple} follows
trivially from item~\ref{thm:naturality:uniquecouplings}. Furthermore,
applying this theorem to the axiom-adding interpretation
$I\interpret{T_\cL}{T}$, where $\cL$ is the language of $T$, we see that all
our main notions do not depend on non-logical axioms. Nonetheless, using theories and theons (as opposed to
arbitrary Euclidean structures) helps to better orient ourselves and put many
of the results in the ``right'' focus.

The next theorem says that both $\Independence$ and $\UCouple$ are preserved
under independent couplings.

\begin{theorem}\label{thm:naturalityindcoup}
  Let $\phi_1\in\HomT{T_1}$ and $\phi_2\in\HomT{T_2}$. The following hold for $\ell\in\NN$.
  \begin{enumerate}
  \item If $\phi_1,\phi_2\in\Independence[\ell]$, then
      $\phi_1\otimes\phi_2\in\Independence[\ell]$.\label{thm:naturalityindcoup:Independence}
  \item If $\phi_1,\phi_2\in\UCouple[\ell]$, then
      $\phi_1\otimes\phi_2\in\UCouple[\ell]$.\label{thm:naturalityindcoup:UCouple}
  \end{enumerate}
\end{theorem}

Remarkably, this is not true for $\UInduce$. For example, the linear order
$\psi\in\HomT{\TLinOrder}$ satisfies $\UInduce[\ell]$ for every $\ell\in\NN$
but does not satisfy $\UCouple[1]$
(Theorem~\ref{thm:separationUInduceUCouple} below). So the quasirandom permuton
$\psi\otimes\psi\in\HomT{\TLinOrder\cup\TLinOrder}$
(see~\cite[Example~6]{CR19}) cannot satisfy $\UInduce[1]$ by
Theorem~\ref{thm:UCouple}\ref{thm:UCouple:UCouple}$\equiv$\ref{thm:UCouple:orderUInduce} below.
This can be also verified by a direct computation, of course.

The next five theorems concern separations between properties, either allowing general theories or restricted
to the theory of hypergraphs.

\begin{theorem}\label{thm:separationupward}
 $\Independence[\ell]$ does not imply
$\UInduce[\ell+1]$, not even when restricted to the theory of $k$-hypergraphs
as long as $k > \ell$.
\end{theorem}

In fact, this theorem is a consequence of Theorems~\ref{thm:UInduce->CliqueDisc}
and~\ref{thm:separationIndependenceDisc} below.

\medskip
The following two theorems are included since the separating objects are quite
natural and explicit and the proofs are simpler. But in a sense they will be
superseded by Theorems~\ref{thm:separationUCoupleIndependencehypergraph} and~\ref{thm:separationUInduceUCouplehypergraph}.

\begin{theorem}\label{thm:separationUCoupleIndependence}
  For every $\ell\in\NN_+$, the quasirandom $(\ell+1)$-tournamon $\psi_{\ell+1}$ satisfies $\UCouple[\ell]$
  but does not satisfy $\Independence[\ell]$.
\end{theorem}

\begin{theorem}\label{thm:separationUInduceUCouple}
  The linear order $\psi\in\HomT{\TLinOrder}$ satisfies $\UInduce[\ell]$ for every $\ell\in\NN$ but does not
  satisfy $\UCouple[1]$.
\end{theorem}

\begin{theorem}\label{thm:separationUCoupleIndependencehypergraph}
  For $\ell\geq 1$, there exists $\phi\in\HomT{\TkHypergraph[(\ell+2)]}$ satisfying $\UCouple[\ell]$ but not
  satisfying $\Independence[\ell]$.
\end{theorem}

\begin{theorem}\label{thm:separationUInduceUCouplehypergraph}
  For $\ell\geq 1$ odd, there exists $\phi\in\HomT{\TkHypergraph[(\ell+2)]}$ satisfying $\UInduce[\ell]$ but
  not satisfying $\UCouple[1]$.
\end{theorem}

The next theorem lists several properties that are equivalent to
$\UCouple[\ell]$. These include both alternative formulations and complete
sets of tests for unique coupleability.

\begin{theorem}[Characterization of $\UCouple$]\label{thm:UCouple}
  Let $\ell\in\NN_+$. The following are equivalent for $\phi\in\HomT{T}$.
  \begin{enumerate}
  \item $\phi\in\UCouple[\ell]$.\label{thm:UCouple:UCouple}
  \item For every $\ell'\in[\ell]$, there exists $p\in(0,1)$ such that $\phi$ is uniquely coupleable with the
    quasirandom $\ell'$-hypergraphon $\psi_{\ell',p}$.\label{thm:UCouple:QRhyper}
  \item There exist $p_1,\ldots,p_\ell\in(0,1)$ such that $\phi$ is uniquely coupleable with the independent
    coupling $\psi_{1,p_1}\otimes\cdots\otimes\psi_{\ell,p_\ell}$ of the quasirandom $\ell'$-hypergraphons
    $\psi_{\ell',p_{\ell'}}$ for $\ell'\in[\ell]$.\label{thm:UCouple:QRhyperindepcoup}
  \item $\phi$ is weakly $\ell$-independent.\label{thm:UCouple:weakindep}
  \item Every $T$-on $\cN$ with $\phi_\cN =\phi$ is weakly $\ell$-independent.\label{thm:UCouple:weakindepall}
  \item $\phi$ is $\ell$-local.\label{thm:UCouple:local}
  \item The independent coupling of $\phi$ with the linear order $\psi\in\HomT{\TLinOrder}$ satisfies
    $\UInduce[\ell]$.\label{thm:UCouple:orderUInduce}
  \end{enumerate}
\end{theorem}

Note that since $\ell'$-hypergraphons have rank at most $\ell'$, a
posteriori, we can also strengthen items~\ref{thm:UCouple:QRhyper}
and~\ref{thm:UCouple:QRhyperindepcoup} by replacing existential quantifiers
on $p,p_1,\ldots,p_\ell$ with universal ones. Also, since the linear order has rank $1$, a posteriori, we can
strengthen item~\ref{thm:UCouple:orderUInduce} to say that \emph{every}
coupling of $\phi$ with the linear order satisfies $\UInduce[\ell]$. In the
actual proof of the implication~\ref{thm:UCouple:QRhyper}$\implies$\ref{thm:UCouple:UCouple} (that, arguably, is our technically most difficult
result), we go in the opposite direction and painstakingly ``bootstrap'' the
premise in~\ref{thm:UCouple:QRhyper} to the unique coupleability with
increasingly larger families of objects.

Let us also point out that, given Theorem~\ref{thm:naturalityindcoup}\ref{thm:naturalityindcoup:UCouple}, one might
expect that, in general, if each one of $\psi_1,\ldots,\psi_t$ is uniquely
coupleable with a given $\phi$, then the same should hold for their
independent coupling $\psi_1\otimes\cdots\otimes\psi_t$; this would
immediately
give~\ref{thm:UCouple:QRhyper}$\implies$\ref{thm:UCouple:QRhyperindepcoup} in
Theorem~\ref{thm:UCouple}. However, this question has turned out surprisingly
difficult in full generality (see Section~\ref{sec:concl} for a discussion).

The next, more modest, theorem provides properties equivalent to $\UInduce[\ell]$.

\begin{theorem}[Characterization of $\UInduce$]\label{thm:UInduce}
  The following are equivalent for $\ell\in\NN_+$ and $\phi\in\HomT{T}$.
  \begin{enumerate}
  \item $\phi\in\UInduce[\ell]$.\label{thm:UInduce:UInduce}
  \item There exists $p\in(0,1)$ such that $\phi$ is uniquely inducible by every
    $\psi\in\HomT{\TkHypergraph[\ell]}$ with $\psi(\rho_\ell) = p$.\label{thm:UInduce:singledensity}
    \item $\phi$ is symmetrically $\ell$-local.\label{thm:UInduce:symlocal}
  \end{enumerate}
\end{theorem}

The next two theorems completely classify $\Independence[k-1]$ and $\UCouple[k-1]$ when all arities are at
most $k$. These can be thought of as analogues of full quasirandomness for these families of properties.

\begin{theorem}\label{thm:fullIndependence}
  Let $k\in\NN_+$ and suppose that $k(P)\leq k$ for all $P\in\cL$. Let $T$ be a theory over $\cL$ and
  $\phi\in\HomT{T}$. Then $\phi\in\Independence[k-1]$ if and only if there exist $c\in\mathbb N_+$,
  $p\in\Pi_c$ and an open interpretation $I\interpret{T}{\TcColkHyp[{c}{k}]}$ such that $\phi = \psi_{k,p}^I$.
\end{theorem}

\begin{theorem}\label{thm:fullUCouple}
  Let $k\in\NN_+$ and suppose that $k(P)\leq k$ for all $P\in\cL$. Let $T$ be a theory over $\cL$ and
  $\phi\in\HomT{T}$. Then $\phi\in\UCouple[k-1]$ if and only if there exists a language $\mathcal L'$ whose
  predicate symbols have arity exactly $k$, an
  action $\Theta\function{S_k\times\cL'}{\cL'}$, a $\Theta$-invariant $p = (p_P)_{P\in\cL'}\in[0,1]^{\cL'}$
  with $\sum_{P\in\cL'} p_P = 1$ and an open interpretation $I\interpret{T}{T_\Theta}$ such that $\phi =
  \psi_{\Theta,p}^I$.
\end{theorem}

Finally, the last three theorems compare our quasirandomness properties with
the discrepancy properties in the literature. As we remarked in the
introduction, the results of~\cite{Tow17} imply that
$\Dev[k-1]=\Disc[\cA_{k-1}]$ is the strongest discrepancy property below full
quasirandomness and $\Disc[\{[\ell+1]\}]$ is the weakest discrepancy property
above $\CliqueDisc[\ell]$. This together with
Theorems~\ref{thm:anti-monotone}, \ref{thm:inter-properties}
and~\ref{thm:separationUInduceUCouplehypergraph} and the three theorems below
justify the Hasse diagram of Figure~\ref{fig:hypergraph} between the families
$\Independence$ and $\UInduce$ and the discrepancy properties in the
literature.

The following theorem trivially follows from definitions.

\begin{theorem}\label{thm:UInduce->CliqueDisc}
  For every $k\geq\ell\geq 1$ and every $\phi\in\HomT{\TkHypergraph}$, if $\phi\in\UInduce[\ell]$, then
  $\phi\in\CliqueDisc[\ell]$.
\end{theorem}

\medskip

\begin{theorem}\label{thm:separationDevUInduce}
  For every $k\in\NN_+$, there exists $\phi\in\HomT{\TkHypergraph}$ satisfying $\Dev[k-1]$ but not satisfying
  $\UInduce[1]$.
\end{theorem}

\begin{theorem}\label{thm:separationIndependenceDisc}
  For every $k > \ell\geq 1$, there exists $\phi\in\HomT{\TkHypergraph}$ satisfying $\Independence[\ell]$ but
  not satisfying $\Disc[\{[\ell+1]\}]$.
\end{theorem}

Table~\ref{tab:references} contains pointers to where each of the theorems (or their parts) are proved.

\begin{table}
  \begin{center}
    \begin{tabular}{lll}
      \multicolumn{2}{l}{Theorem} & Proof location\\
      \hline
      \ref{thm:anti-monotone}
      &
      &
      Section~\ref{sec:UInduce}
      \\
      \ref{thm:inter-properties}
      &
      &
      Section~\ref{sec:basic}
      \\
      \ref{thm:naturality}
      & &
      Section~\ref{sec:naturality}
      \\
      \ref{thm:naturalityindcoup}
      & &
      Section~\ref{sec:basic}
      \\
      \ref{thm:separationupward}
      & &
      Section~\ref{sec:separations}
      \\
      \ref{thm:separationUCoupleIndependence}
      & &
      Section~\ref{sec:separations}
      \\
      \ref{thm:separationUInduceUCouple}
      & &
      Section~\ref{sec:separations}
      \\
      \ref{thm:separationUCoupleIndependencehypergraph}
      & &
      Section~\ref{sec:separations}
      \\
      \ref{thm:separationUInduceUCouplehypergraph}
      & &
      Section~\ref{sec:separations}
      \\
      \ref{thm:UCouple}
      &
      \ref{thm:UCouple:UCouple}$\equiv$\ref{thm:UCouple:QRhyper}$\equiv$\ref{thm:UCouple:QRhyperindepcoup}
      &
      Lemma~\ref{lem:QRhyper}
      \\
      &
      \ref{thm:UCouple:UCouple}$\equiv$\ref{thm:UCouple:weakindep}$\equiv$\ref{thm:UCouple:weakindepall}
      &
      Lemma~\ref{lem:weakindep}
      \\
      &
      \ref{thm:UCouple:weakindep}$\implies$\ref{thm:UCouple:local}
      &
      Lemma~\ref{lem:weakindep->local}
      \\
      &
      \ref{thm:UCouple:local}$\implies$\ref{thm:UCouple:orderUInduce}
      &
      Lemma~\ref{lem:localorderUInduce}
      \\
      &
      \ref{thm:UCouple:orderUInduce}$\implies$\ref{thm:UCouple:QRhyper}
      &
      Lemma~\ref{lem:orderUInducecolors}
      \\
      \ref{thm:UInduce}
      &
      \ref{thm:UInduce:UInduce}$\equiv$\ref{thm:UInduce:singledensity}
      &
      Lemma~\ref{lem:UInduce:singledensity}
      \\
      &
      \ref{thm:UInduce:symlocal}$\implies$\ref{thm:UInduce:UInduce}
      &
      Lemma~\ref{lem:symlocal->UInduce}
      \\
      &
      \ref{thm:UInduce:UInduce}$\implies$\ref{thm:UInduce:symlocal}
      &
      Lemma~\ref{lem:UInduce->symlocal}
      \\
      \ref{thm:fullIndependence}
      & &
      Section~\ref{sec:fullQR}
      \\
      \ref{thm:fullUCouple}
      & &
      Section~\ref{sec:fullQR}
      \\
      \ref{thm:UInduce->CliqueDisc}
      & &
      Trivial (see Definitions~\ref{def:UCoupleUInduce} and~\ref{def:CliqueDisc})
      \\
      \ref{thm:separationDevUInduce}
      & &
      Section~\ref{sec:separations}
      \\
      \ref{thm:separationIndependenceDisc}
      & &
      Section~\ref{sec:separations}
    \end{tabular}
    \caption{Proof locations for theorems of Section~\ref{sec:mainresults}.}
    \label{tab:references}
  \end{center}
\end{table}

\section{Basic properties and the first equivalence}
\label{sec:basic}

In this section we present some initial properties about the notions we have
defined and prove the easiest equivalence in Theorem~\ref{thm:UCouple}
between items~\ref{thm:UCouple:UCouple}, \ref{thm:UCouple:weakindep}
and~\ref{thm:UCouple:weakindepall}. The first proposition says that only
trivial objects can have unique coupleability parameter greater or equal to
its rank; this stems from the fact that non-trivial objects are not uniquely
coupleable with themselves.

\begin{proposition}\label{prop:rankbasic}
  Let $\phi\in\HomT{T}$ and $r\df\rk(\phi)$.
  \begin{enumerate}
  \item $r=0$ if and only if $\phi\in\bigcap_{\ell\in \mathbb N} \UCouple[\ell]$.\label{prop:rankbasic:rank0}
  \item If $r>0$ then $\phi\notin\UCouple[r]$.\label{prop:rankbasic:selfcouple}
  \end{enumerate}
\end{proposition}

\begin{proof}
  Note that $r = 0$ if and only if all peons $\cN_P$ are trivial (that is, $\cN_P=\emptyset$
  or $\cN_P=\cE_{k(P)}$ a.e.), which in turn is equivalent to having $\phi(\langle K\rangle)\in\{0,1\}$ for every finite set $V$ and every $K\in\cK_V[T]$. This implies that there is a unique $K\in\cK_V[T]$ with
  $\phi(\langle K\rangle)=1$ and this $K$ must further have full automorphism group
  $\Aut(K)=S_V$.

  Let now $\psi\in\HomT{T'}$ for some theory $T'$, and assume that $\xi$ is a
  coupling of $\phi$ and $\psi$. Fix a
  $(T\cup T')$-on $\cN$ such that $\xi = \phi_\cN$. Then for every $K\in\cK_V[T\cup T']$ with $V$
  finite we have $\Tind(K,\cN) = \Tind(I(K),I(\cN))\cap\Tind(I'(K),I'(\cN))$, where $I\interpret{T}{T\cup T'}$
  and $I'\interpret{T'}{T\cup T'}$ are the structure-erasing interpretations.

  If $r = 0$, we get $\xi(\langle K\rangle) = \phi(\langle I(K)\rangle)\psi(\langle I'(K)\rangle)$ (since $\phi$ is
  0-1 valued) so the
  forward direction of item~\ref{prop:rankbasic:rank0} follows.

  The backward direction of item~\ref{prop:rankbasic:rank0} clearly follows from
  item~\ref{prop:rankbasic:selfcouple}, so let us prove the latter by contradiction. Suppose that
  $\phi\in\UCouple[r]$ and fix a $T$-on $\cN$ such that $\phi=\phi_\cN$ and $\rk(\cN)=r$. Consider the $(T\cup
  T)$-on $\cH\df\cN\disjcup\cN$ in which both copies of each predicate symbol $P$ get mapped to $\cN_P$, i.e.,
  $\cH$ is the coupling of $\cN$ with itself. Since $\rk(\cH) = \rk(\cN) = r$ and
  $\phi\in\UCouple[r]$, we must have $\phi_\cH = \phi\otimes\phi$.

  Fix a finite set $V$ and $K\in\cK_V[T]$ and let $K_2\in\cK_V[T\cup T]$ be given by setting $R_P(K_2)\df
  R_P(K)$ for both copies of each predicate symbol $P$. Then we have
  \begin{align*}
    \phi(\langle K\rangle)
    & =
    \tind(K,\cN)
    =
    \tind(K_2,\cH)
    =
    (\phi\otimes\phi)(\langle K_2\rangle)
    =
    \phi(\langle K\rangle)^2,
  \end{align*}
  so we must have $\phi(\langle K\rangle)\in\{0,1\}$. Hence $r=0$, and
  item~\ref{prop:rankbasic:selfcouple} follows.
\end{proof}

The next two propositions will make use of the theon uniqueness theorems~\cite[Theorems~3.9 and~3.11,
  Proposition~7.7]{CR19}. Recall from~\cite[Definition~3.8 and Sct.~7]{CR19} that for a sequence of symmetric
(i.e., $S_d$-invariant) functions $f = (f_d)_{d=1}^k$ with $f_d\function{\cE_d(\Omega)}{\Omega'}$ the sequence
of functions $\widehat{f} = (\widehat{f}_d)_{d=1}^k$ with
$\widehat{f}_d\function{\cE_d(\Omega)}{\cE_d(\Omega')}$ is defined by
\begin{align*}
  \widehat{f}_d(x)_A & \df f_{\lvert A\rvert}(\iota_A^*(x)) & (A\in r(d)).
\end{align*}

As we have seen in the introduction, a positive homomorphism
$\phi\in\Independence[\ell]$ can have geometric realizations far from being
$\ell$-independent (cf.~\eqref{eq:W} and~\eqref{eq:Wskew}). The next
proposition says that for rank the situation is precisely the opposite.

\begin{proposition}\label{prop:rankae}
  For every peon $\cN\subseteq \cE_k(\Omega)$ there exists another peon $\cH\subseteq \cE_k(\Omega)$ such that
  $\rk(\cH)=\rk(\phi_\cN)$ and $\cH=\cN$ a.e. Moreover, if $\cN$ is $\ell$-independent for some $\ell\leq
  k$, then $\cH$ can be taken to also be $\ell$-independent.
\end{proposition}

\begin{proof}
  Let $\mu$ be the measure of $\Omega$ and $X$ be its underlying space, let $r \df \rk(\phi_\cN)$ and define
  the function $W\function{\cE_{k,r}(\Omega)}{[0,1]}$ by
  \begin{align}\label{eq:rankae:W}
    W(x) & \df \mu(\{y\in X^{\binom{[k]}{> r}} \mid (x,y)\in\cN\}),
  \end{align}
  defining it arbitrarily when this set is not measurable.
  Fubini's Theorem ensures that this function is measurable
  so we define
  \begin{align*}
    \cH & \df W^{-1}(1)\times X^{\binom{[k]}{> r}}.
  \end{align*}
  Clearly $\rk(\cH)\leq r$ so to prove that $\cH=\cN$ a.e., it is enough to
  show that $W$ is 0-1 valued a.e.

  Since $\rk(\phi_\cN)=r$, we know that there exists a peon $\cG$ over some space $\Omega' = (X',\cA',\mu')$ such
  that $\phi_\cG = \phi_\cN$ and $\rk(\cG) = r$. By theon uniqueness~\cite[Proposition~7.7]{CR19}, there exist
  sequences $f = (f_d)_{d=1}^k, g = (g_d)_{d=1}^k$ of symmetric measure preserving on
  h.o.a.\ (higher order arguments) functions ($f_d\function{\cE_d}{\Omega}$
  and $g_d\function{\cE_d}{\Omega'}$) such that
  \begin{align}\label{eq:rankae:uniqueness}
    \widehat{f}_k(z)\in\cN & \equiv \widehat{g}_k(z)\in\cG
  \end{align}
  for almost every $z\in\cE_k$. From the structure of the function $\widehat{f}_k$, we can decompose
  it as
  \begin{align*}
    \widehat{f}_k(x,y) & = (F_1(x), F_2(x,y)),
  \end{align*}
  for every $(x,y)\in\cE_{k,r}\times[0,1]^{\binom{[k]}{> r}}$, where
  $F_1\function{\cE_{k,r}}{\cE_{k,r}(\Omega)}$ and
  $F_2\function{\cE_k}{X^{\binom{[k]}{> r}}}$ are given by
  \begin{align*}
    F_1(x)_A & \df f_{\lvert A\rvert}(\iota_A^*(x)), &
    F_2(x,y)_A & \df f_{\lvert A\rvert}(\iota_A^*(x,y)).
  \end{align*}
  We perform a
  similar decomposition of $\widehat{g}_k$ in terms of functions $G_1\function{\cE_{k,r}}{\cE_{k,r}(\Omega')}$
  and $G_2\function{\cE_k}{(X')^{\binom{[k]}{> r}}}$.

  Since the functions $f_d$ are measure preserving on h.o.a., it follows that $F_1$ is measure preserving and for every
  $x\in\cE_{k,r}$ the restriction
  $F_2(x,\place)\function{[0,1]^{\binom{[k]}{> r}}}{X^{\binom{[k]}{> r}}}$ is measure
  preserving. Hence Fubini's Theorem applied to~\eqref{eq:rankae:uniqueness} implies
  \begin{align*}
    W(F_1(x)) & = \lambda(\{y\in[0,1]^{\binom{[k]}{> r}} \mid (G_1(x), G_2(x,y))\in\cG\})
  \end{align*}
  for almost every $x\in\cE_{k, r}$. But since $\rk(\cG)=r$, the measure above can
  only be $0$ or $1$ (as $G_2(x,y)$ contains only coordinates with $\lvert A\rvert > r$). Since $F_1$ is
  measure preserving, this implies that $W(z)\in\{0,1\}$ for almost every
  $z\in\cE_{k,r}(\Omega)$ and thus $\cH = \cN$ a.e.

  We have already shown that $\rk(\cH)\leq r$ and since $\cH = \cN$ $\mu$-a.e. implies $\phi_\cH = \phi_\cN$,
  the other inequality must also hold.

 The last statement is obvious from the construction.
\end{proof}

As we have seen in Section~\ref{sec:notation}, given an open interpretation
$I\interpret{T_1}{T_2}$ and a $T_2$-on $\cH$, the $T_1$-on $I(\cH)$
represents the limit object constructed from $\phi_\cH$ via $I$, i.e., we
have $\phi_{I(\cH)} = \phi_\cH^I$. However, given a $T_1$-on $\cN$ and
$\phi\in\HomT{T_2}$ such that $\phi^I = \phi_\cN$, it is \emph{not} true that
there exists a $T_2$-on $\cH$ such that both $I(\cH) =
\cN$ a.e.\ and $\phi_\cH=\phi$ (see~\cite[Example~45]{CR19}). The next proposition says in essence that
this example is the worst that can happen: at the cost of adding an extra
dummy variable, we can find an $\cH$ such that $I(\cH)_P =
\cN_P\times\cE_{k(P)}$ a.e.\ and $\phi_\cH=\phi$.

\begin{proposition}\label{prop:theonalignment}
  Let $I\interpret{T_1}{T_2}$, let $\phi\in\HomT{T_2}$, and let $\cN$ be a $T_1$-on over $\Omega$
  such that $\phi^I = \phi_\cN$. Then there exists a $T_2$-on $\cH$ over $\Omega\times\Omega$ such that
  $\phi_\cH = \phi$ and $I(\cH)_P = \cN_P\times\cE_{k(P)}(\Omega)$ a.e., for every predicate symbol $P$
  in the language of $T_1$.

  Furthermore, if $T_2 = T_1\cup T'$ for some $T'$ and $I$ is the structure-erasing interpretation, then $\cH$
  can be taken to satisfy $I(\cH)_P \df \cH_P = \cN_P\times\cE_{k(P)}(\Omega)$ everywhere for every predicate
  symbol $P$ in the language of $T_1$.
\end{proposition}

\begin{proof}
  For $i\in[2]$, let $\cL_i$ be the language of $T_i$ and let $k_i\df\max_{P\in\cL_i} k(P)$. Let $\cG$ be a
  $T_2$-on over $\Omega$ such that $\phi_\cG = \phi$. Since $\phi_{I(\cG)} = \phi^I = \phi_\cN$, by theon
  uniqueness~\cite[Proposition~7.7]{CR19}, there exists a sequence $h = (h_d)_{d=1}^{k_1}$ of symmetric
  measure preserving on h.o.a.\ functions ($h_d\function{\cE_d(\Omega)\times\cE_d(\Omega)}{\Omega}$) such
  that
  \begin{align}\label{eq:theonalignment:uniqueness}
    \widehat{h}_{k(P)}(x,\widehat{x})\in I(\cG)_P & \equiv x\in\cN_P,
  \end{align}
  for every $P\in\cL_1$ and almost every
  $(x,\widehat{x})\in\cE_{k(P)}(\Omega)\times\cE_{k(P)}(\Omega)$. Extend the family $h$ by defining
  $h_d\function{\cE_d(\Omega)\times\cE_d(\Omega)}{\Omega}$ for $k_1 < d\leq \max\{k_1,k_2\}$ as
  $h_d(x,\widehat{x}) \df x_{[d]}$, and note that $h_d$ is symmetric and measure preserving on h.o.a.

  Define then the $T_2$-on $\cH$ over $\Omega\times\Omega$ by
  \begin{align}\label{eq:theonalignment:definition}
    \cH_Q & \df \widehat{h}_{k(Q)}^{-1}(\cG_Q)
  \end{align}
  for every $Q\in\cL_2$. By (the easy direction of) theon uniqueness~\cite[Proposition~7.7]{CR19}, it follows
  that $\phi_\cH = \phi_\cG = \phi$. On the other hand, the definition of $\cH$ ensures that
  \begin{align*}
    (x,\widehat{x})\in I(\cH)_P & \equiv \widehat{h}_{k(P)}(x,\widehat{x})\in I(\cG)_P
  \end{align*}
  for every $P\in\cL_1$ and every $(x,\widehat{x})\in\cE_{k(P)}(\Omega)\times\cE_{k(P)}(\Omega)$, which
  together with~\eqref{eq:theonalignment:uniqueness} implies $I(\cH)_P = \cN_P\times\cE_{k(P)}(\Omega)$
  a.e.

  \medskip

  For the case when $T_2 = T_1\cup T'$ for some $T'$ and $I$ is the structure-erasing interpretation, we
  define $\cH$ instead by using~\eqref{eq:theonalignment:definition} only for $Q\in\cL_2\setminus\cL_1$ and
  use $\cH_P\df\cN_P\times\cE_{k(P)}(\Omega)$ for every $P\in\cL_1$ (as
  required). By~\eqref{eq:theonalignment:uniqueness} this is only a zero-measure change from the previous
  definition so we still have $\phi_\cH = \phi$.
\end{proof}

Propositions~\ref{prop:rankae} and~\ref{prop:theonalignment} allow us to show
the equivalence in Theorem~\ref{thm:UCouple} between
items~\ref{thm:UCouple:UCouple}, \ref{thm:UCouple:weakindep}
and~\ref{thm:UCouple:weakindepall}.

\begin{lemma}[Theorem~\ref{thm:UCouple}\ref{thm:UCouple:UCouple}$\equiv$\ref{thm:UCouple:weakindep}$\equiv$\ref{thm:UCouple:weakindepall}]\label{lem:weakindep}
  The following are equivalent for $\phi\in\HomT{T}$ and $\ell\in\NN$.
  \begin{enumerate}[label={\roman*.}, ref={(\roman*)}]
  \item $\phi\in\UCouple[\ell]$.\label{prop:weakindep:UCouple}
  \item $\phi$ is weakly $\ell$-independent.\label{prop:weakindep:weakindep}
  \item Every $T$-on $\cN$ with $\phi=\phi_\cN$ is weakly $\ell$-independent.\label{prop:weakindep:weakindepall}
  \end{enumerate}
\end{lemma}

\begin{proof}
  \ref{prop:weakindep:weakindepall}$\implies$\ref{prop:weakindep:weakindep} is trivial.

  \medskip

\ref{prop:weakindep:weakindep}$\implies$\ref{prop:weakindep:UCouple}.

Let $\cN$ be a $T$-on over some space $\Omega=(X,\cA,\mu)$ such that the
  exchangeable array $\rn{K}$ corresponding to $\cN$ with respect to
  $\rn{\theta}$ picked in $\cE_{\NN_+}(\Omega)$ according to $\mu$ is
  independent from $(\rn{\theta}_A \mid A\in r(\NN_+,\ell))$. Let
$\psi\in\HomT{T'}$ for some theory $T'$ be such that $\rk(\psi)\leq\ell$
  and let $\xi\in\HomT{T\cup T'}$ be any coupling of $\phi$ and $\psi$. We
  have to prove that $\xi=\phi\otimes\psi$.

  Let also $I\interpret{T}{T\cup T'}$ and $I'\interpret{T'}{T\cup T'}$ be
  the structure-erasing interpretations. By
  Proposition~\ref{prop:theonalignment}, there exists a $(T\cup T')$-on
  $\cH$ over $\Omega\times\Omega$ such that $\xi=\phi_\cH$ and
  \begin{align}\label{eq:weakindep:theonalignment}
    \cH_P = \cN_P\times\cE_{k(P)}(\Omega)
  \end{align}
  for every $P$ in the language of $T$. By possibly changing zero-measure sets of the peons corresponding to
  $T'$ using Proposition~\ref{prop:rankae}, we may also assume $\rk(I'(\cH))=\rk(\psi)\leq\ell$.

  Let us pick $\rn{\eta}$ in $\cE_{\NN_+}(\Omega)$ according to $\mu$ and independently from
  $\rn{\theta}$; we view $(\rn{\theta},\rn{\eta})$ as a $\cE_{\NN_+}(\Omega\times\Omega)$-valued
  random variable distributed according to $\mu\otimes\mu$. Let
  $\rn{L}$ be the exchangeable array corresponding to $\cH$ with respect to $(\rn{\theta},\rn{\eta})$. Note
  that~\eqref{eq:weakindep:theonalignment} implies that $I(\rn{L}) = \rn{K}$, which
  in turn implies that
  $I(\rn{L})$ is independent from $((\rn{\theta}_A \mid A\in r(\NN_+,\ell)),\rn{\eta})$. On the other hand,
  since $\rk(I'(\cH))\leq\ell$, it follows that $I'(\rn{L})$ is completely determined by
  $((\rn{\theta}_A,\rn{\eta}_A) \mid A\in r(\NN_+,\ell))$, so $I(\rn{L})$ is independent from
  $I'(\rn{L})$. This means that for $m\in\NN_+$ and $K\in\cK_m[T\cup T']$, we have
  \begin{align*}
    \xi(\langle K\rangle)
    & =
    \PP[\rn{L}\rest_{[m]} = K]
    =
    \PP[I(\rn{L})\rest_{[m]} = I(K)\land I'(\rn{L})\rest_{[m]} = I'(K)]
    \\
    & =
    \PP[I(\rn{L})\rest_{[m]} = I(K)]\cdot\PP[I'(\rn{L})\rest_{[m]} = I'(K)]
    =
    \phi(\langle I(K)\rangle)\cdot\psi(\langle I'(K)\rangle),
  \end{align*}
  so $\xi = \phi\otimes\psi$, hence item~\ref{prop:weakindep:UCouple} follows.

  \medskip

 Let us prove~\ref{prop:weakindep:UCouple}$\implies$\ref{prop:weakindep:weakindepall}.
 Let $\Omega=(X,\cA,\mu)$ be an  atomless complete probability space and
   $\cN$ be a $T$-on over $\Omega$ with $\phi = \phi_\cN$. We have to prove that
  the exchangeable array $\rn{K}$ corresponding to $\cN$
  with respect to $\rn{\theta}$ picked in $\cE_{\NN_+}(\Omega)$
  according to $\mu$ is independent
  from $(\rn{\theta}_A \mid A\in r(\NN_+,\ell))$. For that, it is sufficient
  to show that for any $m\in\NN$, any $K\in\cK_m[T]$ and any measurable set
$B\subseteq\cE_{m,\ell}(\Omega)$, the events $\rn{K}\rest_{[m]}=K$ and
$(\rn{\theta}_A \mid A\in r(m,\ell))\in B$ are independent.

  Let $Q$ be a new $m$-ary predicate symbol and consider the $(T\cup T_{\{Q\}})$-on $\cH$ over $\Omega$
  given by $\cH_P\df\cN_P$ for every $P$ in the language of $T$ and $\cH_Q\df B\times
  X^{\binom{[m]}{>\ell}}$. Let also $I\interpret{T}{T\cup T_{\{Q\}}}$ and $I'\interpret{T_{\{Q\}}}{T\cup
    T_{\{Q\}}}$ be the structure-erasing interpretations so that $\phi_\cH$ is a
    coupling of $\phi$ and $\phi_{\cH}^{I'}$. Since $\rk(\phi_{\cH}^{I'})\leq
    \rk(\cH_Q)\leq\ell$ and $\phi\in \UCouple[\ell]$, we have $\phi_{\cH} =
    \phi\otimes \phi_{\cH}^{I'}$. Finally, let $S$ be the set of all $L\in
    \cK_m[T \cup T_{\{Q\}}]$ such that $I(L)=K$ and $(1,2,\ldots, m)\in
    R_Q(L)$. Then we have
    \begin{multline*}
      \PP[\rn K\rest_{[m]}=K \land (\rn{\theta}_A \mid A\in r(m,\ell))\in B]
      =
      \sum_{L\in S} \phi_\cH(\langle L\rangle)
      \\
      =
      \phi(\langle K\rangle )
      \sum_{L\in S}\phi_\cH^{I'}(\langle I'(L)\rangle)
      =
      \PP[\rn K\rest_{[m]}=K]
      \cdot \PP[(\rn{\theta}_A \mid A\in r(m,\ell))\in B],
    \end{multline*}
    which completes the proof.
\end{proof}

The alternative characterization of $\UCouple$ via weak independence gives
easy proofs of Theorems~\ref{thm:inter-properties} and~\ref{thm:naturalityindcoup}.

\begin{proofof}{Theorem~\ref{thm:inter-properties}}
  $\Independence[\ell]\implies\UCouple[\ell]$.

  Let $\cN$ be an $\ell$-independent $T$-on, and let $\rn{K}$ be the exchangeable array corresponding to $\cN$.
  Then each $R_P(\rn{K})$ depends only on the coordinates $\rn{\theta}_A$ with
  $\lvert A\rvert>\ell$
  (see~\eqref{eq:arraycryptomorphism})  and hence is
  independent from $(\rn{\theta}_A \mid A\in r(A,\ell))$. Therefore, $\cN$ is
  weakly $\ell$-independent and $\Independence[\ell]\implies\UCouple[\ell]$
  follows from Lemma~\ref{lem:weakindep}.

  The implication~$\UCouple[\ell]\implies\UInduce[\ell]$ follows trivially from the definitions.
\end{proofof}

\begin{proofof}{{Theorem~\ref{thm:naturalityindcoup}}}

For item~\ref{thm:naturalityindcoup:Independence}, if $\cN^1$ and $\cN^2$ are
$\ell$-independent theons then $\cN^1\otimes \cN^2$ is also $\ell$-independent, from
which the statement follows.

For item~\ref{thm:naturalityindcoup:UCouple}, pick arbitrarily theons $\cN^1$
and $\cN^2$ such that $\phi_i=\phi_{\cN^i}$. Let $(\rn{\theta}^1, \rn{\theta}^2)$
be uniformly distributed in $\cE_{\NN_+} \times \cE_{\NN_+}$, and let $\rn K$ be the
exchangeable array corresponding to $\cN^1\otimes \cN^2$ with respect to
$(\rn{\theta}^1, \rn{\theta}^2)$. Note that for $i\in [2]$ and for the
structure-erasing interpretation $I_i\interpret{T_i}{T_1\cup T_2}$, the
exchangeable array corresponding to $\cN^i$ with respect $\rn{\theta}^i$ is
$I_i(K)$.

By Lemma~\ref{lem:weakindep}, it is sufficient to show that
  if $I_i(\rn{K}$) is independent from $(\rn{\theta^i}_A \mid A\in r(\NN_+,\ell))$ for $i\in[2]$, then
  $\rn{K}$ is independent from $((\rn{\theta^1}_A, \rn{\theta^2}_A) \mid A\in r(\NN_+,\ell))$.
This immediately follows from the following easily verifiable general fact:
\begin{claim}
Let $\rn{X_1},\rn{X_2},\rn{Y_1},\rn{Y_2}$ be mutually independent random
variables, and let $f_1(X_1,Y_1)$, $f_2(X_2,Y_2)$ be functions such that
$f_i(\rn{X_i}, \rn{Y_i})$ is independent from $\rn{X_i}$ ($i=1,2$). Then
$(f_1(\rn{X_1},\rn{Y_1}),f_2(\rn{X_2},\rn{Y_2}))$ is independent from $(\rn{X_1}, \rn{X_2})$.
\end{claim}

In our context, we set $\rn{X_i} =  (\rn{\theta^i}_A \mid \lvert A\rvert\leq\ell)$,
$\rn{Y_i} =  (\rn{\theta^i}_A \mid \lvert A\rvert > \ell)$ and let $f_i$ compute the array $I_i(K)$ from $(X_i,Y_i)$ (thus $(f_1(X_1,Y_1),f_2(X_2,Y_2))$ computes the array $K$ from $(X_1,X_2,Y_1,Y_2)$).
\end{proofof}

The next lemma says that unique coupleability satisfies a ``chain rule''
analogous to the chain rule for mutual independence of random variables.

\begin{lemma}\label{lem:chaincoupling}
  Let $\phi_i\in\HomT{T_i}$ for $i\in[t]$ and suppose that for every $i\in[t-1]$, $\phi_{i+1}$ is uniquely coupleable with $\phi_1\otimes\cdots\otimes\phi_i$. Then $\phi_1,\ldots,\phi_t$ are uniquely coupleable.
\end{lemma}

\begin{proof}
  The proof is by induction in $t$. The result for $t=1$ is trivial. For $t\geq 2$, let
  $\xi\in\HomT{\bigcup_{i=1}^t T_i}$ be a coupling of $\phi_1,\ldots,\phi_t$ and let
  $I\interpret{\bigcup_{i=1}^{t-1} T_i}{\bigcup_{i=1}^t T_i}$ be the structure-erasing interpretation. Since
  $\xi^I$ is a coupling of $\phi_1,\ldots,\phi_{t-1}$, by inductive hypothesis we must have $\xi^I =
  \phi_1\otimes\cdots\otimes\phi_{t-1}$ so $\xi$ is also a coupling of $\phi_1\otimes\cdots\otimes\phi_{t-1}$
  and $\phi_t$, hence we must have $\xi=\phi_1\otimes\cdots\otimes\phi_t$.
\end{proof}

We finish this section with the (almost trivial) implication~\ref{thm:UCouple:weakindep}$\implies$\ref{thm:UCouple:local} of
Theorem~\ref{thm:UCouple}.

\begin{lemma}[Theorem~\ref{thm:UCouple}\ref{thm:UCouple:weakindep}$\implies$\ref{thm:UCouple:local}]\label{lem:weakindep->local}
  Let $\ell\in\NN$. If $\phi$ is weakly $\ell$-independent, then $\phi$ is $\ell$-local.
\end{lemma}

\begin{proof}
  Let $\rn{K}$ be the exchangeable array corresponding to some theon $\cN$ with respect to $\rn{\theta}$ picked in
  $\cE_{\NN_+}(\Omega)$ according to $\mu$ such that $\phi = \phi_\cN$ and suppose $\rn{K}$ is independent from
  $(\rn{\theta}_A\mid A\in r(\NN_+,\ell))$. Since for $V\in r(\NN_+)$ the marginal $\rn{K}\rest_V$ depends
  only on $(\rn{\theta}_A\mid A\in r(V))$, the marginals $(\rn{K}\rest_{V_i}\mid i\in I)$ are
  mutually independent as long as the sets $V_i$ have pairwise intersections of size at most
  $\ell$. This follows from the following general observation.

  \begin{claim} \label{clm:multiple}
  Let $\rn X,\rn{Y_1},\ldots,\rn{Y_n}$ be mutually independent random
  variables and $f_i(X,Y_i)$ be functions such that $(f_1(\rn X,
  \rn{Y_1}),\ldots,f_n(\rn X, \rn{Y_n}))$ is independent of $\rn X$. Then $f_1(\rn X,
  \rn{Y_1}),\ldots,f_n(\rn X, \rn{Y_n})$ are mutually independent.
  \end{claim}

  In our situation, $\rn X= (\rn{\theta}_A\mid A\in r(\NN_+,\ell))$, $\rn{Y_i} =
  (\rn{\theta}_A\mid A\in r(V_i)\setminus r(\NN_+,\ell))$ and $f_i$ computes the marginal
  $K\rest_{V_i}$ from $(\theta_A \mid A \in r(V_i))$.

  This completes the proof that $\phi$ is $\ell$-local.
\end{proof}

\section{Naturality}
\label{sec:naturality}

The objective of this section is to show Theorem~\ref{thm:naturality}, that
is, to show that our quasirandomness properties are preserved under open
interpretations. For this, we need to do a bit of abstract nonsense.

Recall from~\cite[Sct.~2.2]{CR19} that the category \cat{Int} has pushouts
(otherwise known as amalgamated sums, fibred coproducts, etc.). More
concretely, for open interpretations $I_1\interpret{T}{T_1}$ and
$I_2\interpret{T}{T_2}$, a pushout of $(I_1,I_2)$ is given by the theory $T'$
obtained from $T_1\cup T_2$ by adding the axioms
\begin{align}\label{eq:pushoutaxiom}
  \forall\vec{x}, (I_1(P)(\vec{x})\equiv I_2(P)(\vec{x}))
\end{align}
for every $P$ in the language of $T$ and the open interpretations
$J_i\interpret{T_i}{T'}$ ($i\in[2]$) that act identically on the language of
$T_i$ so that
\begin{equation*}
  \begin{tikzcd}
    T\arrow[r, "I_1"]\arrow[d, "I_2"'] & T_1\arrow[d, "J_1"]\\
    T_2\arrow[r, "J_2"'] & T'
  \end{tikzcd}
\end{equation*}
is commutative and has the standard universality property.

The following theorem says that we can also amalgamate limit objects along
pushouts. Let us warn that unless the theory $T$ is trivial (in which case
a ``canonical'' amalgamation is provided by the independent coupling), we
are not aware of any natural, functorial construction here.

\begin{theorem}\label{thm:amalgamation}
  Let
  \begin{equation*}
    \begin{tikzcd}
      T\arrow[r, "I_1"]\arrow[d, "I_2"'] & T_1\arrow[d, "J_1"]\\
      T_2\arrow[r, "J_2"'] & T'
    \end{tikzcd}
  \end{equation*}
  be a pushout of \cat{Int} and let $\phi_1\in\HomT{T_1}$ and $\phi_2\in\HomT{T_2}$ be such that $\phi_1^{I_1}
  = \phi_2^{I_2}$. Then there exists $\psi\in\HomT{T'}$ such that
  $\psi^{J_1} = \phi_1$ and $\psi^{J_2}=\phi_2$.
\end{theorem}

\begin{proof}
  First we claim that it is enough to show the case when $T'$ is obtained from $T_1\cup T_2$ by adding the
  axioms~\eqref{eq:pushoutaxiom}. Indeed, if $\psi$ is constructed for such particular case, then we can get
  our desired element of $\HomT{T'}$ for a general pushout $T'$ as $\psi^I$ for the universal isomorphism $I$
  between the pushout theories.

  Let us prove then the particular case. Let $\cL$, $\cL_1$ and $\cL_2$ be the languages of $T$, $T_1$ and
  $T_2$, respectively. For $i\in[2]$, let $\cN^i$ be a $T_i$-on (over $[0,1]$) such that $\phi_i =
  \phi_{\cN^i}$. Since $\phi_{I_1(\cN^1)} = \phi_1^{I_1} = \phi_2^{I_2} = \phi_{I_2(\cN^2)}$, by
  Proposition~\ref{prop:theonalignment}, there exists a $T_1$-on $\cH^1$ over $[0,1]^2$ such that
  $I_1(\cH^1)_P = I_2(\cN^2)_P\times\cE_{k(P)}$ $\lambda$-a.e.\ for every $P\in\cL$.

  Define then the Euclidean structure $\cH$ on $\cL_1\disjcup\cL_2$ over $[0,1]^2$ by
  \begin{align*}
    \cH_P
    & \df
    \begin{dcases*}
      \cH^1_P, & if $P\in\cL_1$;\\
      \cN^2_P\times\cE_{k(P)}, & if $P\in\cL_2$.
    \end{dcases*}
  \end{align*}
  Let us show that $\cH$ is a (weak) $T'$-on. To show this, it is enough to show (see~\cite[Definition~3.5,
    Remark~5, Theorem~3.7]{CR19}, by reaxiomatizing $T,T_1,T_2$ to be substitutionally closed, $T'$ also becomes
  substitutionally closed) that $T(I_1(P),\cH) = T(I_2(P),\cH)$ $\lambda$-a.e.\ for every $P\in\cL$. But this
  follows from
  \begin{align*}
    T(I_1(P),\cH) & = T(I_1(P),\cH^1) = I_1(\cH^1)_P;\\
    T(I_2(P),\cH) & = T(I_2(P),\cN^2)\times\cE_{k(P)} = I_2(\cN^2)\times\cE_{k(P)}.
  \end{align*}
  Finally, since we trivially have $J_1(\cH) = \cH^1$ and $J_2(\cH)_P = \cN^2_P\times\cE_{k(P)}$ for every
  $P\in\cL_2$, it follows that $\psi\df\phi_\cH$ satisfies $\psi^{J_1} = \phi_1$ and $\psi^{J_2} = \phi_2$.
\end{proof}

The next proposition makes use of this amalgamation property to ``lift''
couplings through interpretations.

\begin{proposition}[Coupling lifting]\label{prop:lifting}
  Let $I\interpret{T_1}{T_2}$ be an open interpretation, let $T$ be a canonical theory and let
  $\phi\in\HomT{T}$ and $\phi_2\in\HomT{T_2}$. If $\xi$ is a coupling of $\phi_2^I$ and $\phi$, then there
  exists a coupling $\widehat{\xi}$ of $\phi_2$ and $\phi$ such that $\xi = \widehat{\xi}^{I\cup\id_T}$.
\end{proposition}

\begin{proof}
  This follows from Theorem~\ref{thm:amalgamation} and the fact that
  \begin{equation*}
    \begin{tikzcd}
      T_1\arrow[r, "I"]\arrow[d] & T_2\arrow[d]\\
      T_1\cup T\arrow[r, "I\cup\id_T"'] & T_2\cup T
    \end{tikzcd}
  \end{equation*}
  is a pushout in \cat{Int}, where the vertical arrows are the structure-erasing interpretations.
\end{proof}

Equipped with this ``lifting'' construction, we can prove
Theorem~\ref{thm:naturality} about naturality of our properties.

\begin{proofof}{Theorem~\ref{thm:naturality}}
  For item~\ref{thm:naturality:uniquecouplings}, let $I_i\interpret{T_i}{T\cup T_i}$ be the structure-erasing
  interpretation for $i\in[2]$ and note that if $\xi$ is a coupling of $\phi^I$ with
  $\psi$, then Proposition~\ref{prop:lifting} gives us a coupling $\widehat{\xi}$ of $\phi$
  with $\psi$ such that $\xi = \widehat{\xi}^{I\cup\id_T}$. Since $\phi$ is uniquely coupleable with $\psi$
  we must have $\widehat{\xi} = \phi\otimes\psi$, from which we get $\xi = \widehat{\xi}^{I\cup\id_T} =
  \phi^I\otimes\zeta$, hence $\phi^I$ and $\psi$ are uniquely coupleable.

  Item~\ref{thm:naturality:Independence} follows trivially from the fact that if $\cN$ is an
  $\ell$-independent $T_2$-on with $\phi = \phi_\cN$, then $I(\cN)$ is an $\ell$-independent $T_1$-on with
  $\phi_{I(\cN)} = \phi^I$.

  Item~\ref{thm:naturality:UCouple} follows trivially from item~\ref{thm:naturality:uniquecouplings}.

  For item~\ref{thm:naturality:UInduce}, we let $\psi\in\HomT{\TkHypergraph[\ell]}$ and $\xi$ be a coupling of
  $\phi^I$ with $\psi$ and we make the same construction of the coupling $\widehat{\xi}$ of $\phi$ and $\psi$
  of item~\ref{thm:naturality:uniquecouplings} using Proposition~\ref{prop:lifting}. For $i\in[2]$, let
  $J_i\interpret{\TkHypergraph[\ell]}{T_i\cup\TkHypergraph[\ell]}$ be the structure-erasing interpretation and
  note that if $M\in\cM[T_1\cup\TkHypergraph[\ell]]$ is such that $J_1(M)\cong K^{(\ell)}_{\lvert M\rvert}$,
  then we have
  \begin{align*}
    \xi(M)
    & =
    \widehat{\xi}^{I\cup\id_T}(M)
    \\
    & =
    \widehat{\xi}\left(\sum
    \set{%
      M'\in\cM_{\lvert M\rvert}[T_2\cup\TkHypergraph[\ell]]
    }{%
      I(I_2(M'))\cong I_1(M)\land
      J_2(M')\cong K^{(\ell)}_{\lvert M\rvert}
    }
    \right)
    \\
    & =
    \psi(K^{(\ell)}_{\lvert M\rvert})\cdot
    \phi\left(\sum\set{M'\in\cM_{\lvert M\rvert}[T_2]}{I(M')\cong I_1(M)}\right)
    \\
    & =
    \psi(K^{(\ell)}_{\lvert M\rvert})\cdot\phi^I(I_1(M))
    \\
    & =
    (\phi^I\otimes\psi)(M),
  \end{align*}
  where the third equality follows from the fact that $\phi\in\UInduce[\ell]$. Hence
  $\phi^I\in\UInduce[\ell]$.
\end{proofof}

\section{Unique inducibility}
\label{sec:UInduce}

In this section we prove Theorem~\ref{thm:UInduce}. We start by showing the
equivalence between items~\ref{thm:UInduce:UInduce}
and~\ref{thm:UInduce:singledensity}. Curiously, the case $\ell=1$ is the
hardest one to prove.

\begin{lemma}[Theorem~\ref{thm:UInduce}\ref{thm:UInduce:UInduce}$\equiv$\ref{thm:UInduce:singledensity}]\label{lem:UInduce:singledensity}
  Let $\ell\in\NN_+$ and $\phi\in\HomT{T}$. Then $\phi\in\UInduce[\ell]$ if and only if there exists
  $p\in(0,1)$ such that $\phi$ is uniquely inducible by every $\psi\in\HomT{\TkHypergraph[\ell]}$ with
  $\psi(\rho_\ell) = p$.
\end{lemma}

\begin{proof}
  The forward implication is obvious.

  For $p\in(0,1)$, let us say that $\phi$ is uniquely $p$-inducible if it is uniquely inducible by every
  $\psi\in\HomT{\TkHypergraph[\ell]}$ with $\psi(\rho_\ell) = p$. Then the backward implication amounts to
  showing that unique $p$-inducibility implies unique $q$-inducibility for every $p,q\in(0,1)$ (the cases
  $q\in\{0,1\}$ are trivial).

  Let $I\interpret{T}{T\cup\TkHypergraph[\ell]}$ and
  $J\interpret{\TkHypergraph[\ell]}{T\cup\TkHypergraph[\ell]}$ be the structure-erasing interpretations. Let
  us assume that $\phi$ is uniquely $p$-inducible and let us show that $\phi$ is uniquely inducible by any
  $\psi\in\HomT{\TkHypergraph[\ell]}$ with $\psi(\rho_\ell)=q$. Let $\xi$ be a coupling of $\phi$ and $\psi$.

  Our objective is to prove that for every $m\in\NN$ and every $M\in\cM_m[T\cup\TkHypergraph[\ell]]$ with
  $J(M)\cong K^{(\ell)}_m$ we have
  \begin{align}\label{eq:UInduce:objective}
    \xi(M) & = \phi(I(M))\psi(K^{(\ell)}_m).
  \end{align}
  For $m < \ell$ this is trivial (as $\psi(K^{(\ell)}_m) = 1$), so
  suppose $m\geq \ell$.

  Let $I'\interpret{\TkHypergraph[\ell]}{\TkHypergraph[\ell]\cup\TcColoring[2]}$ be an open interpretation
  (to be specified later);
  note that the diagram
  \begin{equation}\label{eq:UInduce:coloringcommutative}
    \begin{tikzcd}[column sep={0cm}]
      \TkHypergraph[\ell]
      \arrow[dr, "J"]
      \arrow[dddrr, "I'"', bend right]
      &
      &[-1.5cm]
      T
      \arrow[dl, "I"']
      \arrow[dr, "I"]
      &[-1.5cm]
      &
      \TkHypergraph[\ell]
      \arrow[dl, "J"']
      \arrow[dddll, bend left]
      \\
      &
      T\cup\TkHypergraph[\ell]
      \arrow[dr, "\id_T\cup I'"']
      &
      &
      T\cup\TkHypergraph[\ell]
      \arrow[dl]
      \\
      &
      &
      T\cup\TkHypergraph[\ell]\cup\TcColoring[2]
      \\
      &
      &
      \TkHypergraph[\ell]\cup\TcColoring[2]
      \arrow[u]
    \end{tikzcd}
  \end{equation}
  is commutative, where the unlabeled arrows are structure-erasing interpretations. For $t\in[0,1]$ let
  $\widehat{\xi}_t\df \xi\otimes\psi_{(t,1-t)}$ be the independent coupling of $\xi$ and the $2$-coloring
  $\psi_{(t,1-t)}$ of densities $(t,1-t)$ (see Definition~\ref{def:colorings}); note that the fact
  that~\eqref{eq:UInduce:coloringcommutative} is commutative implies that
  $\widehat{\xi}_t^{\id_T\cup I'}$ is a coupling of $\phi$ and $(\psi\otimes\psi_{(t,1-t)})^{I'}$.

  We start by showing~\eqref{eq:UInduce:objective} in the case $p\leq q$. In this case, we take
  \begin{align*}
    I'(E)(x_1,\ldots,x_\ell) & \df E(x_1,\ldots,x_\ell)\land\bigwedge_{i\in[\ell]}\chi_1(x_i),
  \end{align*}
  that is, $I'$ keeps edges that are monochromatic in color $1$. Let $t\df (p/q)^{1/\ell}$ and note that for
  $n\geq\ell$ we have
  \begin{align*}
    (\psi\otimes\psi_{(t,1-t)})^{I'}(K^{(\ell)}_n)
    & =
    \psi(K^{(\ell)}_n) t^n
    =
    \psi(K^{(\ell)}_n)\left(\frac{p}{q}\right)^{n/\ell},
  \end{align*}
  which in particular implies that $(\psi\otimes\psi_{(t,1-t)})^{I'}(\rho_\ell) = p$. On the other hand, we
  also have $\widehat{\xi}_t^{\id_T\cup I'}(M) = \xi(M) t^m$, so unique $p$-inducibility of $\phi$ gives
  \begin{align*}
    \xi(M) t^m
    =
    \widehat{\xi}_t^{\id_T\cup I'}(M)
    =
    \phi(I(M))(\psi\otimes\psi_{(t,1-t)})^{I'}(K^{(\ell)}_m)
    =
    \phi(I(M))\psi(K^{(\ell)}_m) t^m,
  \end{align*}
  from which~\eqref{eq:UInduce:objective} follows.

  \medskip

  We now show~\eqref{eq:UInduce:objective} in the case $\ell\geq 2$ and $q < p$. In this case, we let
  \begin{align*}
    I'(E)(x_1,\ldots,x_\ell)
    & \df
    \left(E(x_1,\ldots,x_\ell)\land\bigwedge_{i\in[\ell]}\chi_1(x_i)\right)\lor\bigwedge_{i\in[\ell]}\chi_2(x_i)
  \end{align*}
  that is, $I'$ declares edges to be either old edges that are monochromatic in color $1$ or any $\ell$-set
  that is monochromatic in color $2$. Let $f(x) \df x^\ell q + (1-x)^\ell$ and note that $f(0) = 1$ and $f(1)
  = q$, so there exists $t\in(0,1)$ such that $f(t) = p$. Since $\ell\geq 2$, for $n\geq\ell$, we have
  \begin{align*}
    (\psi\otimes\psi_{(t,1-t)})^{I'}(K^{(\ell)}_n)
    & =
    \psi(K^{(\ell)}_n) t^n + (1-t)^n,
  \end{align*}
  which in particular implies that $(\psi\otimes\psi_{(t,1-t)})^{I'}(\rho_\ell) = f(t) = p$. On the other
  hand, we also have $\widehat{\xi}_t^{\id_T\cup I'}(M) = \xi(M) t^m + \phi(I(M))(1-t)^m$, so unique
  $p$-inducibility of $\phi$ gives
  \begin{align*}
    \xi(M) t^m + \phi(I(M))(1-t)^m
    & =
    \widehat{\xi}_t^{\id_T\cup I'}(M)
    =
    \phi(I(M))(\psi\otimes\psi_{(t,1-t)})^{I'}(K^{(\ell)}_m)
    \\
    & =
    \phi(I(M))(\psi(K^{(\ell)}_m) t^m + (1-t)^m),
  \end{align*}
  from which~\eqref{eq:UInduce:objective} follows.

  \medskip

  The case $q < p$ and $\ell=1$ is more complicated as the construction analogous to the above
  does not work: cliques in arity $1$ need not be monochromatic.

  Let us prove first the sub-case $q = p^2$. The idea, roughly speaking, is that when $\ell=1$, unique
  $p$-inducibility says that any ``subset of vertices'' of relative size $p$ in $\phi$ induces $\phi$ and since a
  ``subset of vertices'' of relative size $p^2$ can be seen as having relative size $p$ in some ``subset of
  vertices'' that itself has relative size $p$ in the whole space, it must also induce
  $\phi$.

  It is worth noting that this idea can be implemented almost literally in the
  geometric language. But that would require working with theons that have
  different ground sets in different coordinates so we prefer to present a
  syntactic argument instead, similar to the one above.

  We work with the theory $\TcColoring[2]$ instead of $\TkHypergraph[1]$  (see Remark~\ref{rmk:UInduce1}).
  Let $\xi$ be a coupling of
  $\phi$ and $\psi\df\psi_{(p^2,1-p^2)}\in\HomT{\TcColoring[2]}$ and we want to show that for every
  $M\in\cM[T\cup\TcColoring[2]]$ with $R_{\chi_1}(M) = V(M)$, we have
  \begin{align*}
    \xi(M) & = \phi(I(M))p^{2m},
  \end{align*}
  where $m\df\lvert M\rvert$ and $I\interpret{T}{T\cup\TcColoring[2]}$ is the structure-erasing interpretation.

  Let $I_1,I_2\interpret{\TcColoring[2]}{\TcColoring[3]}$ be the interpretations given by
  \begin{align*}
    \begin{aligned}
      I_1(\chi_1)(x) & \df \chi_1(x)\lor\chi_2(x);\\
      I_1(\chi_2)(x) & \df \chi_3(x);
    \end{aligned}
    & &
    \begin{aligned}
      I_2(\chi_1)(x) & \df \chi_1(x);\\
      I_2(\chi_2)(x) & \df \chi_2(x)\lor\chi_3(x).
    \end{aligned}
  \end{align*}
  Let $\widehat{\psi}\df\psi_{(p^2, p-p^2, 1-p)}\in\HomT{\TcColoring[3]}$ and note that $\widehat{\psi}^{I_i} =
  \psi_{(p^i,1-p^i)}$ for $i\in[2]$.

  Let $J\interpret{\TcColoring[2]}{T\cup\TcColoring[2]}$ and
  $\widehat{J}\interpret{\TcColoring[3]}{T\cup\TcColoring[3]}$ be the structure-erasing interpretations. Our
  definitions ensure that the following diagram is commutative.
  \begin{equation}\label{eq:UInduce:p2commutative}
    \begin{tikzcd}[column sep={0cm}]
      \TcColoring[2]
      \arrow[dr, "J"]
      \arrow[dddrr, "I_1"', bend right]
      &
      &
      T
      \arrow[dl, "I"']
      \arrow[dr, "I"]
      &
      &
      \TcColoring[2]
      \arrow[dl, "J"']
      \arrow[dddll, "I_2", bend left]
      \\
      &
      T\cup\TcColoring[2]
      \arrow[dr, "\id_T\cup I_1"']
      &
      &
      T\cup\TcColoring[2]
      \arrow[dl, "\id_T\cup I_2"]
      \\
      &
      &
      T\cup\TcColoring[3]
      \\
      &
      &
      \TcColoring[3]
      \arrow[u, "\widehat{J}"]
    \end{tikzcd}
  \end{equation}

  For every $n\in\NN$, let $C_n\in\cM_n[\TcColoring[2]]$ be the unique model with all vertices satisfying
  $\chi_1$.

  Since $\widehat{\psi}^{I_2} = \psi$, by Proposition~\ref{prop:lifting}, there exists a coupling
  $\widehat{\xi}$ of $\phi$ and $\widehat{\psi}$ such that $\widehat{\xi}^{\id_T\cup I_2} = \xi$. We now make
  use of the operator $\pi^{(\neg\chi_3,\id_T\cup
    I_2)}\function{\cA[T\cup\TcColoring[2]]}{\cA_u[T\cup\TcColoring[3]]}$ \cite[Definition~4]{Raz07}, where
    $u = \sum\{N\in\cM_1[T\cup\TcColoring[3]] \mid
  I_1(\widehat{J}(N))\cong C_1\}$ and $\cA_u[T\cup\TcColoring[3]]$ is the localization by the
  multiplicative system $\{u,u^2,\ldots,u^n,\ldots\}$. Intuitively, it corresponds to applying
  the
  interpretation $\id_T\cup
    I_2$, followed by throwing away vertices of color 3. (All densities have to be re-normalized by a
    power of $u$, this is why we need to localize.) Since
  \begin{align}\label{eq:UInduce:u}
    \widehat{\xi}(u)
    & =
    \widehat{\xi}^{\widehat{J}\comp I_1}(C_1)
    =
    \widehat{\psi}^{I_1}(C_1)
    =
    p
    >
    0,
  \end{align}
  we can apply~\cite[Theorem~2.6]{Raz07} and form the element
  $\zeta\df\widehat{\xi}\comp\pi^{(\neg\chi_3,\id_T\cup I_2)}\in\HomT{T\cup\TcColoring[2]}$.
  We claim that $\zeta^I = \phi$.

  To see this, note that for $N\in\cM[T]$, we have
  \begin{align*}
    \zeta^I(N)
    & =
    \frac{\sum_{N'}\widehat{\xi}(N')}{\widehat{\xi}(u)^{\lvert N\rvert}},
  \end{align*}
  where the sum is over all $N'\in\cM_{\lvert N\rvert}[T\cup\TcColoring[3]]$ such that $I((\id_T\cup
  I_2)(N'))\cong N$ and $J((\id_T\cup I_1)(N'))\cong C_{\lvert N\rvert}$. But
  since~\eqref{eq:UInduce:p2commutative} is commutative, the condition $I((\id_T\cup I_2)(N'))\cong N$ is
  equivalent to $I((\id_T\cup I_1)(N'))\cong N$, which together with~\eqref{eq:UInduce:u} gives
  \begin{align*}
    \zeta^I(N) & = \frac{\widehat{\xi}^{\id_T\cup I_1}(\widehat{N})}{p^{\lvert N\rvert}},
  \end{align*}
  where $\widehat{N}\in\cM_{\lvert N\rvert}[T\cup\TcColoring[2]]$ is the unique model such that
  $I(\widehat{N})\cong N$ and $J(\widehat{N})\cong C_{\lvert N\rvert}$.

  Since $\widehat{\xi}^{(\id_T\cup I_1)\comp J}(C_1) = \widehat{\psi}^{I_1}(C_1) = p$ and
  $\widehat{\xi}^{(\id_T\cup I_1)\comp I} = \xi^I = \phi$, unique $p$-inducibility of $\phi$ implies that
  $\widehat{\xi}^{\id_T\cup I_1}(\widehat{N}) = p^{\lvert N\rvert}\phi(N)$ and thus $\zeta^I = \phi$.

  Now we claim that $\zeta^J = \psi_{(p,1-p)}$. Indeed, note that
  \begin{align*}
    \zeta^J(C_1)
    & =
    \frac{%
      \sum\{\widehat{\xi}(N) \mid N\in\cM_1[T\cup\TcColoring[3]]\land
      J((\id_T\cup I_2)(N))\cong J((\id_T\cup I_1)(N))\cong C_1\}
    }{%
      \widehat{\xi}(u)%
    }
    \\
    & =
    \frac{\widehat{\xi}^{\widehat{J}}(\widehat{C}_1)}{p}
    =
    \frac{\widehat{\psi}(\widehat{C}_1)}{p}
    =
    p,
  \end{align*}
  where $\widehat{C}_1\in\cM_1[\TcColoring[3]]$ is the model whose unique vertex satisfies $\chi_1$, hence
  $\zeta^J = \psi_{(p,1-p)}$.

  This means that $\zeta$ is a coupling of $\phi$ and $\psi_{(p,1-p)}$, so for our fixed
  $M\in\cM_m[T\cup\TcColoring[2]]$ with $R_{\chi_1}(M) = V(M)$, unique $p$-inducibility of $\phi$ gives
  \begin{align*}
    \xi(M)
    & =
    \widehat{\xi}^{\id_T\cup I_2}(M)
    =
    \widehat{\xi}(\pi^{(\neg\chi_3,\id_T\cup I_2)}(M))\cdot\widehat{\xi}(u)^m
    \\
    & =
    \zeta(M)\cdot p^m
    =
    \phi(I(M))\cdot p^{2m},
  \end{align*}
  as desired.

  \medskip

  From the case $\ell = 1$ and $q = p^2 < p$, with a simple induction, we can derive the case when $\ell=1$
  and $q = p^{2^k} < p$ for some $k\in\NN_+$.

  \medskip

  Finally, for the case $\ell=1$ and arbitrary $q < p$, we let $k\in\NN_+$ be large enough so that $p^{2^k} <
  q$ and putting together the previous cases gives that unique $p$-inducibility implies unique
  $p^{2^k}$-inducibility, which in turn implies unique $q$-inducibility.
\end{proof}

The rest of this section is devoted to various relations between the unique
inducibility and the clique discrepancy for hypergraphons; we will also use
our findings to prove the last remaining equivalence~\ref{thm:UInduce:UInduce}$\equiv$\ref{thm:UInduce:symlocal} in
Theorem~\ref{thm:UInduce}.

It was proved in~\cite{Tow17,ACHP18} that for $\ell < k$, $\CliqueDisc[\ell]$ is equivalent to the non-induced
labeled density of every $\ell$-linear hypergraph $H$ (i.e., hypergraphs whose edges have pairwise
intersections of size at most $\ell$) being $p^{e(H)}$. We restate below this result in the language of
exchangeable arrays.

\begin{theorem}[\cite{Tow17,ACHP18}]\label{thm:CliqueDisc}
  Let $\ell\in[k-1]$, let $\phi\in\HomT{\TkHypergraph}$ and let $\rn{K}$ be the corresponding exchangeable
  array. Then $\phi\in\CliqueDisc[\ell]$ if and only if for every finite collection $(V_i)_{i\in I}$ of finite
  subsets of $\NN_+$ of size $k$ each and with pairwise intersections of size at most $\ell$ we have
  \begin{align*}
    \PP[\forall i\in I, \rn{K}\rest_{V_i}\cong \rho_k] & = \prod_{i\in I}\PP[\rn{K}\rest_{V_i}\cong\rho_k].
  \end{align*}
\end{theorem}

Even though this theorem only makes sense in the theory of hypergraphs, we
can derive the
implication~\ref{thm:UInduce:symlocal}$\implies$\ref{thm:UInduce:UInduce} of
Theorem~\ref{thm:UInduce} for general theories from it.

\begin{lemma}[Theorem~\ref{thm:UInduce}\ref{thm:UInduce:symlocal}$\implies$\ref{thm:UInduce:UInduce}]\label{lem:symlocal->UInduce}
  If $\phi\in\HomT{T}$ is symmetrically $\ell$-local, then $\phi\in\UInduce[\ell]$.
\end{lemma}

\begin{proof}
  Let $I\interpret{T}{T\cup\TkHypergraph[\ell]}$ and
  $J\interpret{\TkHypergraph[\ell]}{T\cup\TkHypergraph[\ell]}$ be the structure-erasing interpretations.

  Our objective is to show that for every $\psi\in\HomT{\TkHypergraph[\ell]}$, every coupling $\xi$ of $\phi$
  and $\psi$, every $m\in\NN$ and every $M\in\cM_m[T\cup\TkHypergraph[\ell]]$ with $J(M)\cong K^{(\ell)}_m$,
  we have
  \begin{align}\label{eq:symlocal->UInduce:objective}
    \xi(M) & = \phi(I(M))\psi(K^{(\ell)}_m).
  \end{align}

  Let us first consider the case $m\leq\ell$. In this case, note that for the exchangeable array $\rn{K}$
  corresponding to $\phi$, by letting $V_1=V_2=[m]$, symmetric
  $\ell$-locality of $\phi$ gives
  \begin{align*}
    \phi(I(M)) & = \PP[\rn{K}\rest_{[m]}\cong I(M)] = \PP[\rn{K}\rest_{[m]}\cong I(M)]^2 = \phi(I(M))^2,
  \end{align*}
  so $\phi(I(M))\in\{0,1\}$, hence~\eqref{eq:symlocal->UInduce:objective} follows.

  Suppose now that $m>\ell$ and let $I'\interpret{\TkHypergraph[m]}{T}$ be the open interpretation that
  declares $m$-edges to be \emph{isomorphic} copies of $I(M)$, that is, it is given by
  \begin{align*}
    I'(E)(x_1,\ldots,x_m) & \df \bigvee_{\sigma\in S_m} \Dopen(I(M))(x_{\sigma(1)},\ldots,x_{\sigma(m)}).
  \end{align*}

  Let us show that $\phi^{I'}\in\HomT{\TkHypergraph[m]}$ satisfies
  $\CliqueDisc[\ell]$. Let $\rn{K}$ be the exchangeable array corresponding to $\phi$ so that $I'(\rn{K})$ is
  the exchangeable array corresponding to $\phi^{I'}$. Then if $(V_i)_{i\in[t]}$ is a finite collection of
  finite subsets of $\NN_+$ of size $m$ each and with pairwise intersections of size at most $\ell$, then
  \begin{align*}
    \PP[\forall i\in[t], I'(\rn{K})\rest_{V_i}\cong \rho_m]
    & =
    \PP[\forall i\in[t], \rn{K}\rest_{V_i}\cong M]
    \\
    & =
    \prod_{i\in[t]}\PP[\rn{K}\rest_{V_i}\cong M]
    =
    \prod_{i\in[t]}\PP[I'(\rn{K})\rest_{V_i}\cong\rho_m],
  \end{align*}
  where the second equality follows from the fact that $\phi$ is symmetrically $\ell$-local. By
  Theorem~\ref{thm:CliqueDisc}, it follows that $\phi^{I'}$ satisfies $\CliqueDisc[\ell]$.

  Note now that the diagram
  \begin{equation*}
    \begin{tikzcd}
      \TkHypergraph[m]\arrow[r]\arrow[d, "I'"'] &
      \TkHypergraph[m]\cup\TkHypergraph[\ell]\arrow[d, "I'\cup\id_{\TkHypergraph[\ell]}"'] &
      \TkHypergraph[\ell]\arrow[l]\arrow[dl, "J"]
      \\
      T\arrow[r, "I"'] &
      T\cup\TkHypergraph[\ell]
    \end{tikzcd}
  \end{equation*}
  is commutative, where the unlabeled arrows are structure-erasing interpretations. This implies that
  $\xi^{I'\cup\id_{\TkHypergraph[\ell]}}$ is a coupling of $\phi^{I'}$ and $\psi$, so we get
  \begin{align*}
    \xi(M)
    & =
    \xi^{I'\cup\id_{\TkHypergraph[\ell]}}(K^{(m,\ell)}_m)
    =
    \phi^{I'}(\rho_m)\psi(K^{(\ell)}_m)
    =
    \phi(I(M))\psi(K^{(\ell)}_m),
  \end{align*}
  where the second equality follows from $\phi^{I'}\in\CliqueDisc[\ell]$.
\end{proof}

Let us now prove an important fact about $\CliqueDisc[\ell]$ and
$\ell$-flattenings defined below.

\begin{definition}\label{def:flattening}
  For a peon $\cN$ over $\Omega=(X,\cA,\mu)$ and $\ell\in\NN$, the \emph{$\ell$-flattening} of $\cN$ is the
  function $W^\ell_\cN\function{\cE_{k,\ell}(\Omega)}{[0,1]}$ defined by
  \begin{align*}
    W^\ell_\cN(x) & \df \mu(\{y\in X^{\binom{[k]}{>\ell}} \mid (x,y)\in\cN\}),
  \end{align*}
  and defined arbitrarily when the set above is not measurable.
\end{definition}

Note that the construction in~\eqref{eq:rankae:W} is precisely an
$\ell$-flattening, and so is the construction of a graphon in the ordinary
sense from $\TGraph$-on (cf.~\eqref{eq:W}, \eqref{eq:Wskew} and~\eqref{eq:theonsVSgraphons}).

\begin{lemma}\label{lem:CliqueDiscflatteningconstant}
  Let $\cN$ be a $\TkHypergraph$-on over $\Omega = (X,\cA,\mu)$ such that $\phi_\cN$ satisfies $\CliqueDisc[\ell]$. Then
  $W_\cN^\ell = \phi_\cN(\rho_k)$ a.e.
\end{lemma}

\begin{proof}
  It is sufficient to prove that the two
  measures on $X^{r(k,\ell)}$ given by $Y\mapsto \int_Y W_\cN^\ell\ d\mu$ and $\nu(Y) \df
  \phi_\cN(\rho_k)\mu(Y)$ coincide, and for that we only have to consider the basis of our $\sigma$-algebra,
  i.e., sets of the form
  \begin{align*}
    Y = \prod_{A\in r(k,\ell)}V_A.
  \end{align*}
  In other words, for every collection $V_A\subseteq X$ ($A\in r(k,\ell)$) of measurable sets we have to prove
  that
  \begin{align}\label{eq:CliqueDiscintegral}
    \int_Y W^\ell_\cN\ d\mu & = \phi_\cN(\rho_k)\cdot \mu(Y).
  \end{align}

  Recall from~\cite{Tow17,ACHP18} that $\CliqueDisc[\ell]$ is equivalent to $\Disc[\binom{[k]}{\ell}]$ (see
  Definition~\ref{def:CliqueDisc}) and for the language $\cL_{\binom{[k]}{\ell}}$ containing one predicate
  symbol $P_A$ of arity $\ell$ for each $A\in\binom{[k]}{\ell}$, define the
  $T_{\cL_{\binom{[k]}{\ell}}}\cup\TkHypergraph$-on $\cH$ over $\Omega$ by
  \begin{align*}
    \cH_E & \df \cN_E; &
    \cH_{P_A} & \df \iota_A^\ast(Y)
    = \{x\in\cE_\ell(\Omega) \mid \forall A'\in r(A), x_{\iota_A^{-1}(A')}\in V_{A'}\}.
  \end{align*}
  Let then $\rn{K}$ be the exchangeable array corresponding to $\cH$. Since $\phi_\cN$ satisfies
  $\CliqueDisc[\ell]=\Disc[\binom{[k]}{\ell}]$, we get
  \begin{align*}
    \int_Y W^\ell_\cN\ d\mu
    & =
    \PP\left[(1,\ldots,k)\in R_E(\rn{K})\land \forall A\in\binom{[k]}{\ell}, \iota_A\in R_{P_A}(\rn{K})\right]
    \\
    & =
    \phi_\cN(\rho_k)\cdot\PP\left[\forall A\in\binom{[k]}{\ell}, \iota_A\in R_{P_A}(\rn{K})\right]
    \\
    & =
    \phi_\cN(\rho_k)\cdot\mu(Y),
  \end{align*}
  as desired.
\end{proof}

To prove the final
implication~\ref{thm:UInduce:UInduce}$\implies$\ref{thm:UInduce:symlocal} in
Theorem~\ref{thm:UInduce}, we will need a small generalization of the easier
direction of Theorem~\ref{thm:CliqueDisc} for disjoint unions of theories of
hypergraphs.

\begin{definition}[$\vec{k}$-hypergraphs]
  Given $\vec{k} = (k_1,\ldots,k_t)\in\NN_+^t$, we let
  $\TkHypergraph[\vec{k}]\df\bigcup_{i\in[t]}\TkHypergraph[k_i]$ and in this theory, we denote the predicate
  symbol corresponding to the $i$-th hypergraph by $E_i$. Models of $\TkHypergraph[\vec{k}]$ will be called
  \emph{$\vec{k}$-hypergraphs} and for one such model $M$, we let $E_i(M)\df\{\im(\alpha) \mid \alpha\in
  R_{E_i}(M)\}$ be its \emph{$i$-th edge set}.  We also denote by
  $I_i\interpret{\TkHypergraph[k_i]}{\TkHypergraph[\vec{k}]}$ the structure-erasing interpretation
  corresponding to the $i$-th edge set.
\end{definition}

\begin{proposition}\label{prop:vecCliqueDisc}
  Let $\vec{k} = (k_1,\ldots,k_t)$, let $\ell\leq\min_{i\in[t]} k_i$, let $i_1,\ldots,i_s\in[t]$ and let
  $(V_j)_{j=1}^s$ be such that $V_j\in\binom{\NN_+}{k_{i_j}}$ and $\lvert V_j\cap V_{j'}\rvert\leq\ell$,
  whenever $j\neq j'$.

  Let $\phi\in\HomT{\TkHypergraph[\vec{k}]}$ be such that all $\phi^{I_i}$ ($i\in [t]$) satisfy $\CliqueDisc[\ell]$ and let $\rn{K}$ be the corresponding
  exchangeable array. Then
  \begin{align*}
    \PP[\forall j\in[s], V_j\in E_{i_j}(\rn{K})] & = \prod_{j\in[s]}\PP[V_j\in E_{i_j}(\rn{K})].
  \end{align*}
\end{proposition}

\begin{proof}
  Let $\cN$ be a $\TkHypergraph$-on such that $\phi_\cN=\phi$ and note that
  \begin{align*}
    \PP[\forall j\in[s], V_j\in E_{i_j}(\rn{K})]
    & =
    \lambda\left(\bigcap_{j\in[s]}(\alpha_j^*)^{-1}(\cN_{E_{i_j}})\right)
    \\
    & =
    \lambda(\{x\in\cE_{\NN_+} \mid \forall j\in[s], \alpha_j^*(x)\in\cN_{E_{i_j}}\}),
  \end{align*}
  where $\alpha_j\in(\NN_+)_{k_{i_j}}$ is such that $\im(\alpha_j) = V_j$. Since the sets $V_j$ have pairwise
  intersections of size at most $\ell$, in the set above, the coordinates $x_A$ with $\lvert A\rvert > \ell$
  are only constrained by at most one of the $\alpha_j^*$, so Fubini's Theorem gives
  \begin{align*}
    \PP[\forall j\in[s], V_j\in E_{i_j}(\rn{K})]
    & =
    \int_{\cE_{V,\ell}} \prod_{j\in[s]} W_{\cN_{E_{i_j}}}^\ell(\alpha_j^*(x))\ d\lambda(x),
  \end{align*}
  where $V\df\bigcup_{j\in[s]} V_j$.

  Since each $\phi^{I_i}$ satisfies $\CliqueDisc[\ell]$, by
  Lemma~\ref{lem:CliqueDiscflatteningconstant}, it follows that $W_{\cN_{E_i}}^\ell =
  \phi^{I_i}(\rho_{k_i})$ a.e., so we get
  \begin{align*}
    \PP[\forall j\in[s], V_j\in E_{i_j}(\rn{K})]
    & =
    \prod_{j\in[s]} \phi^{I_{i_j}}(\rho_{k_{i_j}})
    =
    \prod_{j\in[s]}\PP[V_j\in E_{i_j}(\rn{K})],
  \end{align*}
  as desired.
\end{proof}

Proposition~\ref{prop:vecCliqueDisc} (and Theorem~\ref{thm:UInduce->CliqueDisc})
will be sufficient to handle the case in the definition of
symmetric $\ell$-locality when all sets have size at least $\ell$. For smaller sets, we need
the notion of categoricity of elements of $\HomT{T}$ defined below.

\begin{definition}\label{def:categorical}
  For $\phi\in\HomT{T}$, let $\Th(\phi)$ be the theory obtained from $T$ by adding the axiom
  $\forall\vec{x},\neg\Dopen(M)(\vec{x})$ for every $M\in\cM[T]$ such that $\phi(M) = 0$, i.e., it is the
  theory whose models are precisely the ones that have positive density in $\phi$.

  Recall that in model theory a theory $T$ is called \emph{$\ell$-categorical} if it has exactly one model of size $\ell$ up
  to isomorphism. We say that $\phi\in\HomT{T}$ is \emph{$\ell$-categorical} if $\Th(\phi)$ is
$\ell$-categorical.
\end{definition}

\begin{remark}\label{rmk:categorical}
  Since $\sum_{M\in\cM_\ell[T]} \phi(M) = 1$, it follows that $\phi$ is $\ell$-categorical if and only if
  $\phi(M)\in\{0,1\}$ for every $M\in\cM_\ell[T]$.
\end{remark}

\begin{lemma}\label{lem:categoricalnat}
  Let $I\interpret{T_1}{T_2}$ be an open interpretation and let $\phi\in\HomT{T_2}$ be
  $\ell$-categorical. Then $\phi^I$ is $\ell$-categorical.
\end{lemma}

\begin{proof}
  Since for $M\in\cM_\ell[T_1]$, we have $\phi^I(M) = \sum\{\phi(N) \mid N\in\cM_\ell[T_2]\land I(N)\cong
  M\}$, it follows that $\phi^I(M) > 0$ if and only if $M\cong I(N_0)$ for the unique model
  $N_0\in\cM_\ell[\Th(\phi)]$.
\end{proof}

\begin{lemma}\label{lem:ramsey}
If $\phi\in \HomT{\TkHypergraph}$ is $\ell$-categorical for $\ell\geq k$ then
$\phi(\rho_k)\in \{0,1\}$, that is, the hypergraphon $\phi$ is either empty
or complete.
\end{lemma}

\begin{proof}
Let $M$ be the unique $k$-hypergraph on $\ell$ vertices such that
$\phi(M)=1$. Then $M\in \{K_\ell^{(k)}, \overline K_\ell^{(k)}\}$ as
$\phi(K_\ell^{(k)})= \phi(\overline K_\ell^{(k)})=0$ would have contradicted
Ramsey's Theorem. The lemma follows.
\end{proof}

\begin{lemma}\label{lem:categoricalanti-mon}
  If $\phi\in\HomT{T}$ is $\ell$-categorical and $0\leq\ell'\leq\ell$, then $\phi$ is $\ell'$-categorical.
\end{lemma}

\begin{proof}
  Let $M\in\cM_{\ell'}[T]$ and consider the open interpretation $I\interpret{\TkHypergraph[\ell']}{T}$ that
  declares $m$-edges to be isomorphic copies of $M$. By Lemma~\ref{lem:categoricalnat}, it follows that
  $\phi^I$ is $\ell$-categorical, and it follows from Lemma~\ref{lem:ramsey} that $\phi^I$ is either the empty
  or the complete hypergraphon. Now, $\phi$ is $\ell'$-categorical by Remark~\ref{rmk:categorical}.
\end{proof}

\begin{lemma}\label{lem:UInduce->categorical}
  If $\phi\in\HomT{T}$ satisfies $\UInduce[\ell]$, then $\phi$ is $\ell'$-categorical for every
  $0\leq\ell'\leq\ell$.
\end{lemma}

\begin{proof}
  By Lemma~\ref{lem:categoricalanti-mon}, it is enough to show the case $\ell'=\ell$. Let
  $I\interpret{T}{T\cup\TkHypergraph[\ell]}$ and $J\interpret{\TkHypergraph[\ell]}{T\cup\TkHypergraph[\ell]}$
  be the structure-erasing interpretations. Let $\cN$ be a $T$-on such that $\phi_\cN = \phi$ and for
  $M\in\cM_\ell[T]$, let $\cH$ be the $T\cup\TkHypergraph[\ell]$-on given by
  \begin{align*}
    \cH_P & \df \cN_P; &
    \cH_E & \df \bigcup_{\substack{K\in\cK_\ell[T]\\ K\cong M}}\Tind(K,\cN)
  \end{align*}
  for every predicate symbol $P$ in the language of $T$.

  Let $\widehat{M}\in\cM_\ell[T\cup\TkHypergraph[\ell]]$ be such that $I(\widehat{M})\cong M$ and
  $J(\widehat{M})\cong \rho_\ell$. Then
  \begin{align*}
    \phi(M) & = \phi_\cH(\widehat{M}) = \phi(M)\phi_\cH^J(\rho_\ell) = \phi(M)^2,
  \end{align*}
  where the second equality follows since $\phi\in\UInduce[\ell]$. Hence $\phi(M)\in\{0,1\}$ for every
  $M\in\cM_\ell[T]$, so $\phi$ is $\ell$-categorical by Remark~\ref{rmk:categorical}.
\end{proof}

\begin{remark}
The converse to Lemma~\ref{lem:UInduce->categorical} is very far from being true.
For example, every graphon is 1-categorical, and, slightly less trivially, every
tournamon is 2-categorical. They are seldom uniquely 1-inducible.
\end{remark}

We can finally prove the last implication of Theorem~\ref{thm:UInduce}.

\begin{lemma}[Theorem~\ref{thm:UInduce}\ref{thm:UInduce:UInduce}$\implies$\ref{thm:UInduce:symlocal}]\label{lem:UInduce->symlocal}
  If $\phi\in\HomT{T}$ satisfies $\UInduce[\ell]$, then $\phi$ is symmetrically $\ell$-local.
\end{lemma}

\begin{proof}
  Let $\rn{K}$ be the exchangeable array corresponding to $\phi$. We need to show that for every finite
  collection $(V_i)_{i\in[t]}$ of finite subsets of $\NN_+$ with pairwise intersections of size at most $\ell$
  and every collection $(M_i)_{i\in[t]}$ of models of $T$, we have
  \begin{align*}
    \PP[\forall i\in[t], \rn{K}\rest_{V_i}\cong M_i]
    & =
    \prod_{i\in[t]} \PP[\rn{K}\rest_{V_i}\cong M_i].
  \end{align*}

  By Lemma~\ref{lem:UInduce->categorical}, we know that $\phi$ is $\ell'$-categorical for every
  $0\leq\ell'\leq\ell$, which implies that if $\lvert V\rvert\leq\ell$, then $\PP[\rn{K}\rest_V\cong M] =
  \phi(M)\in\{0,1\}$, i.e., the event $\rn{K}\rest_V \cong M$ is trivial. So we may assume that $\lvert
  V_i\rvert > \ell$ for every $i\in[t]$.

  Let $\vec{k} = (k_1,\ldots,k_t)$ be given by $k_i\df\lvert V_i\rvert$ and consider the interpretation
  $I\interpret{\TkHypergraph[\vec{k}]}{T}$ that declares $E_i$-edges to be isomorphic copies of $M_i$. In
  other words, $I$ is given by
  \begin{align*}
    I(E_i)(x_1,\ldots,x_{k_i}) & \df \bigvee_{\sigma\in S_{k_i}} \Dopen(M_i)(x_{\sigma(1)},\ldots,x_{\sigma(k_i)}).
  \end{align*}

  By Theorem~\ref{thm:naturality}, we know that for every $i\in[t]$ we have $\phi^{I\comp
    I_i}\in\UInduce[\ell]$ and by Theorem~\ref{thm:UInduce->CliqueDisc}, it follows that $\phi^{I\comp
    I_i}\in\CliqueDisc[\ell]$. Then we have
  \begin{align*}
    \PP[\forall i\in[t], \rn{K}\rest_{V_i}\cong M_i]
    & =
    \PP[\forall i\in[t], V_i\in E_i(I(\rn{K}))]
    \\
    & =
    \prod_{i\in[t]} \PP[V_i\in E_i(I(\rn{K}))]
    =
    \prod_{i\in[t]} \PP[\rn{K}\rest_{V_i}\cong M_i],
  \end{align*}
  where the second equality follows from Proposition~\ref{prop:vecCliqueDisc}.
\end{proof}

We finish this section with the (now trivial) proof of Theorem~\ref{thm:anti-monotone}.

\begin{proofof}{Theorem~\ref{thm:anti-monotone}}
  The facts $\Independence[\ell]\implies\Independence[\ell-1]$ and $\UCouple[\ell]\implies\UCouple[\ell-1]$
  follow easily from definitions. The fact that $\UInduce[\ell]\implies\UInduce[\ell-1]$ follows since
  symmetric $\ell$-locality trivially implies symmetric $(\ell-1)$-locality and from
  Lemmas~\ref{lem:symlocal->UInduce} and~\ref{lem:UInduce->symlocal}.
\end{proofof}

\section{Unique coupleability}
\label{sec:UCouple}

In this section we prove Theorem~\ref{thm:UCouple}. We start with the
equivalence~\ref{thm:UCouple:UCouple}$\equiv$\ref{thm:UCouple:QRhyper}$\equiv$\ref{thm:UCouple:QRhyperindepcoup}. While
implications~\ref{thm:UCouple:UCouple}$\implies$\ref{thm:UCouple:QRhyperindepcoup}
and~\ref{thm:UCouple:QRhyperindepcoup}$\implies$\ref{thm:UCouple:QRhyper} are fairly straightforward, the
proof of the implication~\ref{thm:UCouple:QRhyper}$\implies$\ref{thm:UCouple:UCouple} is more involved and
naturally splits into five rather independent parts:

\begin{enumerate}[label={\arabic*.}, ref={\arabic*}]
\item Show that unique coupleability of $\phi$ with the quasirandom $\ell'$-hypergraphon $\psi_{\ell',p}$ for
  some $p\in(0,1)$ implies the same statement for \emph{every} $p\in(0,1)$.\label{it:changedensities}
\item Show that unique coupleability of $\phi$ with the quasirandom $\ell'$-hypergraphon $\psi_{\ell',p}$ for
  all $p\in(0,1)$ implies $\phi$ is unique coupleable with the quasirandom $c$-colored $\ell'$-hypergraphon
  $\psi_{\ell',q}$ for every $c\geq 2$ and every $q\in\Pi_c$.\label{it:changenumberofcolors}
\item Show that unique coupleability of $\phi$ with all quasirandom colored $\ell'$-hypergraphons for
  $\ell'\in[\ell]$ implies $\phi$ is uniquely coupleable with all independent couplings
  $\psi_{1,p_1}\otimes\cdots\otimes\psi_{\ell,p_\ell}$ of quasirandom colored $\ell'$-hypergraphons for $\ell'\in[\ell]$.\label{it:quasirandomcolhyp}
\item Show that in an arbitrary theory $T'$, the set of elements that are uniquely coupleable with
  $\phi\in\HomT{T}$ is closed in $\HomT{T'}$ in the \emph{$L_1$-topology}.\label{it:closedness}
\item Show that for any pure canonical theory $T_\cL$, the set of all elements of the form
  $(\psi_{1,p}\otimes\cdots\otimes\psi_{\ell,p})^I$, where $I\interpret{T_\cL}{\bigcup_{k=1}^\ell\TcColkHyp}$ is
  an open interpretation, is dense in the set of $\psi\in\HomT{T_\cL}$ of rank at most $\ell$ (again in the
  $L_1$-topology) and apply Theorem~\ref{thm:naturality}.\label{it:density}
\end{enumerate}

Let us point out that items~\ref{it:changedensities},
\ref{it:changenumberofcolors} and~\ref{it:quasirandomcolhyp} combined show a
strengthened version of
implication~\ref{thm:UCouple:QRhyper}$\implies$\ref{thm:UCouple:QRhyperindepcoup}
that allows for multiple colors and arbitrary densities. Furthermore, most
likely items~\ref{it:closedness} and~\ref{it:density} in this program can be
replaced with an ad hoc argument but we prefer this more structured approach.

\smallskip

We start with item~\ref{it:changedensities}.

\begin{lemma}\label{lem:UCouple:densities}
  Let $\ell\in\NN_+$ and $\phi\in\HomT{T}$. If there exists $p\in(0,1)$ such that $\phi$ is uniquely
  coupleable with the quasirandom $\ell$-hypergraphon $\psi_{\ell,p}$, then $\phi$ is uniquely coupleable with
  $\psi_{\ell,q}$ for every $q\in(0,1)$.
\end{lemma}

\begin{proof}
  Let $\cC_q$ be the set of all couplings of $\phi$ with $\psi_{\ell,q}$. Our objective is to show that
  $\lvert\cC_q\rvert = 1$. Without loss of generality, let us suppose that $p < q$ (otherwise, we can use the
  complementation automorphism $C\interpret{\TkHypergraph[\ell]}{\TkHypergraph[\ell]}$ given by
  $C(E)(\vec{x})\df\bigwedge_{1\leq i < j\leq\ell} x_i\neq x_j\land\neg E(\vec{x})$ and Theorem~\ref{thm:naturality}).
  Intuitively, we are going to ``dilute'' $\psi_{\ell,q}$ by a factor $t=p/q$ so
  that it will turn into $\psi_{\ell,p}$. The simplest way to make this intuition
  precise is by introducing yet another quasi-random hypergraphon $\psi_{\ell,t}$
  on the same ground set and then taking its intersection with $\psi_{\ell,q}$.

  Formally, we consider the commutative diagram
  \begin{equation}\label{eq:UCouple:coloringcommutative}
    \begin{tikzcd}[column sep={-0.8cm}]
      \TkHypergraph[\ell]
      \arrow[dr, "J"]
      \arrow[dddrr, "I'"', bend right]
      &
      &
      T
      \arrow[dl, "I"']
      \arrow[dr, "I"]
      &
      &
      \TkHypergraph[\ell]
      \arrow[dl, "J"']
      \arrow[dddll, bend left]
      \\
      &
      T\cup\TkHypergraph[\ell]
      \arrow[dr, "\id_T\cup I'"']
      &
      &
      T\cup\TkHypergraph[\ell]
      \arrow[dl]
      \\
      &
      &
      T\cup\TkHypergraph[\ell]\cup\TkHypergraph[\ell]
      \\
      &
      &
      \TkHypergraph[\ell]\cup\TkHypergraph[\ell]
      \arrow[u, "\widehat{J}"]
    \end{tikzcd}
  \end{equation}
  where $I$, $J$, $\widehat{J}$ and the unlabeled arrows are the structure-erasing interpretations, with
  the unlabeled arrows keeping the second copy of $\TkHypergraph[\ell]$, and $I'$ is given by
  \begin{align*}
    I'(E)(x_1,\ldots,x_\ell)
    & =
    E(x_1,\ldots,x_\ell)\land E'(x_1,\ldots,x_\ell).
  \end{align*}
  Here $E$ corresponds to the first copy of $\TkHypergraph[\ell]$ and $E'$ corresponds to the second one.

  We now define the dilution map $F\function{\cC_q}{\cC_p}$ by
  \begin{align*}
    F(\xi) & \df (\xi\otimes\psi_{\ell,t})^{\id_T\cup I'},
  \end{align*}
  where $t\df p/q\in(0,1)$. The fact that $F(\xi)$ is indeed an element of $\cC_p$ follows from
  \begin{align*}
    ((\xi\otimes\psi_{\ell,t})^{\id_T\cup I'})^I & = (\phi\otimes\psi_{\ell,t})^I = \phi;\\
    ((\xi\otimes\psi_{\ell,t})^{\id_T\cup I'})^J & = (\psi_{\ell,q}\otimes\psi_{\ell,t})^{I'} = \psi_{\ell,p}.
  \end{align*}

  For $M\in\cM[T]$ and $U\subseteq\binom{V(M)}{\ell}$, let $M_U$ be the model of $T\cup\TkHypergraph[\ell]$
  obtained from $M$ by declaring the $\ell$-hypergraph edge set to be $U$, that is, we have $I(M_U) = M$ and
  $E(J(M_U)) = U$. Then we have
  \begin{align*}
    F(\xi)(\langle M_U\rangle)
    & =
    t^{\lvert U\rvert}\sum_{\substack{W\subseteq\binom{[m]}{\ell}\\U\subseteq W}}
    (1-t)^{\lvert W\setminus U\rvert}\xi(\langle M_W\rangle).
  \end{align*}
  By M\"{o}bius inversion, it follows that $F$ is injective\footnote{The left-inverse is given by
    \begin{align*}
      F^{-1}(\zeta)(\langle M_U\rangle)
      & =
      t^{-\lvert U\rvert}
      \sum_{\substack{W\subseteq\binom{[m]}{\ell}\\U\subseteq W}}
      \left(1 - \frac{1}{t}\right)^{\lvert W\setminus U\rvert}\zeta(\langle M_W\rangle).
    \end{align*}
  }, hence $\lvert\cC_q\rvert\leq\lvert\cC_p\rvert = 1$ as claimed.
\end{proof}

We now proceed to item~\ref{it:changenumberofcolors} of our program.

\begin{lemma}\label{lem:changenumberofcolors}
  Let $\phi\in\HomT{T}$ and $\ell\in\NN_+$ and suppose that for every $p\in(0,1)$, $\phi$ is uniquely
  coupleable with the quasirandom $\ell$-hypergraphon $\psi_{\ell,p}$. Then for every $c\geq 2$ and every
  $q\in\Pi_c$, $\phi$ is uniquely coupleable with the quasirandom $c$-colored $\ell$-hypergraphon
  $\psi_{\ell,q}$.
\end{lemma}

\begin{proof}
  For $i\in[c]$, consider the following commutative diagram
  \begin{equation*}
    \begin{tikzcd}
      \TkHypergraph[\ell]\arrow[r, "J"]\arrow[d, "I'_i"'] &
      T\cup\TkHypergraph[\ell]\arrow[d, "\id_T\cup I'_i"'] &
      T\arrow[l, "I"']\arrow[dl, "I_c"]
      \\
      \TcColkHyp[{c}{\ell}]\arrow[r, "J_c"'] &
      T\cup\TcColkHyp[{c}{\ell}]
    \end{tikzcd}
  \end{equation*}
  where $I$, $I_c$, $J$ and $J_c$ are structure-erasing and $I'_i$ is given by
  \begin{align*}
    I'_i(E)(x_1,\ldots,x_\ell) & \df E_i(x_1,\ldots,x_\ell).
  \end{align*}

  The set $\cK_m[\TcColkHyp[{c}{\ell}]]$ of labeled models of size $m$ can
  be naturally identified with functions $f\function{\binom{[m]}{\ell}}{[c]}$:
given $m\in\NN$ and $f\function{\binom{[m]}{\ell}}{[c]}$,
$C_f\in\cK_m[\TcColkHyp[{c}{\ell}]]$ is given by
  \begin{align*}
    V(C_f) & \df [m]; &
    R_{E_i}(C_f) & \df \{\alpha\in([m])_\ell \mid f(\im(\alpha)) = i\}\quad(i\in[c]).
  \end{align*}
  Let $F\df C^{-1}$.
  Given further $K\in\cK_m[T]$ and $f\function{\binom{[m]}\ell}{[c]}$, let $K_f$
  be the alignment of $K$ and $C_f$, that is, $K_f$ is the unique
  model in $\cK_m[T\cup\TcColkHyp[{c}{\ell}]]$ such that $I_c(K_f)=K$ and $J_c(K_f)=C_f$. Similarly, given $U\subseteq\binom{[m]}{\ell}$, let
  $K_U\in\cK_m[T\cup\TkHypergraph[\ell]]$ be the unique model such that $I(K_U)=K$ and $R_E(K_U) =
  \{\alpha\in([m])_\ell \mid \im(\alpha)\in U\}$.

  Let $\psi\df\psi_{\ell,q}\in\HomT{\TcColkHyp[{c}{\ell}]}$ and let $\xi$ be a coupling of $\phi$ and
  $\psi$. Our goal is to show that
  \begin{align}\label{eq:changenumberofcolors:objective}
    \xi(\langle K_f\rangle)
    & =
    \psi(\langle C_f\rangle)\phi(\langle K\rangle)
  \end{align}
  for every $m\in\NN$, every $K\in\cK_m[T]$ and every
  $f\function{\binom{[m]}{\ell}}{[c]}$. Note that to improve readability, here
  and in the forthcoming
  calculations, $K$ and $K_f$ are identified with their isomorphism classes
  $[K]$, $[K_f]$ in $\cM_m$.

  If $m < \ell$, then~\eqref{eq:changenumberofcolors:objective} holds trivially and if $\phi(\langle K\rangle)
  = 0$, then both sides of~\eqref{eq:changenumberofcolors:objective} are $0$, so suppose $m\geq\ell$ and
  $\phi(\langle K\rangle) > 0$.  Note that $\xi^{(\id_T\cup I_i')\comp J} = \psi^{I_i'} = \psi_{\ell,q_i}\in\HomT{\TkHypergraph[\ell]}$,
  hence $\xi^{\id_T\cup I_i'}$ is a coupling of $\phi$ and $\psi_{\ell,q_i}$, so we must have $\xi^{\id_T\cup
    I_i'} = \phi\otimes\psi_{\ell,q_i}$. Note also that for $m\in\NN$, $K\in\cK_m[T]$ and
  $U\subseteq\binom{[m]}{\ell}$, we have
  \begin{align}\label{eq:changenumberofcolors:piidTIi'}
    \pi^{\id_T\cup I_i'}(\langle K_U\rangle)
    & =
    \sum_{\substack{f\function{\binom{[m]}{\ell}}{[c]}\\ f^{-1}(i) = U}}
    \langle K_f\rangle.
  \end{align}

  Pick now $\rn{f}\function{\binom{[m]}{\ell}}{[c]}$ at random according to
  the distribution
  \begin{align*}
    \PP[\rn{f} = f]
    & \df
    \frac{\xi(\langle K_f\rangle)}{\phi(\langle K\rangle)}.
  \end{align*}

  The identity~\eqref{eq:changenumberofcolors:piidTIi'} allows us to compute, for $A\in\binom{[m]}{\ell}$ and $i\in[c]$, that
  \begin{align*}
    \PP[\rn{f}(A) = i]
    & =
    \sum_{\substack{f\function{\binom{[m]}{\ell}}{[c]}\\f(A) = i}}
    \frac{\xi(\langle K_f\rangle)}{\phi(\langle K\rangle)}
    =
    \sum_{\substack{U\subseteq\binom{[m]}{\ell}\\A\in U}}
    \frac{\xi^{\id_T\cup I_i'}(\langle K_U\rangle)}{\phi(\langle K\rangle)}
    \\
    & =
    \sum_{\substack{U\subseteq\binom{[m]}{\ell}\\A\in U}}
    q_i^{\lvert U\rvert}(1-q_i)^{\binom{m}{\ell} - \lvert U\rvert}
    =
    q_i,
  \end{align*}
  where the the second equality follows from~\eqref{eq:changenumberofcolors:piidTIi'} and the third equality
  follows since $\xi^{\id_T\cup I_i'} =
  \phi\otimes\psi_{\ell,q_i}$. Since $\psi(\langle C_f\rangle) =
  \prod_{A\in\binom{[m]}{\ell}}q_{f(A)}$, to complete the proof
  of~\eqref{eq:changenumberofcolors:objective}, it remains to
  show that the values
  $(\rn{f}(A) \mid A\in\binom{[m]}{\ell})$ of $\rn{f}$ are mutually independent.

  For that purpose, it is in turn sufficient to prove that for every fixed $A_0\in\binom{[m]}{\ell}$ and every fixed
  $i_0\in[c]$, the event $\rn{f}(A_0)=i_0$ is independent from $\rn{f}\rest_W$, where
  $W\df\binom{[m]}{\ell}\setminus\{A_0\}$.

  To do so, we will generate
  the distribution of $\rn{f}$ in a very specific way. Let $\cN$ be a $T$-on such that $\phi = \phi_\cN$ and
  note that $\psi_{\ell,q_{i_0}}=\phi_{\cN'}$ for the $(\ell-1)$-independent $\TkHypergraph[\ell]$-on $\cN'$
  given by
  \begin{align}\label{eq:changenumberofcolors:cN'E}
    \cN'_E & \df \{x\in\cE_\ell \mid x_{[\ell]} < q_{i_0}\}.
  \end{align}

  Since $\xi^{\id_T\cup I_{i_0}'} = \phi\otimes\psi_{\ell,q_{i_0}} = \phi_{\cN\otimes\cN'}$, by
  Proposition~\ref{prop:theonalignment} applied to the interpretation $\id_T\cup I_{i_0}'$,
  there exists a $(T\cup\TcColkHyp[{c}{\ell}])$-on $\cH$ over $[0,1]^4$ such
  that $\phi_\cH = \xi$ and
  \begin{equation}\label{eq:changenumberofcolors:cH}
    \begin{aligned}
      \cH_P & = \cN_P\times\cE_{k(P)}([0,1]^3)\ \text{a.e}\qquad (P\in\cL);
      \\
      \cH_{E_{i_0}} & = \cE_\ell\times\cN'_E\times\cE_\ell([0,1]^2)\ \text{a.e.},
    \end{aligned}
  \end{equation}
  where $\cL$ is the language of $T$.

  Let now $(\rn{\theta^1},\rn{\theta^2},\rn{\theta^3},\rn{\theta^4})$ be picked at random in
  $\cE_{\NN_+}([0,1]^4)$ according to $\lambda$ and let $\rn{K}$ be
  the exchangeable array corresponding to $\cH$ with respect to
  $(\rn{\theta^1},\rn{\theta^2},\rn{\theta^3},\rn{\theta^4})$. Denote also $\rn{F}\df
  F(J_c(\rn{K}\rest_{[m]}))$; $\rn{F} = (\rn{F}(A_0), \rn{F}\rest_W)$, and let $E$ be the event
  $I_c(\rn{K}\rest_{[m]})=K$. Then the function $\rn{f}$ is equidistributed with the function $\rn{F}$
  conditioned by the event $E$. It remains to note that by~\eqref{eq:changenumberofcolors:cH}, the event
  $\rn{F}(A_0)=i_0$ depends only on the coordinate $\rn{\theta^2}_{A_0}$ (warning: we do \emph{not} claim
  that the whole random variable $\rn{F}(A_0)$ depends only on $\rn{\theta^2}_{A_0}$).
  On the other hand, both $E$ and
  $\rn{F}\rest_W$ do not depend on it; more precisely, $E$ depends only on $\rn{\theta}^1$ and $\rn{F}\rest_W$
  depends on those $\rn{\theta^j}_B$ with $j\in[4]$, $\lvert B\rvert\leq\ell$
  and $B\neq A_0$.
\end{proof}

We now address item~\ref{it:quasirandomcolhyp} of our program (cf.\ the second
remark made after the statement of Theorem~\ref{thm:UCouple}).

\begin{lemma}\label{lem:chainuniquecoupleability}
  Let $\phi\in\HomT{T}$ and $\psi_i\in\HomT{T_i}$ for $i\in[t]$. Let also $\ell_1\leq\cdots\leq\ell_t$ and suppose
  that the following hold.
  \begin{enumerate}
  \item For every $i\in\{1,\ldots,t-1\}$, we have $\rk(\psi_i)\leq\ell_i$.
  \item For every $i\in\{2,\ldots,t\}$, we have $\psi_{i}\in\Independence[\ell_{i-1}]$.
  \item For every $i\in\{1,\ldots,t\}$, $\phi$ and $\psi_i$ are uniquely coupleable.
  \end{enumerate}
  Then $\phi,\psi_1,\ldots,\psi_t$ are uniquely coupleable.
\end{lemma}

\begin{proof}
  The proof is by induction on $t$. For $t=1$, the result is trivial. For $t=2$, let $I_i\interpret{T\cup
    T_i}{T\cup T_1\cup T_2}$, $J_i\interpret{T_i}{T\cup T_1\cup T_2}$ and $J\interpret{T}{T\cup T_1\cup T_2}$
  be the structure-erasing interpretations. Let $\cL$, $\cL_1$ and $\cL_2$ be the languages of $T$, $T_1$ and
  $T_2$, respectively. Let also $\cN$ be a $T$-on with $\phi_\cN=\phi$ and $\cH^2$ be an $\ell_1$-independent
  $T_2$-on with $\phi_{\cH^2} = \psi_2$. Fix a coupling $\xi$ of $\phi,\psi_1,\psi_2$. Since $\phi$ and
  $\psi_2$ are uniquely coupleable, we know that $\xi^{I_2} = \phi\otimes\psi_2 = \phi_{\cN\otimes\cH^2}$. By
  Proposition~\ref{prop:theonalignment}, there exists a $(T\cup T_1\cup T_2)$-on $\cG$ over $[0,1]^4$ such
  that $\phi_\cG=\xi$ and
  \begin{align*}
    \cG_P & =
    \begin{dcases*}
      \cN_P\times\cE_{k(P)}([0,1]^3), & if $P\in\cL$;\\
      \cE_{k(P)}\times\cH^2_P\times\cE_{k(P)}([0,1]^2), & if $P\in\cL_2$.
    \end{dcases*}
  \end{align*}
  On the other hand, for the predicate symbols $P$ in $\cL_1$, by possibly changing zero-measure sets of the
  corresponding $P$-ons $\cG_P$ using Proposition~\ref{prop:rankae}, we may suppose that
  $\rk(J_1(\cG))\leq\rk(\psi_1)\leq\ell_1$.

  Let us pick $\rn{\theta}\df(\rn{\theta}^1,\rn{\theta}^2,\rn{\theta}^3,\rn{\theta}^4)$ at random in
  $\cE_{\NN_+}([0,1]^4)$ according to $\lambda$ and let $\rn{K}$ be the exchangeable array corresponding to
  $\cG$ with respect to $\rn{\theta}$. Then we know that $J(\rn{K})$ depends only on $\rn{\theta}^1$,
  $J_1(\rn{K})$ depends only on $((\rn{\theta}^1_A,\rn{\theta}^2_A,\rn{\theta}^3_A,\rn{\theta}^4_A) \mid
  \lvert A\rvert\leq \ell_1)$ and $J_2(\rn{K})$ depends only on $(\rn{\theta}^2_A \mid \lvert A\rvert >
  \ell_1)$ (as $\cH^2$ is $\ell_1$-independent), so $J_2(\rn{K})$ is independent from
  $(J(\rn{K}),J_1(\rn{K}))$. This means that for every $m\in\NN$ and every $K\in\cK_m[T\cup T_1\cup T_2]$, we
  have
  \begin{align*}
    \xi(\langle K\rangle)
    & =
    \PP[\rn{K}\rest_{[m]}=K]
    \\
    & =
    \PP[J(\rn{K})\rest_{[m]}=J(K)\land J_1(\rn{K})\rest_{[m]}=J_1(K)\land J_2(\rn{K})\rest_{[m]}=J_2(K)]
    \\
    & =
    \PP[J(\rn{K})\rest_{[m]}=J(K)\land J_1(\rn{K})\rest_{[m]}=J_1(K)]\cdot\PP[J_2(\rn{K})\rest_{[m]}=J_2(K)]
    \\
    & =
    \PP[I_1(\rn{K})\rest_{[m]}=I_1(K)]\cdot\PP[J_2(\rn{K})\rest_{[m]}=J_2(K)]
    \\
    & =
    \xi^{I_1}(\langle I_1(K)\rangle)\cdot\psi_2(\langle J_2(K)\rangle)
    \\
    & =
    \phi(\langle J(K)\rangle)\cdot\psi_1(\langle J_1(K)\rangle)\cdot\psi_2(\langle J_2(K)\rangle),
  \end{align*}
  where the last equality follows since $\phi$ is uniquely coupleable with $\psi_1$ and $\xi^{I_1}$ is a
  coupling of $\phi$ and $\psi_1$. Therefore $\xi=\phi\otimes\psi_1\otimes\psi_2$.

  \medskip

  For the case $t\geq 3$, let $I\interpret{T\cup\bigcup_{i=2}^t T_i}{T\cup\bigcup_{i=1}^t T_i}$ be the
  structure-erasing interpretation and note that for a coupling $\xi$ of $\phi,\psi_1,\ldots,\psi_t$, it
  follows that $\xi^I$ is a coupling of $\phi,\psi_2,\ldots,\psi_t$. By inductive hypothesis, we must have
  $\xi^I=\phi\otimes\widehat{\psi}$, where $\widehat{\psi}\df\bigotimes_{i=2}^t\psi_i$. In fact, since
  $\phi,\psi_2,\ldots,\psi_t$ are uniquely coupleable, it also follows that $\phi$ is uniquely coupleable with
  $\widehat{\psi}$ (as any coupling of $\phi$ with $\widehat{\psi}$ can be seen as a coupling of
  $\phi,\psi_2,\ldots,\psi_t$). But by Theorem~\ref{thm:naturalityindcoup}, we know that
  $\widehat{\psi}\in\Independence[\ell_1]$ and since $\xi$ can also be seen as a coupling of
  $\phi,\psi_1,\widehat{\psi}$, we get $\xi=\phi\otimes\bigotimes_{i=1}^t\psi_i$ from the previous case.
\end{proof}

\begin{lemma}\label{lem:quasirandomcoloredhypergraphons}
  Let $c\geq 2$, $p\in\Pi_c$ and $k\in\NN_+$. Then the quasirandom $c$-colored $k$-hypergraphon $\psi_{k,p}$
  satisfies $\Independence[k-1]$ and $\rk(\psi_{k,p})=k$.
\end{lemma}

\begin{proof}
  Note that $\psi_{k,p}$ can be represented by the $\TcColkHyp$-on $\cN^{k,p}$ given by
  \begin{align*}
    \cN^{k,p}_{E_i}
    & \df
    \set{x\in\cE_k}{\sum_{j=1}^{i-1} p_j\leq x_{[k]} < \sum_{j=1}^i p_j} &
    (i\in[c]),
  \end{align*}
  hence $\psi_{k,p}\in\Independence[k-1]$ and $\rk(\psi_{k,p})\leq k$. Since $c\geq 2$, it follows
  that $\rk(\psi_{k,p})>0$, so by Theorem~\ref{thm:inter-properties} and Proposition~\ref{prop:rankbasic}, we
  must have $\rk(\psi_{k,p})=k$.
\end{proof}

Proceeding to item~\ref{it:closedness} in the program, we introduce the $L_1$-topology on theons that is a
direct analogue of the $L_1$-topology on graphons~\cite[Sct.~8.2.5 and Sct.~8.3]{Lov12}.

\begin{definition}\label{def:L1}
  If $T$ is a theory in a language $\cL$ and $\phi_1,\phi_2\in\HomT{T}$, then the \emph{$L_1$-distance}
  between $\phi_1,\phi_2$ is defined as
  \begin{align}\label{eq:d1}
    \delta_1(\phi_1,\phi_2) & \df \min_{\cN^1,\cN^2} \sum_{P\in\cL} \mu(\cN^1_P\symdiff\cN^2_P),
  \end{align}
  where the minimum is taken over $T$-ons $\cN^1$ and $\cN^2$ over the same space such that
  $\phi_1 = \phi_{\cN^1}$ and $\phi_2 = \phi_{\cN^2}$.
\end{definition}

It is not immediately clear from this definition that the minimum in~\eqref{eq:d1} is actually attained, nor
is it clear why $\delta_1$ is a metric.

The first issue is easy to address by giving an alternative purely algebraic
definition. Namely, for any $P\in\mathcal L$
introduce the element $d_P\in \cA[T\cup T]$ as
\begin{align*}
  d_P
  & \df
  \sum_{\substack{K\in\cK_{k(P)}[T\cup T]\\ \id_{k(P)}\in R_{P_1}(K)\symdiff R_{P_2}(K)}}\langle K \rangle,
\end{align*}
where $P_1$ and $P_2$ are the two copies of $P$ in $\cL\disjcup\cL$, and let
\begin{align*}
  d_T & \df \sum_{P\in\cL}d_P.
\end{align*}
This element measures the distance in a coupling of $\phi_1,\phi_2$ so we
have
\begin{equation}\label{eq:d1coupling}
  \delta_1(\phi_1,\phi_2) = \inf_\xi \xi(d_T),
\end{equation}
where $\xi$ runs over all couplings of $\phi_1$ and $\phi_2$. Their set is determined in $\HomT{T\cup T}$
by countably many linear equations and hence compact. Therefore the minimum
in~\eqref{eq:d1coupling} and~\eqref{eq:d1} is actually achieved.

The second issue is trickier, and the proof is similar to the analogous proof
that $\delta_1$ is a metric in the case of graphons. Fortunately, we already
did most of the necessary (and notationally heavy) work in the proof of
Proposition~\ref{prop:theonalignment}; we defer the remaining part to
Appendix~\ref{sec:L1topology}.

\smallskip

Let us finally remark why we need $L_1$-topology at all instead of the standard and much nicer density
topology (i.e., the one induced by the inclusion $\HomT{T}\subseteq [0,1]^{\cM[T]}$ from the product
topology). One simple explanation is that the set of all $\psi\in\HomT{T'}$ that are uniquely coupleable with
some $\phi\in\HomT{T}$ is not closed in the latter.

\begin{example}\label{ex:nonclosed}
  Let $\phi_p\in\HomT{\TGraph}$ be the quasirandom graphon of density $p\in(0,1)$. If $(G_n)_{n\in\NN}$
  ($G_n\in\cM_n[\TGraph]$) is a sequence of graphs converging to $\phi_p$, then the associated step functions
  $\psi_n$ converge to $\phi_p$ in the density topology. Since $\rk(\psi_n)=1$ and
  $\phi_p\in\Independence[1]$, it follows that $\phi_p$ and $\psi_n$ are uniquely coupleable, but $\phi_p =
  \lim_{n\to\infty}\psi_n$ is obviously not uniquely coupleable with itself.
\end{example}

The example above illustrates another crucial difference between the $L_1$-topology and density topology: rank
is lower semi-continuous in the former but not the latter. In fact, for pure canonical theories $T_\cL$, the set
$\{\psi\in\HomT{T_\cL} \mid \rk(\psi)\leq r\}$ is closed in $L_1$-topology but dense in $\HomT{T_\cL}$ in
density topology (if $r\geq 1$).

\begin{lemma}\label{lem:ucL1closed}
  Let $\phi\in\HomT{T}$ and $T'$ be an arbitrary theory. Then the set of $\psi\in\HomT{T'}$ that are uniquely
  coupleable with $\phi$ is closed in the $L_1$-topology.
\end{lemma}

\begin{proof}
  Let $(\psi_n)_{n\in\NN}$ be a sequence in $\HomT{T'}$ converging to $\psi$ in the $L_1$-topology and suppose
  every $\psi_n$ is uniquely coupleable with $\phi$. It is clear from the definition that
  $\delta_1(\phi\otimes\psi_n,\phi\otimes\psi)= \delta_1(\psi_n\otimes\psi)$, so
  $\phi\otimes\psi_n$ also
  converges to $\phi\otimes\psi$ in the $L_1$-topology.
  For each $n\in\NN$, let $\zeta_n$ be a coupling of $\psi$ and $\psi_n$
  attaining the $L_1$-distance in~\eqref{eq:d1coupling}.

  Let $\xi$ be a coupling of $\phi$ and $\psi$; we have to show that $\xi = \phi\otimes\psi$. Let
  $I\interpret{T'\cup T'}{T\cup T'\cup T'}$ and $J_i\interpret{T'}{T'\cup T'}$ be the structure-erasing
  interpretations, where $J_i$ keeps the $i$-th copy of $T'$. Since $\xi$ is a coupling of $\phi$ and $\psi =
  \zeta_n^{J_1}$, by Proposition~\ref{prop:lifting}, there exists a coupling $\widehat{\xi}_n$ of $\xi$ and
  $\zeta_n$ such that $\widehat{\xi}_n^{\id_T\cup J_1} = \xi$. Note that $\widehat{\xi}_n$ can also be seen
  as a coupling of $\phi$, $\psi$ and $\psi_n$ as $\widehat{\xi}_n^I = \zeta_n$.

  Let now $\cN^n$ be a $(T\cup T'\cup T')$-on such that $\widehat{\xi}_n = \phi_{\cN^n}$. By considering
  the $(T\cup T')$-ons $(\id_T\cup J_1)(\cN^n)$ and $(\id_T\cup J_2)(\cN^n)$, since $\psi_n$ is uniquely
  coupleable with $\phi$, we conclude from~\eqref{eq:d1} that
  \begin{align*}
    \delta_1(\xi,\phi\otimes\psi_n)
    & \leq
    \sum_{P\in\cL'} \lambda(J_1(I(\cN^n))_P\symdiff J_2(I(\cN^n))_P)
    =
    \zeta_n(d_{T'})
    =
    \delta_1(\psi,\psi_n),
  \end{align*}
  where $\cL'$ is the language of $T'$. Since $\psi_n\to\psi$ and $\phi\otimes\psi_n\to\phi\otimes\psi$ in the
  $L_1$-topology, it follows that $\xi = \phi\otimes\psi$.
\end{proof}

We proceed to the last item~\ref{it:density} in our program, which is to provide a way
of approximating Euclidean structures with
interpretations of independent couplings
$\psi_{1,p}\otimes\cdots\otimes\psi_{\ell,p}$ of quasirandom colored
hypergraphons in the $L_1$-topology.

\begin{lemma}\label{lem:L1approximation}
  Let $\cL$ be a language, $\phi\in\HomT{T_\cL}$, $r\df\rk(\phi)$ and $\epsilon > 0$. Then there exist
  $c\geq 2$, $p\in\Pi_c$ and an open interpretation $I\interpret{T_\cL}{\bigcup_{k=1}^r\TcColkHyp}$ such
  that $\delta_1(\phi,(\bigotimes_{k=1}^r\psi_{k,p})^I)\leq\epsilon$.
\end{lemma}

\begin{proof}
  Let $\cN$ be a $T_\cL$-on such that $\phi_\cN = \phi$ and $\rk(\cN)=r$, that is, for each $P\in\cL$, there
  exists $\cH_P\subseteq\cE_{k(P),r}$ such that $\cN_P = \cH_P\times[0,1]^{\binom{[k(P)]}{> r}}$. By standard
  measure theory arguments, for each $P\in\cL$, there exists a finite family of pairwise disjoint closed cubes
  $(C^P_j)_{j=1}^{m_P}$ ($C^P_j\subseteq\cE_{k(P),r}$) such that setting
  $\cH'_P\df\bigcup_{j=1}^{m_P} C^P_j$ gives $\lambda(\cH_P\symdiff\cH'_P)\leq\epsilon/\lvert\cL\rvert$.

  Let $X$ be the set of all coordinates of vertices of all cubes $C^P_j$ for all $P\in\cL$. The set $X$ induces
  a partition of $[0,1]$ into intervals $J_1,\ldots,J_c$ of positive length (we can ensure $c\geq 2$ by
  including an extra point if necessary). Define then $p\in\Pi_c$ by letting $p_i\df\lambda(J_i) > 0$ and define
  the $(\bigcup_{k=1}^r\TcColkHyp)$-on $\cG$ by
  \begin{align*}
    \cG_{E^k_i}
    & \df
    \{x\in\cE_k \mid x_{[k]}\in J_i\}
    & (i\in[c], k\in[r]),
  \end{align*}
  where for each $k\in[r]$, the symbols $E_1^k,\ldots,E_c^k$ correspond to $\TcColkHyp$.

  Let $\psi\df\phi_\cG$ and note that $\psi$ is a coupling of $\psi_{1,p},\ldots,\psi_{r,p}$, so we must have
  $\psi=\bigotimes_{k=1}^r\psi_{k,p}$ by Lemmas~\ref{lem:chainuniquecoupleability}
  and~\ref{lem:quasirandomcoloredhypergraphons}.

  Note now that from the definition of $X$, each cube $C^P_j\subseteq\cE_{k(P),r}$ can be written as a finite
  union of the form $\bigcup_{u\in U_{P,j}} \prod_{A\in r(k(P),r)} J_{i_{P,u,A}}$. We then define
  $I\interpret{T_\cL}{\bigcup_{k=1}^r\TcColkHyp}$ by
  \begin{align*}
    I(P)(x_1,\ldots,x_{k(P)})
    & \df
    \bigvee_{j=1}^{m_P}\bigvee_{u\in U_{P,j}} \bigwedge_{A\in r(k(P),r)}
    E^k_{i_{P,u,A}}(x_{\iota_A(1)},\ldots,x_{\iota_A(\lvert A\rvert)})
    & (P\in\cL).
  \end{align*}
  Our definition ensures that
  \begin{align*}
    I(\cG)_P
    & =
    \bigcup_{j=1}^{m_P}\bigcup_{u\in U_{P,j}}\left(\prod_{A\in r(k(P),r)} J_{i_{P,u,A}}\times [0,1]^{\binom{[k(P)]}{> r}}\right)
    \\
    & =
    \bigcup_{j=1}^{m_P} (C^P_j\times [0,1]^{\binom{[k(P)]}{> r}})
    =
    \cH'_P\times [0,1]^{\binom{[k(P)]}{> r}}.
  \end{align*}
  This implies that
  \begin{align*}
    \delta_1(\phi,\psi)
    & \leq
    \sum_{P\in\cL}\lambda(\cN_P\symdiff(\cH'_P\times [0,1]^{\binom{[k(P)]}{> r}}))
    =
    \sum_{P\in\cL}\lambda(\cH_P\symdiff\cH'_P)
    \leq
    \epsilon,
  \end{align*}
  as desired.
\end{proof}

We now have all the ingredients to show the
equivalence~\ref{thm:UCouple:UCouple}$\equiv$\ref{thm:UCouple:QRhyper}$\equiv$\ref{thm:UCouple:QRhyperindepcoup}
of Theorem~\ref{thm:UCouple}.

\begin{lemma}[Theorem~\ref{thm:UCouple}\ref{thm:UCouple:UCouple}$\equiv$\ref{thm:UCouple:QRhyper}$\equiv$\ref{thm:UCouple:QRhyperindepcoup}]\label{lem:QRhyper}
  Let $\phi\in\HomT{T}$ and $\ell\in\NN_+$. Then the following are equivalent.
  \begin{enumerate}
  \item $\phi\in\UCouple[\ell]$.\label{thm:QRhyper:UCouple}
  \item For every $\ell'\in[\ell]$, there exists $p\in(0,1)$ such that $\phi$ is uniquely coupleable with
    the quasirandom $\ell'$-hypergraphon $\psi_{\ell',p}$.\label{thm:QRhyper:QRhyper}
  \item There exist $p_1,\ldots,p_\ell\in(0,1)$ such that $\phi$ is uniquely coupleable with the independent
    coupling $\psi_{1,p_1}\otimes\cdots\otimes\psi_{\ell,p_\ell}$ of quasirandom $\ell'$-hypergraphons
    $\psi_{\ell',p_{\ell'}}$ for $\ell'\in[\ell]$.\label{thm:QRhyper:QRhyperindepcoup}
  \end{enumerate}
\end{lemma}

\begin{proof}
  Since $\ell'$-hypergraphons have rank at most $\ell'$, by Proposition~\ref{prop:rankae}, we have
  $\rk(\psi_{1,p_1}\otimes\cdots\otimes\psi_{\ell,p_\ell})\leq\ell$, so the
  implication~\ref{thm:QRhyper:UCouple}$\implies$\ref{thm:QRhyper:QRhyperindepcoup} follows.

  \medskip

  Implication~\ref{thm:QRhyper:QRhyperindepcoup}$\implies$\ref{thm:QRhyper:QRhyper} follows from
  Theorem~\ref{thm:naturality} by considering the structure-erasing interpretations
  $I_k\interpret{\TkHypergraph}{\bigcup_{\ell'=1}^\ell\TkHypergraph[\ell']}$.

  \medskip

  For the non-trivial implication~\ref{thm:QRhyper:QRhyper}$\implies$\ref{thm:QRhyper:UCouple}, we want to show that
  $\phi$ is uniquely coupleable with any $\psi\in\HomT{T'}$ of rank at most $\ell$. We can assume
  w.l.o.g.\ that $T' = T_\cL$ for some language $\cL$. Using Lemma~\ref{lem:L1approximation}, for each
  $n\in\NN$, we can find $c_n\geq 2$, $p_n\in\Pi_{c_n}$ and
  $I_n\interpret{T_\cL}{\bigcup_{k=1}^r\TcColkHyp[{c_n}{k}]}$ such that
  $\delta_1(\phi,(\bigotimes_{k=1}^r\psi_{k,p_n})^{I_n})\leq 1/n$.

  By Lemmas~\ref{lem:UCouple:densities}, \ref{lem:changenumberofcolors}, \ref{lem:chainuniquecoupleability}
  and~\ref{lem:quasirandomcoloredhypergraphons}, we know that $\phi$ is uniquely coupleable with
  $\bigotimes_{k=1}^r\psi_{k,p_n}$ and by Theorem~\ref{thm:naturality}, it follows that $\phi$ is also
  uniquely coupleable with $(\bigotimes_{k=1}^r\psi_{k,p_n})^{I_n}$.

  Finally, since $((\bigotimes_{k=1}^r\psi_{k,p_n})^{I_n})_{n\in\NN}$ converges to $\psi$ in the $L_1$-topology,
  by Lemma~\ref{lem:ucL1closed}, it follows that $\phi$ is uniquely coupleable with $\psi$.
\end{proof}

We now proceed to add items~\ref{thm:UCouple:local} and~\ref{thm:UCouple:orderUInduce} to the list of
equivalent properties of Theorem~\ref{thm:UCouple} (recall
that~\ref{thm:UCouple:UCouple}$\equiv$\ref{thm:UCouple:weakindep}$\equiv$\ref{thm:UCouple:weakindepall}
and~\ref{thm:UCouple:weakindep}$\implies$\ref{thm:UCouple:local} were proved in Section~\ref{sec:basic}).

\begin{lemma}[Theorem~\ref{thm:UCouple}\ref{thm:UCouple:local}$\implies$\ref{thm:UCouple:orderUInduce}]\label{lem:localorderUInduce}
  If $\phi\in\HomT{T}$ is $\ell$-local, then the independent coupling of $\phi$ with the linear order
  $\psi\in\HomT{\TLinOrder}$ satisfies $\UInduce[\ell]$.
\end{lemma}

\begin{proof}
  By Lemma~\ref{lem:symlocal->UInduce}, it is enough to show that $\phi\otimes\psi$ is symmetrically
  $\ell$-local. Let $\rn{K}$ be the exchangeable array corresponding to $\phi\otimes\psi$, and
  fix a finite family of finite sets $(V_i)_{i\in[t]}$ ($V_i\subseteq\NN_+$) with
pairwise intersections of size at
  most $\ell$. We let $\rn{K_i}\df \rn{K}\rest_{V_i} \in
  \cK_{V_i}[T\cup\TLinOrder]$ and let $\rn{M_i}\df[\rn{K_i}] \in \cM_{\lvert V_i\rvert}[T\cup\TLinOrder]$
  be the isomorphism type of $\rn{K_i}$. We have to prove that $\rn{M_1},\ldots,\rn{M_t}$
  are mutually independent, and for that purpose we are going to apply Claim~\ref{clm:multiple} again.

  More specifically, let $I\interpret{T}{T\cup\TLinOrder}$ be the
  structure-erasing interpretation and $\rn{L_i}=I(\rn{K_i})\in
  \cK_{V_i}[T]$ be the results of
  erasing linear order. Likewise, let
  $J\interpret{\TLinOrder}{T\cup\TLinOrder}$, and let
  ${\rn{\leq_i}}=J(\rn{K_i})$ be the corresponding (random) linear order on $V_i$ so
  that $\rn{K_i} = (\rn{L_i}, {\rn{\leq_i}})$. In Claim~\ref{clm:multiple}, we set
  $\rn X= (\rn{\leq_1},\ldots,\rn{\leq_n})$, $\rn{Y_i}=\rn{L_i}$, and let
  $f_i(\leq_1,\ldots, \leq_n, L_i)$ be the function first computing
  $K_i$ from $L_i$ and $\leq_i$ and then taking its isomorphism type
  $M_i=[K_i]$.

  We know that the tuple $(\rn{L_1},\ldots, \rn{L_t})$
  is independent from $\rn X = ({\rn{\leq_1}},\ldots,{\rn{\leq_t}})$ (as the coupling of $\phi$ and $\psi$
  is independent) and that $\rn{L_1},\ldots,\rn{L_t}$ are mutually
  independent (as $\phi$ is $\ell$-local). This gives us the first assumption in
  Claim~\ref{clm:multiple}: $\rn X,\rn{Y_1},\ldots, \rn{Y_n}$ are mutually
  independent (note that we do {\em not} claim that $\rn{\leq_1},\ldots,\rn{\leq_n}$
  are mutually independent, this is in general not true). It remains to show
  that $(\rn{M_1},\ldots, \rn{M_n})$ is independent from $(\rn{\leq_1},\ldots,
  \rn{\leq_n})$, and it essentially follows from the observation that the
  function $f_i(X, Y_i)$ becomes invertible after fixing it first argument.

  More specifically, we compute $\rn{L_i}=g_i(\rn{\leq_i},\rn{M_i})$, where
  $g_i(\leq_i,M_i)$ is obtained by first aligning
  the internal order of $V(M_i)$ with the order $\leq_i$ on $V_i$, and
  then discarding it. The crucial property is that $L_i=g_i(\leq_i,M_i)$
  if and only if $M_i = f_i((\leq_1,\ldots,\leq_n),L_i)$. Using this, fixing
  arbitrary models $M_i\in \cM_{\lvert V_i\rvert}[T\cup\TLinOrder]$ and a particular
tuple of values $({\leq_1},\ldots,{\leq_t})$, we have the
  calculation

  \begin{align*}
    &\!\!\!\!\!\!
    \PP[\forall i\in [t], \rn{M_i}\cong M_i \mid \forall i\in [t], {\rn{\leq_i}}={\leq_i}]
    \\
    & = \PP[\forall i\in [t], \rn{L_i}=g_i(\leq_i, M_i) \mid \forall i\in [t], {\rn{\leq_i}}={\leq_i}]
    \\
    & = \PP[\forall i\in [t], \rn{L_i}=g_i(\leq_i, M_i)]
    \\
    & = \PP[\forall i\in[t], \rn{M_i}\cong M_i].
  \end{align*}

  This shows that $(\rn{M_1},\ldots,\rn{M_t})$ is indeed independent from $(\rn{\leq_1},\ldots,\rn{\leq_t})$. We are now in position to apply Claim~\ref{clm:multiple} which completes the proof.
\end{proof}

\begin{lemma}[Theorem~\ref{thm:UCouple}\ref{thm:UCouple:orderUInduce}$\implies$\ref{thm:UCouple:QRhyper}]\label{lem:orderUInducecolors}
  If the independent coupling of $\phi\in\HomT{T}$ with the linear order $\psi\in\HomT{\TLinOrder}$ satisfies
  $\UInduce[\ell]$, then for every $\ell'\in[\ell]$, $\phi$ is uniquely coupleable with the quasirandom
  $\ell'$-hypergraphon $\psi_{\ell',1/2}\in\HomT{\TkHypergraph[\ell']}$.
\end{lemma}

\begin{proof}
  Let $\cL$ be the language of $T$ and note that since $\UInduce[\ell]$ implies $\UInduce[\ell']$
  (Theorem~\ref{thm:anti-monotone}), it is sufficient to consider the case $\ell'=\ell$. Let us first
  assume
  $\ell\geq 2$.

  Recall that the unique $\psi\in\HomT{\TLinOrder}$ can be represented by the $\TLinOrder$-on $\cN^<$ given by
  \begin{align*}
    \cN^< & \df \{x\in\cE_2 \mid x_{\{1\}} < x_{\{2\}}\},
  \end{align*}
  and that $\psi_{\ell,1/2}$ can be represented as
  \begin{align*}
    \cN_E & \df \{x\in\cE_\ell \mid x_{[\ell]} \leq 1/2 \}.
  \end{align*}

  Let $\xi$ be a coupling of $\phi$ and $\psi_{\ell,1/2}$ and let $\cN$ be a $(T\cup\TkHypergraph[\ell])$-on
  such that $\phi_\cN = \xi$. As in the proof of Lemma~\ref{lem:changenumberofcolors},
  for every $m\in\NN$ and every $U\subseteq\binom{[m]}{\ell}$, let $H_U\in\cK_m[\TkHypergraph[\ell]]$ be the
  hypergraph given by $V(H_U)\df[m]$ and $R_E(H_U)\df \set{\alpha\in([m])_\ell}{\im(\alpha)\in U}$. If we are further given $K\in\cK_m[T]$, let
  $K_U\in\cK_m[T\cup\TkHypergraph[\ell]]$ be the alignment of $K$ and $H_U$, that is, we have $R_P(K_U)\df
  R_P(K)$ ($P\in\cL$) and $R_E(K_U)\df R_E(H_U)$. Finally, we let
  $K_U^<\in\cK_m[T\cup\TkHypergraph[\ell]\cup\TLinOrder]$ be the model obtained from $K_U$ by equipping it
  with the natural order of $[m]$. Note that while we do need labels in $K$
  to properly define the models $K_U$ and $K_U^<$, in the computations below
  they are treated as unlabeled models $[K_U]$, $[K_U^<]$, i.e., labels are discarded.

  To show that $\xi$ is the independent coupling of $\phi$ and $\psi_{\ell,1/2}$, we need to show that for
  every $m\in\NN$, every $K\in\cK_m[T]$ and every $U\subseteq\binom{[m]}{\ell}$, we have
  \begin{align}\label{eq:orderUInducecolors:objective}
    \xi(\langle K_U\rangle)
    & =
    \phi(\langle K\rangle)\cdot\psi_{\ell,1/2}(\langle H_U\rangle)
    =
    \frac{\phi(\langle K\rangle)}{2^{\binom{m}{\ell}}}.
  \end{align}

  The assertion is trivial if $m < \ell$, so suppose $m\geq\ell$. Fix $U\subseteq\binom{[m]}{\ell}$ and
  for every $v\in[m]$, define
  \begin{align*}
    V_v & \df
    \begin{dcases*}
      \left[\frac{v-1}{m},\frac{v}{m}\right), & if $v < m$;\\
      \left[\frac{m-1}{m},1\right], & if $v = m$.
    \end{dcases*}
  \end{align*}
  For $n\in\NN$ and $y\in\cE_n$, let $\alpha_y\function{[n]}{[m]}$ be the unique function such that
  $y_{\{j\}}\in V_{\alpha_y(j)}$ for every $j\in[n]$. Finally, define the set
  \begin{align*}
    W_U
    & \df
    \set{(x,y)\in\cE_\ell\times\cE_\ell}{
      \lvert\im(\alpha_y)\rvert = \ell\land\left(x_{[\ell]}\leq \frac{1}{2} \equiv \im(\alpha_y)\in U\right)
    };
  \end{align*}
  clearly, $W_U$ is $S_\ell$-invariant. This means that we can define the
  $(T\cup\TkHypergraph[\ell]\cup\TLinOrder)$-on $\cH^U$ over $[0,1]^2$ by
  \begin{align*}
    \cH^U_P & \df \cN_P\times\cE_{k(P)} \qquad (P\in\cL), &
    \cH^U_\prec & \df \cE_2\times\cN^<, &
    \cH^U_E & \df W_U.
  \end{align*}

 Obviously, if $(x,y)\in\Tind(K_m^{(\ell)},W_U)$, then each $y_{\{j\}}$ must belong to
  a different $V_v$. Indeed, if there exist $j_1,j_2\in[m]$ with $y_{\{j_1\}},y_{\{j_2\}}\in V_v$ but $j_1\neq
  j_2$, since $m\geq\ell\geq 2$, there exists $\beta\in([m])_\ell$ with $j_1,j_2\in\im(\beta)$ and thus
  $(x,y)\notin (\beta^*)^{-1}(W_U)$, a
  contradiction.

  Our claim and the definition of $W_U$ then imply
  \begin{align*}
    \Tind(K_m^{(\ell)},W_U)
    & =
    \set{(x,y)\in\cE_m\times\cE_m}{
      \lvert\im(\alpha_y)\rvert = m
      \land\bigwedge_{\beta\in([m])_\ell}(\beta^*(x)\in\cN_E\equiv \im(\alpha_{\beta^*(y)})\in U)
    }.
  \end{align*}
  Thus, denoting by $J_\ell\interpret{\TkHypergraph[\ell]}{T\cup \TkHypergraph[\ell]\cup\TLinOrder}$ the structure-erasing
  interpretation, we get
  \begin{align}\label{eq:orderUInducecolors:cliques}
    \phi_{J_\ell(\cH^U)}(K_m^{(\ell)}) & = \frac{m!}{m^m}\psi_{\ell,1/2}(H_U) = \frac{m!}{m^m\cdot 2^{\binom{m}{\ell}}}.
  \end{align}

  Let now $J\interpret{T}{T\cup\TkHypergraph[\ell]\cup\TLinOrder}$ be another structure-erasing interpretation; we have
  \begin{align*}
    \Tind(K_{\binom{[m]}{\ell}}^<,\cH^U)
    & =
    \Tind(K,J(\cH^U))\cap\Tind(K_m^{(\ell)},J_\ell(\cH^U))\cap
    \{(x,y)\in\cE_m\times\cE_m \mid y_{\{1\}} < \cdots < y_{\{m\}}\}
    \\
    & =
    \{(x,y)\in\cE_m\times\cE_m \mid x\in\Tind(K_U,\cN)\land\forall v\in[m], y_{\{v\}} \in V_v\}.
  \end{align*}

  Since $\phi_\cN = \xi$, we get
  \begin{align*}
    \xi(\langle K_U\rangle)
    & =
    m^m\cdot\phi_{\cH^U}(\langle K_{\binom{[m]}{\ell}}^<\rangle)
    =
    \frac{m^m\cdot\phi(\langle K\rangle)\cdot\phi_{J_\ell(\cH^U)}(K_m^{(\ell)})}{m!}
    =
    \frac{\phi(\langle K\rangle)}{2^{\binom{m}{\ell}}},
  \end{align*}
  where the second equality follows since $\phi_{\cH^U}$ is a coupling of
  $\phi_{J_\ell(\cH^U)}\in\HomT{\TkHypergraph[\ell]}$ and $\phi\otimes\psi$ (and the latter satisfies
  $\UInduce[\ell]$), and the third equality follows
  from~\eqref{eq:orderUInducecolors:cliques}. Hence~\eqref{eq:orderUInducecolors:objective} holds.

  \medskip

  Let us now show the case $\ell = 1$. In this case, since $\TkHypergraph[1]\cong\TcColoring[2]$, we will work
  with the latter theory. Let $\xi$ be a coupling of $\phi$ and $\psi_{(1/2,1/2)}\in\HomT{\TcColoring[2]}$ and
  let $\cN$ be a $(T\cup\TcColoring[2])$-on such that $\phi_\cN = \xi$.

  For every $m\in\NN$, every $K\in\cK_m[T]$ and every $j\in\{0,\ldots,m\}$, let
  $K_j\in\cK_m[T\cup\TcColoring[2]]$ be the model obtained from $K$ by coloring the first $j$ vertices with
  color $1$ and all others with color $2$, that is, we have $R_P(K_j)\df R_P(K)$ ($P\in\cL$),
  $R_{\chi_1}(K_j)\df[j]$ and $R_{\chi_2}(K_j)\df\{j+1,\ldots,m\}$. Again, we let
  $K_j^<\in\cK_m[T\cup\TcColoring[2]\cup\TLinOrder]$ be the model obtained from $K_j$ by equipping it with the
  natural order of $[m]$, and, again, in the computations below we view $K, K_j,
  K_j^<$ as unlabeled models.

  Due to exchangeability, in order to show that $\xi$ is the independent coupling of $\phi$ and $\psi_{(1/2,1/2)}$, it is
  sufficient to show that for every $m\in\NN$,
  every $K\in\cK_m[T]$ and every $j\in\{0,\ldots,m\}$, we have
  \begin{align}\label{eq:orderUInducecolors:secondobjective}
    \xi(\langle K_j\rangle)
    & =
    \frac{\phi(\langle K\rangle)}{2^m}.
  \end{align}

  For every $t\in(0,1)$, let
  \begin{align*}
    U_t\df\{(x,y)\in\cE_1\times\cE_1 \mid x\in\cN_{\chi_1}\equiv y < t\}
  \end{align*}
  ($\chi_1$ corresponds to the first color)
  and note that $\lambda(U_t)=1/2$. Define the $(T\cup\TLinOrder\cup\TcColoring[2])$-on $\cH^t$ over
  $[0,1]^2$ by
  \begin{align*}
    \cH^t_P & \df \cN_P\times\cE_{k(P)} \qquad (P\in\cL), &
    \cH^t_\prec & \df \cE_2\times\cN^<,
    \\
    \cH^t_{\chi_1} & \df U_t, &
    \cH^t_{\chi_2} & \df (\cE_1\times\cE_1)\setminus U_t.
  \end{align*}

  Since $\phi_{\cH^t}$ is a coupling of $\psi_{(1/2,1/2)}$ and $\phi\otimes\psi$ and the latter satisfies
  $\UInduce[1]$, we get
  \begin{align}\label{eq:orderUInducecolors:UInduce1}
    \phi_{\cH^t}(\langle K_m^<\rangle)
    & =
    \frac{\phi(\langle K\rangle)}{m!\cdot 2^m}.
  \end{align}

  On the other hand, from the definition of $\cH^t$, we have
  \begin{align*}
    \phi_{\cH^t}(\langle K_m^<\rangle)
    & =
    \sum_{j=0}^m \frac{t^j(1-t)^{m-j}}{j!(m-t)!}\xi(\langle K_j\rangle)
    \\
    & =
    \sum_{k=0}^m\left(\sum_{j=0}^k\frac{1}{j!(m-j)!}\binom{m-j}{k-j}(-1)^{k-j}\xi(\langle K_j\rangle)\right) t^k.
  \end{align*}
  Since this identity is true for any $t$, putting it together with~\eqref{eq:orderUInducecolors:UInduce1} and comparing coefficients of the
  polynomials in $t$, we conclude that
  \begin{align}\label{eq:orderUInducecolors:polynomialcomparison}
    \sum_{i=0}^k\frac{1}{i!(m-i)!}\binom{m-i}{k-i}(-1)^{k-i}\xi(\langle K_i\rangle)
    & =
    \begin{dcases*}
      \frac{\phi(\langle K\rangle)}{m!\cdot 2^m}, & if $k = 0$;\\
      0, & if $k\in[m]$.
    \end{dcases*}
  \end{align}

  We can finally prove~\eqref{eq:orderUInducecolors:secondobjective} by induction in $j\in\{0,\ldots,m\}$. For
  $j = 0$, the assertion follows from~\eqref{eq:orderUInducecolors:polynomialcomparison} for $k=0$. Suppose
  then that $j\geq 1$ and by using the inductive hypothesis, note
  that~\eqref{eq:orderUInducecolors:polynomialcomparison} for $k=j$ gives
  \begin{align*}
    \xi(\langle K_j\rangle)
    & =
    - j!(m-j)!\sum_{i=0}^{j-1}\frac{1}{i!(m-i)!}\binom{m-i}{j-i}(-1)^{j-i}\frac{\phi(\langle K\rangle)}{2^m}
    \\
    & =
    - \sum_{i=0}^{j-1}\binom{j}{i}(-1)^{j-i}\frac{\phi(\langle K\rangle)}{2^m}
    =
    \frac{\phi(\langle K\rangle)}{2^m}.
  \end{align*}
  Thus~\eqref{eq:orderUInducecolors:secondobjective} holds.
\end{proof}

\section{Separations}
\label{sec:separations}

In this section we prove all separation theorems.

Recall from Section~\ref{sec:useful} that for $x\in\cE_n$, $\sigma_x\in S_n$ denotes the unique permutation
such that $x_{\{\sigma_x^{-1}(1)\}} < \cdots < x_{\{\sigma_x^{-1}(n)\}}$ when the coordinates $(x_{\{i\}} \mid
i\in[n])$ are distinct, and is defined arbitrarily otherwise.

\begin{proofof}{Theorem~\ref{thm:separationUCoupleIndependence}}
  First note that the quasirandom $(\ell+1)$-tournamon $\psi_{\ell+1}$ can be represented by the
  $\TkTournament[(\ell+1)]$-on
  \begin{align} \label{eq:theonquasirandom}
    \cN
    & \df
    \set{x\in\cE_{\ell+1}}{
      x_{[\ell+1]} < \frac{1}{2}
      \equiv
      \sgn(\sigma_x)=1}.
  \end{align}
  Let $\rn{K}$ be the exchangeable array corresponding to $\cN$ with respect to $\rn{\theta}$ picked in
  $\cE_{\NN_+}$. By Theorem~\ref{thm:UCouple}, to show that $\psi_{\ell+1}\in\UCouple[\ell]$ it is sufficient to prove that
  $\psi_{\ell+1}$ is weakly $\ell$-independent, that is for every $m\in\NN$,
  the random variable
  $\rn{K}\rest_{[m]}$ is independent from $(\rn{\theta}_A \mid
  A\in r(m,\ell))$. Indeed, $\rn{K}\rest_{[m]}$ is completely determined
  by $\sigma_{\iota_{[m]}^*(\rn{\theta})}$ and $(\rn{\theta}_A \mid A\in\binom{[m]}{\ell+1})$, and any changes in the values of the signs
 $\sgn(\sigma_{\iota_A^*(\rn{\theta})})$
 can be offset by flipping the corresponding variables $\rn{\theta}_A$ (cf.~\eqref{eq:theonquasirandom}) so that the
 {\em distribution} of $\rn{K}\rest_{[m]}$ does not change from fixing
 $\sigma_{\iota_{[m]}^*(\rn{\theta})}$.

  \medskip

  Suppose now toward a contradiction that $\psi_{\ell+1}\in\Independence[\ell]$, that is
  $\psi_{\ell+1} = \phi_\cH$ for some $\TkTournament[(\ell+1)]$-on $\cH$ of the form $\cH =
  \cE_{\ell+1,\ell}\times\cG$ for some $\cG\subseteq [0,1]$. Note that for any $\sigma\in S_{\ell+1}$, we have
  $\cH\cdot\sigma = \cH$. But this is a contradiction as the axioms of $\TkTournament$ imply
  that $\lambda((\cH\cdot\sigma)\cap\cH) = 0$ whenever $\sgn(\sigma)=-1$.
\end{proofof}

\begin{proofof}{Theorem~\ref{thm:separationUInduceUCouple}}
  Since the linear order $\psi\in\HomT{\TLinOrder}$ is represented by the $\TLinOrder$-on $\cN\df\{x\in\cE_2
  \mid x_{\{1\}} < x_{\{2\}}\}$, we know $\rk(\psi)=1$, thus by Proposition~\ref{prop:rankbasic}, we have
  $\psi\notin\UCouple[1]$.

  Since $\psi$ is $n$-categorical for every $n\in\NN$, it is symmetrically $\ell$-local for trivial
  reasons (namely, all events $\rn{K}\rest_{V_i}\cong M_i$ have probability $1$), for any integer $\ell$. Hence
  $\psi\in\UInduce[\ell]$ by Theorem~\ref{thm:UInduce}.
\end{proofof}

To prove Theorems~\ref{thm:separationUCoupleIndependencehypergraph}
and~\ref{thm:separationUInduceUCouplehypergraph}, the alternating tournament defined below will play a key
role.

\begin{definition}
  Let $k\geq 1$. For $\alpha\injection{[k]}{[k+1]}$, denote by $\sigma_\alpha$
  the unique extension of $\alpha$ to an element of $S_{k+1}$, and let
  $\sgn(\alpha)\df \sgn(\sigma_\alpha)$. This definition behaves well with
  respect to the actions of $S_k$ and $S_{k+1}$: for every $\eta\in S_k$ we
  have $\sgn(\alpha\circ\eta)=\sgn(\alpha)\sgn(\eta)$, and for every $\sigma\in
  S_{k+1}$ we have $\sgn(\sigma\circ\alpha) = \sgn(\sigma)\sgn(\alpha)$.

  The \emph{alternating $k$-tournament} is the model $A^{(k)}_{k+1}\in\cK_{k+1}[\TkTournament]$ of
  $\TkTournament$ of size $k+1$ given by
  \begin{align*}
    V(A^{(k)}_{k+1})
    & \df
    [k+1];
    &
    R_E(A^{(k)}_{k+1})
    & \df
    \{\alpha\in([k+1])_k \mid \sgn(\alpha)=1\}.
  \end{align*}
  For example, $A^{(2)}_3$ is the oriented cycle $\vec C_3$.
\end{definition}

\begin{proofof}{Theorem~\ref{thm:separationUCoupleIndependencehypergraph}}
  For this proof, let us denote the predicate symbols of $\TkHypergraph[(\ell+2)]$ and
  $\TkTournament[(\ell+1)]$ by $E$ and $P$, respectively. Let
  $\psi\df\psi_{\ell+1}\in\HomT{\TkTournament[(\ell+1)]}$ be the quasirandom~$(\ell+1)$-tournamon and let
  $I\interpret{\TkHypergraph[(\ell+2)]}{\TkTournament[(\ell+1)]}$ be
  given by
  \begin{align*}
    I(E)(x_1,\ldots,x_{\ell+2})
    & \df
    \bigvee_{1\leq i_1 < \cdots < i_\ell\leq\ell+2} (P(x_{i_1},\ldots,x_{i_\ell},x_{j_1})\equiv P(x_{i_1},\ldots,x_{i_\ell},x_{j_2})),
  \end{align*}
  where $j_1,j_2\in[\ell+2]$ are such that $\{i_1,\ldots,i_\ell,j_1,j_2\}=[\ell+2]$. By
  Theorems~\ref{thm:naturality} and~\ref{thm:separationUCoupleIndependence}, we know that
  $\phi\df\psi^I\in\HomT{\TkHypergraph[(\ell+2)]}$ satisfies $\UCouple[\ell]$.

  To show that $\phi\notin\Independence[\ell]$, we will make use of the theory $T$
  (isomorphic to $\TkTournament[(\ell+1)]$) that is obtained from
  $\TkHypergraph[(\ell+2)]\cup\TkTournament[(\ell+1)]$ by adding the axiom
  \begin{align} \label{eq:addedaxiom}
    \forall\vec{x}, & E(x_1,\ldots,x_{\ell+2})\equiv I(E)(x_1,\ldots,x_{\ell+2})
  \end{align}
  and the commutative diagram
  \begin{equation*}
    \begin{tikzcd}
      \TkHypergraph[(\ell+2)]\arrow[r, "I"]\arrow[d, "S"'] & \TkTournament[(\ell+1)]\arrow[d, "J"]\\
      \TkHypergraph[(\ell+2)]\cup\TkTournament[(\ell+1)]\arrow[r, "A"'] & T
    \end{tikzcd}
  \end{equation*}
  where $S$ is the structure-erasing interpretation, $A$ is the axiom-adding interpretation and $J$ is the
  isomorphism mentioned above that acts identically on $P$ (the inverse $J^{-1}$ acts identically on $P$ and acts as $I$ on
  $E$). Let $\xi = \psi^{J^{-1}}$ so that $\psi = \xi^J$ and $\phi = \xi^{A\comp S}$.

  Suppose toward a contradiction that $\phi\in\Independence[\ell]$ and let $\cN$ be an $\ell$-independent
  $\TkHypergraph[(\ell+2)]$-on over $\Omega$ such that $\phi_\cN = \phi = \psi^I$. By Proposition~\ref{prop:theonalignment},
  there exists a $T$-on $\cN'$ over $\Omega\times\Omega$ such that $\phi_{\cN'} = \xi$ and $S(A(\cN'))_E = \cN'_E =
  \cN_E\times\cE_{\ell+2}$ a.e. Note that $\rk(\phi)\leq\rk(\psi)\leq\ell+1$, so by possibly changing
  zero-measure sets using Proposition~\ref{prop:rankae}, we may also suppose that $\rk(\cN')\leq\ell+1$. By
  applying a measure-isomorphism between $\Omega\times\Omega$ and $[0,1]$, we conclude that there exists a $T$-on
  $\cH$ (over $[0,1]$) such that $\phi_\cH = \xi$, $\rk(\cH)\leq\ell+1$ and the peon $\cH_E$ is
  $\ell$-independent.

  Since $\cH_E$ has rank at most $\ell+1$ and is $\ell$-independent, we can write it as $\cH_E =
  \cE_{\ell+2,\ell}\times\cG\times [0,1]^{\{\ell+2\}}$ for some measurable
  $\cG\subseteq[0,1]^{\binom{[\ell+2]}{\ell+1}}$. Using the symmetry axiom~\eqref{eq:hypergraphsymmetry} of
  $\TkHypergraph[(\ell+2)]$ and making a zero-measure change in $\cG$, we may assume that it is
  $S_{\ell+2}$-invariant.

  For every $t\in[\ell+2]$, define the sets
  \begin{align*}
    V_t^{\ell+1} & \df \set{A\in\binom{[\ell+2]}{t}}{\ell+1\in A\land\ell+2\notin A};\\
    V_t^{\ell+2} & \df \set{A\in\binom{[\ell+2]}{t}}{\ell+1\notin A\land\ell+2\in A};\\
    V_t^{\ell+1,\ell+2} & \df \set{A\in\binom{[\ell+2]}{t}}{\ell+1,\ell+2\in A}.
  \end{align*}
  Define also the sets
  \begin{align*}
    W^{\ell+1}_t & \df [0,1]^{V_t^{\ell+1}}; &
    W^{\ell+2}_t & \df [0,1]^{V_t^{\ell+2}}; &
    W^{\ell+1,\ell+2}_t & \df [0,1]^{V_t^{\ell+1,\ell+2}};
    \\
    Y^{\ell+1} & \df \prod_{t=1}^\ell W^{\ell+1}_t; &
    Y^{\ell+2} & \df \prod_{t=1}^\ell W^{\ell+2}_t; &
    Z & \df \prod_{t=1}^{\ell+2} W^{\ell+1,\ell+2}_t.
  \end{align*}
  Note that
  \begin{align*}
    \cE_{\ell+1}
    & =
    \cE_\ell\times Y^{\ell+1}\times W^{\ell+1}_{\ell+1};
    \\
    \cE_{\ell+2}
    & =
    \cE_\ell\times Y^{\ell+1}\times W^{\ell+1}_{\ell+1}\times Y^{\ell+2}\times W^{\ell+2}_{\ell+1}\times Z.
  \end{align*}

  Let $\iota\function{[\ell]\cup\{\ell+2\}}{[\ell+1]}$ be the function that maps $\ell+2$ to $\ell+1$ and
  fixes all other points and note that $\iota$ induces maps $\iota^*\function{Y^{\ell+1}}{Y^{\ell+2}}$ and
  $\iota^*_{\ell+1}\function{W^{\ell+1}_{\ell+1}}{W^{\ell+2}_{\ell+1}}$ (given by $\iota^*(y)_A\df
  y_{\iota(A)}$ and $\iota^*_{\ell+1}(w)_A\df w_{\iota(A)}$).

  For every $x\in\cE_\ell$ and every $w\in W^{\ell+1}_{\ell+1}$, define the sections
  \begin{align*}
    \cH^\alpha_P(x,w) & \df \{y\in Y^{\ell+1} \mid (x,y,w)\in\cH_P\};\\
    \cH^\beta_P(x,w) & \df \{y\in Y^{\ell+1} \mid (x,y,w)\notin\cH_P\};
  \end{align*}
  and for every $x\in\cE_\ell$, define
  \begin{align*}
    \cH^\alpha_P(x) & \df \{w\in W^{\ell+1}_{\ell+1} \mid \lambda(\cH^\alpha_P(x,w)) > 0\};\\
    \cH^\beta_P(x) & \df \{w\in W^{\ell+1}_{\ell+1} \mid \lambda(\cH^\beta_P(x,w)) > 0\}.
  \end{align*}
  It is clear that
  \begin{align}\label{eq:tournamentinherited}
    \cH^\alpha_P(x)\cup\cH^\beta_P(x) & = W^{\ell+1}_{\ell+1}
  \end{align}
  for every $x\in\cE_\ell$.

  Note that the axiom~\eqref{eq:addedaxiom} of $T$ and an application of Fubini's Theorem imply that for
  a.e.~$x\in\cE_\ell$, a.e.~$w,\widehat{w}\in W^{\ell+1}_{\ell+1}$, a.e.~$y\in\cH^\alpha_P(x,w)$,
  a.e.~$\widehat{y}\in\cH^\alpha_P(x,\widehat{w})$ and a.e.~$z\in Z$, we have
  \begin{align}\label{eq:addedaxiomtheon}
    (x,y,w,\iota^*(\widehat{y}),\iota^*_{\ell+1}(\widehat{w}),z)\in\cH_E.
  \end{align}
  Since the definition of $I(P)$ is invariant under negating $P$, the same assertion
  also holds with $\beta$ in place of $\alpha$.

  Recalling that $\cH_E=\cE_{\ell+2,\ell}\times\cG\times [0,1]^{\{\ell+2\}}$, \eqref{eq:addedaxiomtheon}
  implies that for a.e.~$x\in\cE_\ell$, a.e.~$w,\widehat{w}\in\cH^\alpha_P(x)$ and a.e.~$z\in
  W^{\ell+1,\ell+2}_{\ell+1}$, we have
  \begin{align}\label{eq:addedaxiomtheonreduced}
    (w,\iota^*_{\ell+1}(\widehat{w}),z)\in\cG.
  \end{align}
  Again, the analogous statement with $\beta$ in place of $\alpha$ also holds.

  From~\eqref{eq:tournamentinherited} and~\eqref{eq:addedaxiomtheonreduced}, it follows that there exists
  $x_0\in\cE_\ell$ such that the following hold for $W^\alpha\df\cH^\alpha_P(x_0)$ and
  $W^\beta\df\cH^\beta_P(x_0)$.
  \begin{enumerate}
  \item We have $W^\alpha\cup W^\beta = W^{\ell+1}_{\ell+1}$.
  \item For a.e.~$w,\widehat{w}\in W^\alpha$ and a.e.~$z\in W^{\ell+1,\ell+2}_{\ell+1}$, we have
    $(w,\iota^*_{\ell+1}(\widehat{w}),z)\in\cG$.
  \item For a.e.~$w,\widehat{w}\in W^\beta$ and a.e.~$z\in W^{\ell+1,\ell+2}_{\ell+1}$, we have
    $(w,\iota^*_{\ell+1}(\widehat{w}),z)\in\cG$.
  \end{enumerate}

  Since $\lvert V_{\ell+1}^{\ell+1}\rvert=1$, let us for simplicity identify $W^{\ell+1}_{\ell+1}$ with $[0,1]$ and let $h\df\One_{W^\alpha}$ be the indicator function
  of $W^\alpha\subseteq [0,1]$. For every $A\in\binom{[\ell+2]}{\ell+1}$, let
  $\pi_A\function{[0,1]^{\binom{[\ell+2]}{\ell+1}}}{[0,1]}$ be the projection on the $A$-th coordinate and
  note that the properties above imply that for a.e.~$u\in[0,1]^{\binom{[\ell+2]}{\ell+1}}$, if
  $h(\pi_{[\ell+1]}(u))=h(\pi_{[\ell]\cup\{\ell+2\}})$, then $u\in\cG$. Since $\cG$ is $S_{\ell+2}$-invariant,
  this in turn implies that for a.e.~$u\in[0,1]^{\binom{[\ell+2]}{\ell+1}}$, if there exist
  $j_1,j_2\in[\ell+2]$ distinct such that $h(\pi_{[\ell+2]\setminus\{j_1\}}(u)) =
  h(\pi_{[\ell+2]\setminus\{j_2\}})$, then $u\in\cG$. But since at least two of the values
  $h(\pi_{[\ell+1]}(u))$, $h(\pi_{[\ell+2]\setminus\{\ell+1\}}(u))$ and
  $h(\pi_{[\ell+2]\setminus\{\ell\}}(u))$ must be equal, it follows that $\lambda(\cG) = 1$. So we must have
  \begin{align*}
    \phi(\rho_{\ell+2}) & = \lambda(\cH_E) = \lambda(\cG) = 1,
  \end{align*}
  which implies $\phi(\overline{K}^{(\ell+2)}_{\ell+2}) = 0$.

  However, note that for the alternating $(\ell+1)$-tournament $A^{(\ell+1)}_{\ell+2}$, we have
  $I(A^{(\ell+1)}_{\ell+2})\cong\overline{K}^{(\ell+2)}_{\ell+2}$, hence
  \begin{align*}
    \phi(\overline{K}^{(\ell+2)}_{\ell+2})
    & \geq
    \psi(A^{(\ell+1)}_{\ell+2})
    =
    \frac{(\ell+2)!}{2^{\ell+2}\lvert\Aut(A^{(\ell+1)}_{\ell+2})\rvert}
    =
    \frac{1}{2^{\ell+1}},
  \end{align*}
  a contradiction.
\end{proofof}

The following is needed for the proof of Theorem~\ref{thm:separationUInduceUCouplehypergraph}.

\begin{lemma}\label{lem:twoalternating}
  If $M\in\cM_{k+2}[\TkTournament]$ is a $k$-tournament on $k+2$ vertices, then $M$ has at most two
  (unlabeled) copies of the alternating $k$-tournament $A^{(k)}_{k+1}$.
\end{lemma}

\begin{proof}
  Suppose toward a contradiction that $M\in\cM_{k+2}[\TkTournament]$ contains three copies of $A^{(k)}_{k+1}$
  and without loss of generality, let us suppose that these three copies are induced by $V_1\df[k+1]$,
  $V_2\df[k]\cup\{k+2\}$ and $V_3\df[k-1]\cup\{k+1,k+2\}$. Let
  $\alpha_{12},\alpha_{13},\alpha_{23}\in([k+2])_k$ be given by
  \begin{align*}
    \alpha_{12}(v) & \df v;
    &
    \alpha_{13}(v) & \df
    \begin{dcases*}
      v, & if $v < k$;\\
      k+1 & if $v = k$;
    \end{dcases*}
    &
    \alpha_{23}(v) & \df
    \begin{dcases*}
      v, & if $v < k$;\\
      k+2 & if $v = k$;
    \end{dcases*}
  \end{align*}
  and note that $\im(\alpha_{ij})=V_i\cap V_j$.

  But then $M\rest_{V_1}\cong A^{(k)}_{k+1}$, $M\rest_{V_2}\cong A^{(k)}_{k+1}$ and $M\rest_{V_3}\cong
  A^{(k)}_{k+1}$ imply respectively that
  \begin{align*}
    \alpha_{12}\in R_E(M) & \equiv \alpha_{13}\notin R_E(M),\\
    \alpha_{12}\in R_E(M) & \equiv \alpha_{23}\notin R_E(M),\\
    \alpha_{13}\in R_E(M) & \equiv \alpha_{23}\notin R_E(M).
  \end{align*}
  This is a contradiction as all three equivalences above cannot be true at the same time.
\end{proof}

\begin{proofof}{Theorem~\ref{thm:separationUInduceUCouplehypergraph}}
  For this proof, let us again denote the predicate symbols of $\TkHypergraph[(\ell+2)]$ and
  $\TkTournament[(\ell+1)]$ by $E$ and $P$, respectively. For $p\in[0,1]$, let $\cN^p$ be
the
  $\TkTournament[(\ell+1)]$-on given by
  \begin{align*}
    \cN^p_E
    & \df
    \set{x\in\cE_{\ell+1}}{
      x_{[\ell+1]} < p
      \equiv
      \sgn(\sigma_x)=1
    }
  \end{align*}
(note that for $p=1/2$ this is precisely the theon~\eqref{eq:theonquasirandom}
representing the quasirandom $(\ell+1)$-tournamon).

Let $I\interpret{\TkHypergraph[(\ell+2)]}{\TkTournament[(\ell+1)]}$ be the
interpretation that declares $(\ell+2)$-edges to be isomorphic copies of $A_{\ell+2}^{(\ell+1)}$, and let
$\phi_p\df\phi_{\cN^p}^I\in\HomT{\TkHypergraph[(\ell+2)]}$. We will show that
$\phi_p$ satisfies
  $\UInduce[\ell]$ for every $p\in[0,1]$, but does not satisfy
  $\UCouple[1]$ unless $p\in\{0,1/2,1\}$.

  To show the former, recall that the quasirandom $(\ell+1)$-hypergraphon
  $\psi_{\ell+1,p}\in\HomT{\TkHypergraph[(\ell+1)]}$ satisfies $\Independence[\ell]$
  (cf.~Lemma~\ref{lem:quasirandomcoloredhypergraphons}) and hence $\UCouple[\ell]$
  (by Theorem~\ref{thm:inter-properties}). Note also that $\phi_{\cN^p} =
  (\psi_{\ell+1,p}\otimes\psi)^{I'}$, where $\psi\in\HomT{\TLinOrder}$ is the linear order and
  $I'\interpret{\TkTournament[(\ell+1)]}{\TkHypergraph[(\ell+1)]\cup\TLinOrder}$ is given by\footnote{This is a
    generalization of the ``arc-orientation'' interpretation used implicitly in the implications
    $P_{10}\Longrightarrow P_{11}\Longrightarrow P_1(s)$ of~\cite{CG91}.}
  \begin{align*}
    I'(P)(x_1,\ldots,x_{\ell+1})
    & \df \left(\bigwedge_{1\leq i<j\leq \ell+1} x_i\neq x_j\right) \\ & \land
    \left(E(x_1,\ldots,x_{\ell+1})\equiv
    \bigvee_{\substack{\sigma\in S_{\ell+1}\\\sgn(\sigma)=1}}\bigwedge_{1\leq i < j\leq \ell+1}
    x_{\sigma(i)}\prec x_{\sigma(j)}\right).
  \end{align*}
  By Theorem~\ref{thm:UCouple}\ref{thm:UCouple:UCouple}$\implies$\ref{thm:UCouple:orderUInduce}, we know that
  $\psi_{\ell+1,p}\otimes\psi\in\UInduce[\ell]$ and by Theorem~\ref{thm:naturality}, we get that $\phi_p =
  (\psi_{\ell+1,p}\otimes\psi)^{I'\comp I}$ satisfies $\UInduce[\ell]$.

  \medskip

  Let us now show that for every $p\in(0,1)\setminus\{1/2\}$, $\phi_p$ does not satisfy $\UCouple[1]$. Since
  the linear order $\psi\in\HomT{\TLinOrder}$ has rank $1$, it is enough to show that $\phi_p$ is not uniquely
  coupleable with $\psi$. Consider the $(\TkTournament[(\ell+1)]\cup\TLinOrder)$-on $\cN^{p,<}$ given by
  \begin{align*}
    \cN^{p,<}_P & \df \cN^p_P; &
    \cN^{p,<}_\prec & \df \{x\in\cE_2 \mid x_{\{1\}} < x_{\{2\}}\}
  \end{align*}
  and note that $\phi_{\cN^{p,<}}$ is a coupling of $\phi_{\cN^p}$ and $\psi$, hence
  $\xi\df\phi_{\cN^{p,<}}^{I\cup\id_{\TLinOrder}}$ is a coupling of $\phi_p$ and $\psi$.
We will show that $\xi\neq \phi_p\otimes\psi$ by a direct computation
exhibiting an $(\ell+2)$-hypergraph $H$ and two different orders on it such
that $\xi(H_1) \neq \xi(H_2)$ for the corresponding models of the theory
$\TkHypergraph[(\ell+2)]\cup\TLinOrder$. That will suffice since, clearly,
$(\phi_p\otimes\psi)(H_1) = (\phi_p\otimes\psi)(H_2)$.

  Let
  $H\in\cK_{\ell+3}[\TkHypergraph[(\ell+2)]]$ be given by
  \begin{align*}
    V(H) & \df [\ell+3]; &
    E(H) & \df \{[k+1],[k]\cup\{k+2\}\};
  \end{align*}
  and let $H_1,H_2\in\cK_{\ell+3}[\TkHypergraph[(\ell+2)]\cup\TLinOrder]$ be obtained from $H$ by equipping it
  with the orders $\prec_1$ and $\prec_2$, respectively, where $\prec_1$ is the natural order of $[\ell+3]$
  and $\prec_2$ is obtained from $\prec_1$ by swapping the order position of $\ell+1$ and $\ell+3$, that is,
  we have
  \begin{align*}
    1 \prec_2 2 \prec_2 \cdots \prec_2 \ell \prec_2 \ell+3 \prec_2 \ell+2 \prec_2 \ell+1.
  \end{align*}

  Let $\rn{\theta}$ be picked at random in $\cE_{\NN_+}$ according to $\lambda$ and let $\rn{K}$ be the
  exchangeable array corresponding to $\cN^{p,<}$ with respect to $\rn{\theta}$ (so that
  $(I\cup\id_{\TLinOrder})(\rn{K})$ corresponds to $(I\cup\id_{\TLinOrder})(\cN^{p,<})$). Let
  $\rn{\sigma}\df \sigma_{\iota_{[\ell+3]}^*(\rn{\theta})}$. Then we have
  \begin{align*}
    \xi(\langle H_1\rangle)
    & =
    \PP[I(J(\rn{K}\rest_{[\ell+3]})) = H\land\rn{\sigma} = \id_{\ell+3}];
    \\
   \xi(\langle H_2\rangle)
    & =
    \PP[I(J(\rn{K}\rest_{[\ell+3]})) = H\land\rn{\sigma} = \tau];
  \end{align*}
  where $J\interpret{\TkTournament[(\ell+1)]}{\TkTournament[(\ell+1)]\cup\TLinOrder}$ is the structure-erasing
  interpretation and $\tau$ is the transposition that swaps $\ell+1$ and $\ell+3$. Then by
  Lemma~\ref{lem:twoalternating}, $I(J(\rn{K}\rest_{[\ell+3]})) = H$ is equivalent to
  \begin{align}\label{eq:IJKtwoalternating}
    J(\rn{K}\rest_{[\ell+2]}) \cong J(\rn{K}\rest_{[\ell+1]\cup\{\ell+3\}}) \cong A^{(\ell+1)}_{\ell+2}.
  \end{align}

  Since $\Aut(A^{(\ell+1)}_{\ell+2})$ is the alternating group on $[\ell+2]$, on any fixed set of $\ell+2$
  vertices, there are exactly two models $M_1$ and $M_2$ that are isomorphic to $A^{(\ell+1)}_{\ell+2}$ and
  they satisfy $R_P(M_1)\cap R_P(M_2) = \emptyset$. This means that on the event~\eqref{eq:IJKtwoalternating},
  out of the a priori four presentations of $A^{(\ell+1)}_{\ell+2}$ induced on $[\ell+2]$ and
  $[\ell+1]\cup\{\ell+3\}$, only two are actually possible. Since $\ell$ is odd, a straightforward calculation
  gives
  \begin{align*}
   \xi(\langle H_1\rangle)
    & =
    p^{(\ell+2)}(1-p)^{\ell+1} + p^{\ell+1}(1-p)^{\ell+2}
    =
    p^{\ell+1}(1-p)^{\ell+1};
    \\
    \xi(\langle H_2\rangle)
    & =
    p^\ell(1-p)^{\ell+3} + p^{\ell+3}(1-p)^\ell
    =
    p^\ell(1-p)^\ell(3p^2 - 3p + 1).
  \end{align*}

  Thus we get
  \begin{align*}
    \xi(\langle H_2\rangle)
    - \xi(\langle H_1\rangle)
    & =
    p^\ell(1-p)^\ell(4p^2 - 4p + 1)
    \\
    & =
    p^\ell(1-p)^\ell\left(2p - 1\right)^2,
  \end{align*}
  which is non-zero as long as $p\in(0,1)\setminus\{1/2\}$.
\end{proofof}

\begin{proofof}{Theorem~\ref{thm:separationDevUInduce}}
  For $p\in(0,1)$, let $\cN$ be the $\TkHypergraph$-on given by
  \begin{align*}
    \cN
    \df
    \biggl\{x\in\cE_k
    & \bigg\vert\: (\min\{x_{\{v\}} \mid v\in[k]\} < 1/2\land x_{[k]} < p)
    \\
    & \; \lor\biggl(\min\{x_{\{v\}} \mid v\in[k]\}\geq 1/2\land\sum_{v\in[k]} x_{[k]\setminus\{v\}} \bmod 1 < p\biggr)
    \biggr\}.
  \end{align*}
  Let us show that $\phi\df\phi_\cN$ satisfies $\Dev[k-1]$; recall that $\Dev[k-1]=\Disc[\cA_{k-1}]$, where
  $\cA_{k-1}\df\{A\in\binom{[k]}{k-1} \mid \{1\}\subseteq A\} =
  \binom{[k]}{k-1}\setminus\{[k]\setminus\{1\}\}$ (see Definition~\ref{def:CliqueDisc}) and for
  $\psi\in\HomT{T_{\cL_{\cA_{k-1}}}}$, let $\xi$ be a coupling of $\phi$ and $\psi$. By
  Proposition~\ref{prop:theonalignment}, there exists a $(T\cup T_{\cL_{\cA_{k-1}}})$-on $\cH$ over $[0,1]^2$
  such that $\phi_\cH = \xi$ and $\cH_E = \cN\times\cE_k$.

  Let $(\rn{\theta^1},\rn{\theta^2})$ be picked in $\cE_{\NN_+}([0,1]^2)$ according to $\lambda$ and
  let $\rn{K}$ be the exchangeable array corresponding to $\cH$ with respect to
  $(\rn{\theta^1},\rn{\theta^2})$. Our objective is to show that the events $(1,2,\ldots,k)\in R_E(\rn K)$
  and $\forall A\in \cA_{k-1},\iota_A\in R_{P_A}(\rn K)$ are independent.

  Since the event $\iota_A\in R_{P_A}(\rn{K})$ is completely determined by
  $((\rn{\theta^1}_B,\rn{\theta^2}_B) \mid B\subseteq A)$, it is sufficient to show that the event
  $(1,\ldots,k)\in R_E(\rn{K})$ is independent from $((\rn{\theta^1}_B,\rn{\theta^2}_B) \mid B\in
  r(k,k-1)\land B\neq[k]\setminus\{1\})$. But the event $(1,\ldots,k)\in R_E(\rn{K})$ is equivalent to
  $(\rn{\theta^1}_B)_{B\in r(k)}\in\cN$, and it is easy to see that the conditional probability of
  $(1,\ldots,k)\in R_E(\rn{K})$ given $((\rn{\theta^1}_B,\rn{\theta^2}_B) \mid B\in r(k,k-1)\land
  B\neq[k]\setminus\{1\})$ is $p$ a.e. Hence $\phi$ satisfies $\Dev[k-1]$.

  \medskip

  Let us now show that $\phi$ does not satisfy $\UInduce[1]$. To do so, for each $i\in[2]$, we consider the
  $(\TkHypergraph\cup\TcColoring[2])$-on $\cH^i$ (see Remark~\ref{rmk:UInduce1}) given by
  \begin{align*}
    \cH_E & = \cN;\\
    \cH_{\chi_i} & = \{x\in\cE_1 \mid x_{\{1\}} < 1/2\};\\
    \cH_{\chi_{3-i}} & = \{x\in\cE_1 \mid x_{\{1\}}\geq 1/2\}.
  \end{align*}

  Then by a straightforward calculation, for every $H\in\cM[\TkHypergraph\cup\TcColoring[2]]$ with
  $R_{\chi_1}(H) = V(H)$, we have
  \begin{align*}
    \phi_{\cH^1}(H) & = \frac{\psi_{k,p}(I(H))}{2^{\lvert H\rvert}}; &
    \phi_{\cH^2}(H) & = \frac{\phi_{\cN'}(I(H))}{2^{\lvert H\rvert}};
  \end{align*}
  where $I\interpret{\TkHypergraph}{\TkHypergraph\cup\TcColoring[2]}$ is the structure-erasing interpretation,
  $\psi_{k,p}$ is the quasirandom $k$-hypergraphon (see Definition~\ref{def:hypergraphons}) and $\cN'$ is the
  $\TkHypergraph$-on given by
  \begin{align*}
    \cN' & = \set{x\in\cE_k}{\sum_{v\in[k]} x_{[k]\setminus\{v\}} \bmod 1 < p}.
  \end{align*}
  Since $\phi_{\cN'}\neq\psi_{k,p}$ (since $\rk(\psi_{k,p}) = k > k - 1\geq\rk(\psi_{\cN'})$), it follows that
  $\phi_{\cH^1}(H)\neq\phi_{\cH^2}(H)$ for some $H\in\cM[\TkHypergraph\cup\TcColoring[2]]$ with $R_{\chi_1}(H)
  = V(H)$, hence $\phi$ does not satisfy $\UInduce[1]$.
\end{proofof}

\begin{proofof}{Theorem~\ref{thm:separationIndependenceDisc}}
  For $p\in(0,1)$, let $\cN$ be the $\TkHypergraph$-on given by
  \begin{align*}
    \cN
    & \df
    \set{x\in\cE_k}{\max\set{x_A}{A\in\binom{[k]}{\ell+1}} < p}.
  \end{align*}
  It is clear that $\phi\df\phi_\cN$ satisfies $\Independence[\ell]$. Consider now the
  $T_{\cL_{\{[\ell+1]\}}}$-on $\cH$ given by
  \begin{align*}
    \cH_E & \df \cN; &
    \cH_{P_{[\ell+1]}} & \df \{x\in\cE_{\ell+1} \mid x_{[\ell+1]}\geq p\}
  \end{align*}
  and note that if $\rn{K}$ is the exchangeable array corresponding to $\cH$,
  then
  \begin{multline*}
    \PP[(1,\ldots,k)\in R_E(\rn{K})\land (1,\ldots,\ell+1)\in R_{P_{[\ell+1]}}(\rn{K})]
    =
    0
    \\
    \neq
    p^{\binom{k}{\ell+1}}\cdot (1-p)
    =
    \phi(\rho_k)\cdot\PP[(1,\ldots,\ell+1)\in R_{P_{[\ell+1]}}(\rn{K})],
  \end{multline*}
  so $\phi$ does not satisfy $\Disc[\{[\ell+1]\}]$.
\end{proofof}

\begin{proofof}{Theorem~\ref{thm:separationupward}}
  Follows from Theorems~\ref{thm:UInduce->CliqueDisc} ($\UInduce[\ell+1]\implies\CliqueDisc[\ell+1]$)
  and~\ref{thm:separationIndependenceDisc} ($\Independence[\ell]\nRightarrow\Disc[\{[\ell+1]\}]$), and the fact
  that $\CliqueDisc[\ell+1]\implies\Disc[\{[\ell+1]\}]$ (see~\cite{Tow17,ACHP18}).
\end{proofof}

\section{Top level quasirandomness}
\label{sec:fullQR}

In this section we prove Theorems~\ref{thm:fullIndependence} and~\ref{thm:fullUCouple}, which completely
characterize the properties $\Independence[k-1]$ and $\UCouple[k-1]$, respectively when all arities are at
most $k$. These can be seen as analogues of full quasirandomness for arbitrary universal theories (just as
$\Dev[k]=\CliqueDisc[k-1]=\Disc[\binom{[k]}{k-1}]$ gives full quasirandomness in $\TkHypergraph$).

\begin{proofof}{Theorem~\ref{thm:fullIndependence}}
  By Lemma~\ref{lem:quasirandomcoloredhypergraphons}, $\psi_{k,p}\in\HomT{\TcColkHyp}$ satisfies
  $\Independence[k-1]$, so the backward direction follows from Theorem~\ref{thm:naturality}.

  For the forward direction, first we claim that it is enough to show the case when $T = T_\cL$. (This is not
  completely immediate as $I\interpret{T}{\TcColkHyp}$ is required to satisfy
  $\TcColkHyp\vdash\forall\vec{x},I(F)(\vec{x})$ for every axiom $\forall\vec{x},F(\vec{x})$ of $T$.) Let
  $A\interpret{T_\cL}{T}$ be the axiom-adding interpretation and suppose $\phi^A$ (which satisfies
  $\UCouple[k-1]$ by Theorem~\ref{thm:naturality}) can be written as $\phi^A = \psi_{k,p}^J$ for some $c\geq
  2$, some $p\in\Pi_c$ and some $J\interpret{T_\cL}{\TcColkHyp}$, then we define $I\interpret{T}{\TcColkHyp}$
  to act as $J$ and we have to show that it is indeed an interpretation, i.e., that $\TcColkHyp\vdash\forall\vec{x},I(F)(\vec{x})$ for every axiom
  $\forall\vec{x},F(\vec{x})$ of $T$ ($\psi_{k,p}^I = \phi$ will then follow trivially). Equivalently, we have
  to show that if $M\in\cM[\TcColkHyp]$, then $J(M)\in\cM[T]$. But since all $p_i$ are positive, we have
  $\psi_{k,p}(M) > 0$, so $\phi^A(J(M)) > 0$, hence trivially $J(M)\in\cM[T]$.

  \medskip

  Let us now prove the case $T = T_\cL$. Let $\cN$ be a $(k-1)$-independent $T_\cL$-on such that
  $\phi_\cN=\phi$. Note that if $P\in\cL$ is such that $k(P)\leq k-1$, then $\cN_P$ must be either $\emptyset$
  or $\cE_{k(P)}$, so we can write $\cL=\cL'\cup\cL_0\cup\cL_1$, where
  \begin{align*}
    \cL' & \df \{P\in\cL \mid k(P) = k\};\\
    \cL_0 & \df \{P\in\cL \mid k(P)\leq k-1\land\cN_P=\emptyset\};\\
    \cL_1 & \df \{P\in\cL \mid k(P)\leq k-1\land\cN_P=\cE_{k(P)}\}.
  \end{align*}

  Recall from Definition~\ref{def:categorical} that $\cK_k[\Th(\phi)] = \{K\in\cK_k[T_\cL] \mid \phi(K) > 0\}$ and
  enumerate its elements as $K_1,\ldots,K_c$. Note that since $\cN$ is $(k-1)$-independent, it follows that
  every peon $\cN_P$ with $P\in\cL'$ is $S_k$-invariant, hence we must have $\Aut(K_i)=S_k$ for every
  $i\in[c]$. Suppose first that $c\geq 2$ and define $p\in\Pi_c$ by $p_i = \phi(K_i) > 0$ and let
  $I\interpret{T_\cL}{\TcColkHyp}$ be given by
  \begin{align}
    I(P)(x_1,\ldots,x_{k(P)}) & \df x_1 \neq x_1 & (P\in\cL_0);
    \notag\\
    I(P)(x_1,\ldots,x_{k(P)}) & \df \bigwedge_{1\leq i < j\leq k(P)} x_i \neq x_j & (P\in\cL_1);
    \notag\\
    I(P)(x_1,\ldots,x_k)
    & \df
    \bigvee_{\substack{i\in[c]\\ \id_k\in R_P(K_i)}} E_i(x_1,\ldots,x_k).
    & (P\in\cL').
    \label{eq:fullIndependence:topI}
  \end{align}
  Since $\cN$ is $(k-1)$-independent, it follows that each $\Tind(K_i,\cN)$ is $(k-1)$-independent and has
  measure $p_i$, which implies that the $\TcColkHyp$-on $\cH$ defined by $\cH_{E_i}\df\Tind(K_i,\cN)$
  ($i\in[c]$) satisfies $\phi_\cH = \psi_{k,p}$ and since clearly $I(\cH) = \cN$, it follows that
  $\psi_{k,p}^I=\phi$.

  If $c = 1$, then we can define $I$ by replacing~\eqref{eq:fullIndependence:topI} with
  \begin{align*}
    I(P)(x_1,\ldots,x_k) & \df \bigwedge_{1\leq i < j\leq k(P)} x_i \neq x_j & (P\in\cL', \id_k\in R_P(K_1));\\
    I(P)(x_1,\ldots,x_k) & \df x_1 \neq x_1 & (P\in\cL', \id_k\notin R_P(K_1))
  \end{align*}
  instead and we trivially get $\phi = \psi_{k,p}^I$ for any $p\in\Pi_{c'}$ with $c'\geq 2$ as we must have
  $\Tind(K_1,\cN)=\cE_k$ a.e.
\end{proofof}

Before we show Theorem~\ref{thm:fullUCouple}, let us first see that the $(\Theta,p)$-quasirandom homomorphisms
$\psi_{\Theta,p}\in\HomT{T_\Theta}$ from Definition~\ref{def:actiontheories} are well-defined (i.e., their
definition as $\psi_{\Theta,p}\df\phi_{\cN^Z}$ is independent of the choice of $Z$) and satisfy $\UCouple[k-1]$.

\begin{proposition}\label{prop:ThetapQR}
  With the notation and conditions of Definition~\ref{def:actiontheories}, we have
  \begin{align}\label{eq:ThetapQR}
    \phi_{\cN^Z}(\langle M\rangle) & = \prod_{P\in\cL} p_P^{\lvert R_P(M)\rvert/k!}
  \end{align}
  for every $M\in\cM[T_\Theta]$. Furthermore, $\psi_{\Theta,p}\df\phi_{\cN^Z}$ satisfies $\UCouple[k-1]$.
\end{proposition}

\begin{proof}
  First, let us show that $\cN^Z$ is indeed a $T_\Theta$-on.

  Note first that $T_\Theta$ trivially proves that
  \begin{align}\label{eq:subsclosed}
    \neg P(x,y,\ldots,t) \qquad (P\in\cL,\text{the tuple $(x,y,\ldots,t)$ contains repeated variables})
  \end{align}
  and if we add~\eqref{eq:subsclosed} to the axioms of $T_\Theta$, then it becomes substitutionally closed
  (see~\cite[Definition~3.5, Remark~5]{CR19}), then by~\cite[Theorem~3.7]{CR19}, to show that $\cN^Z$ is a
  $T_\Theta$-on, it is enough to show that $\cN^Z$ satisfies the axioms of $T_\Theta$
  and~\eqref{eq:subsclosed} a.e. It is trivial that~$\cN^Z$ satisfies~\eqref{eq:subsclosed} a.e.

  Note that the fact that $Z$ is a partition implies that there exists a unique $P_x\in\cL$ such that
  $x_{[k]}\in Z_{P_x}$, thus there exists a unique $Q_x\in\cL$ such that $x\in\cN^Z_{Q_x}$, namely $Q_x =
  \sigma_x^{-1}\cdot P_x$ (where $\sigma_x$ is as in Definition~\ref{def:actiontheories}). This implies that
  $\cN^Z$ satisfies axioms~\eqref{eq:action:existsone} and~\eqref{eq:action:unique} a.e.

  Note now that if $\tau\in S_k$, then we have $\sigma_{x\cdot\tau} = \sigma_x\comp\tau$, hence
  \begin{align*}
    x\cdot\tau\in\cN^Z_P & \equiv x_{[k]}\in Z_{\sigma_{x\cdot\tau}\cdot P}
    \equiv x_{[k]}\in Z_{\sigma_x\cdot(\tau\cdot P)} \equiv x\in\cN^Z_{\tau\cdot P},
  \end{align*}
  so $\cN^Z$ also satisfies axiom~\eqref{eq:action:action} a.e., hence $\cN^Z$ is a $T_\Theta$-on.

  \medskip

  Let $\rn{K}$ be the exchangeable array corresponding to $\cN^Z$ with respect to $\rn{\theta}$ picked in
  $\cE_{\NN_+}$ according to $\lambda$. Since for $m\in\NN$ and $K\in\cK_m[T_\Theta]$, we have
  $\phi_{\cN^Z}(\langle K\rangle) = \PP[\rn{K}\rest_{[m]} = K]$, if we show that for every measurable
  $U\subseteq\cE_{m,k-1}$ with $\lambda(U)>0$, we have
  \begin{align}\label{eq:Thetapobjective}
    \PP[\rn{K}\rest_{[m]} = K \mid E]
    & =
    \prod_{P\in\cL} p_P^{\lvert R_P(K)\rvert/k!},
  \end{align}
  where $E$ is the event $(\rn{\theta}_B \mid B\in r(m,k-1))\in U$, then both~\eqref{eq:ThetapQR} and
  $\psi_{\Theta,p}\in\UCouple[k-1]$ will follow (the former follows by taking $U=\cE_{m,k-1}$ and the latter
  implies weak $(k-1)$-independence of $\cN^Z$, which is equivalent to $\phi_{\cN^Z}\in\UCouple[k-1]$ by
  Theorem~\ref{thm:UCouple}).

  If $m < k$, \eqref{eq:Thetapobjective} trivially holds, so suppose $m\geq k$ and note that the axioms of
  $T_\Theta$ imply that for each $\alpha\injection{[k]}{[m]}$, there exists a unique $P_\alpha\in\cL$ such
  that $\alpha\in R_{P_\alpha}(K)$ and we must further have $P_\alpha = \tau\cdot P_{\alpha\comp\tau}$ for
  every $\tau\in S_k$. Note that for any choice of $(\alpha_A)_{A\in\binom{[m]}{k}}$ with
  $\alpha_A\injection{[k]}{[m]}$ and $\im(\alpha_A)=A$, we have
  \begin{align*}
    \PP[\rn{K}\rest_{[m]} = K\mid E]
    & =
    \PP\left[\forall\alpha\in([m])_k, \alpha\in R_{P_\alpha}(\rn{K})
      \:\middle\vert\:
      E\right]
    \\
    & =
    \PP\left[\forall A\in\binom{[m]}{k}, \alpha_A\in R_{P_{\alpha_A}}(\rn{K})
      \:\middle\vert\:
      E\right].
  \end{align*}
  Now, the event $\alpha_A\in R_{P_{\alpha_A}}(\rn{K})$ depends only on the relative order of
  $(\rn{\theta}_{\{i\}} \mid i\in A)$ and on the variable $\rn{\theta}_A$ and, since $p$ is
  $\Theta$-invariant, we have $\lambda(Z_{\sigma\cdot P_\alpha}) = p_{P_\alpha}$ for every $\sigma\in S_k$ and
  every $\alpha\injection{[k]}{[m]}$. This means that if $\leq$ is an ordering of $A$ and $E_\leq$ is the
  event that says that the relative order of $(\rn{\theta}_{\{i\}} \mid i\in A)$ is $\leq$, then
  $\PP[\alpha\in R_{P_\alpha}(\rn{K}) \mid E\land E_\leq] = p_{P_\alpha}$ and thus
  \begin{align*}
    \PP[\rn{K}\rest_{[m]} = K \mid E]
    & =
    \prod_{A\in\binom{[m]}{k}} p_{P_{\alpha_A}}.
  \end{align*}
  Since this holds for any choice of $(\alpha_A)_{A\in\binom{[m]}{k}}$ with $\im(\alpha_A)=A$, by considering
  all possible $k!^{\binom{m}{k}}$ such choices we get
  \begin{align*}
    \PP[\rn{K}\rest_{[m]} = K \mid E]^{k!^{\binom{m}{k}}} & = \prod_{P\in\cL} p_P^{k!^{\binom{m}{k} - 1}\cdot\lvert R_P(K)\rvert},
  \end{align*}
  from which~\eqref{eq:Thetapobjective} follows.
\end{proof}

\begin{definition}
  Given a $T$-on $\cN$ over $\Omega=(X,\cA,\mu)$ and $K\in\cK_V[T]$, let $W^K_\cN\function{\cE_{V,\lvert
      V\rvert - 1}(\Omega)}{[0,1]}$ be defined by
  \begin{align*}
    W^K_\cN(x) & \df \mu(\{y\in X\mid (x,y)\in\Tind(K,\cN)\}).
  \end{align*}
\end{definition}
Note that $W^K_\cN$ is essentially a $(\lvert V\rvert - 1)$-flattening of the peon
$\Tind(K,\cN)\subseteq\cE_V(\Omega)$ (see Definition~\ref{def:flattening}).

The next two simple lemmas are fundamental in the proof of
Theorem~\ref{thm:fullUCouple}.

\begin{lemma}\label{lem:densityfromWKcN}
  Let $k\in\NN_+$ and suppose that $k(P)\leq k$ for all $P\in\cL$. Let $T$ be a theory over $\cL$ and $\cN$ be
  a $T$-on over $\Omega=(X,\cA,\mu)$. Then for every $m\in\NN$ and every $K\in\cK_m[T]$, we have
  \begin{align*}
    \phi_\cN(\langle K\rangle) & = \int_{X^{r(m,k-1)}} \prod_{A\in\binom{[m]}{k}} W^{K\rest_A}_\cN(\pi_A(x))\ d\mu(x),
  \end{align*}
  where $\pi_A\function{\cE_{m,k-1}(\Omega)}{\cE_{A,k-1}(\Omega)}$ is the projection on the coordinates
  indexed by $r(A,k-1)$.
\end{lemma}

\begin{proof}
  Follows by considering the exchangeable array corresponding to $\cN$ with respect to $\rn{\theta}$ picked in
  $\cE_{\NN_+}(\Omega)$ according to $\mu$, noting that $\rn{K}\rest_{[m]}=K$ is equivalent to $\forall
  A\in\binom{[m]}{k},\rn{K}\rest_A=K\rest_A$ (since $k(P)\leq k$ for every $P\in\cL$) and integrating out
  the top variables $(\rn{\theta}_A \mid A\in\binom{[m]}{k})$.
\end{proof}

\begin{lemma}\label{lem:UCoupleWKcNaeconstant}
  If a $T$-on $\cN$ over $\Omega$ is such that $\phi_\cN$ satisfies $\UCouple[\ell]$ and $K\in\cK_V[T]$ with
  $\lvert V\rvert\leq\ell+1$, then $W^K_\cN$ is a.e.\ constant.
\end{lemma}

\begin{proof}
  Without loss of generality, we may suppose that $V = [m]$. Write $\Omega=(X,\cA,\mu)$. Then it is sufficient
  to show that for every measurable $U\in\cE_{m,\ell}(\Omega)$, we have $\int_U W^K_\cN\ d\mu =
  \mu(U)\phi_\cN(\langle K\rangle)$. But for the exchangeable array $\rn{K}$ corresponding to $\cN$ with
  respect to $\rn{\theta}$ picked in $\cE_{\NN_+}(\Omega)$ according to $\mu$, it follows that
  \begin{multline*}
    \int_U W^K_\cN\ d\mu
    =
    \PP[\rn{K}\rest_{[m]} = K\land (\rn{\theta}_A \mid A\in r(m,k-1))\in U]
    \\
    =
    \PP[\rn{K}\rest_{[m]} = K]\cdot\PP[(\rn{\theta}_A \mid A\in r(m,k-1))\in U]
    =
    \mu(U)\phi_\cN(\langle K\rangle),
  \end{multline*}
  where the second equality follows since $\cN$ is weakly $\ell$-independent by Theorem~\ref{thm:UCouple}.
\end{proof}

\begin{proofof}{Theorem~\ref{thm:fullUCouple}}
  The backward direction follows from Proposition~\ref{prop:ThetapQR} and Theorem~\ref{thm:naturality}.

  For the forward direction, we will show that in fact we can take $p = (p_P)_{P\in\cL}$ satisfying
  $p_P > 0$ for every $P\in\cL$. Note that when $p_P > 0$ for every $P\in\cL$, we have $\psi_{\Theta,p}(M) >
  0$ for every $M\in\cM[T_\Theta]$, so by an argument analogous to that of the proof of
  Theorem~\ref{thm:fullIndependence}, it is enough to consider the case when $T = T_\cL$.

  \medskip

  Suppose then that $T = T_\cL$ and let $\cN$ be a $T$-on such that $\phi_\cN = \phi$. Note that if $P\in\cL$
  is such that $k(P)\leq k-1$, then $\rk(\cN_P)\leq k-1$, so by Theorem~\ref{thm:naturality} and
  Proposition~\ref{prop:rankbasic}, it follows that $\rk(\cN_P)=0$, that is, $\lambda(\cN_P)\in\{0,1\}$. This
  means that we can write $\cL=\widehat{\cL}\cup\cL_0\cup\cL_1$, where
  \begin{align*}
    \widehat{\cL} & \df \{P\in\cL \mid k(P) = k\};\\
    \cL_i & \df \{P\in\cL \mid k(P)\leq k-1\land\lambda(\cN_P) = i\} \qquad (i\in\{0,1\}).
  \end{align*}

  Consider the (left) action of $S_k$ on $\cK_k[\Th(\phi)]$ given by letting $\sigma\cdot
  K\in\cK_k[\Th(\phi)]$ ($\sigma\in S_k$, $K\in\cK_k[\Th(\phi)]$) be the model obtained from $K$ by permuting
  its vertices by $\sigma$, that is, we have
  \begin{align*}
    R_P(\sigma\cdot K) & \df \{\sigma\comp\alpha \mid \alpha\in R_P(K)\} & (P\in\widehat{\cL});\\
    R_P(\sigma\cdot K) & \df \emptyset & (P\in\cL_0);\\
    R_P(\sigma\cdot K) & \df ([k])_{k(P)} & (P\in\cL_1).
  \end{align*}
  Note that this definition ensures that for a.e.~$x\in\cE_k$ and every $\sigma\in S_k$, we have
  \begin{align}\label{eq:actioncN}
    x\cdot\sigma\in\Tind(K,\cN) & \equiv x\in\Tind(\sigma\cdot K,\cN).
  \end{align}
  It is also clear that for a.e.~$x\in\cE_k$, there exists exactly one $K\in\cK_k[\Th(\phi)]$ such that
  $x\in\Tind(K,\cN)$.

  Let then $\cL'$ be a language containing one predicate symbol $P_K$ of arity $k$ for each
  $K\in\cK_k[\Th(\phi)]$ and let $\Theta\function{S_k\times\cL'}{\cL'}$ be the induced action $\sigma\cdot
  P_K\df P_{\sigma\cdot K}$ ($\sigma\in S_k$, $K\in\cK_k[\Th(\phi)]$). Define then $\cH$ by
  \begin{align*}
    \cH_{P_K} & \df \Tind(K,\cN)
  \end{align*}
  and note that~\eqref{eq:actioncN} and the remark below it ensure that $\cH$ is a $T_\Theta$-on.

  Define $I\interpret{T}{T_\Theta}$ by
  \begin{align*}
    I(P)(x_1,\ldots,x_{k(P)}) & \df
    \begin{dcases*}
      \bigvee_{\substack{K\in\cK_k[\Th(\phi)]\\ \id_k\in R_P(K)}} P_K(x_1,\ldots,x_{k(P)}), & if $P\in\widehat{\cL}$;\\
      x_1\neq x_1, & if $P\in\cL_0$;\\
      \bigwedge_{1\leq i < j\leq k(P)} x_i\neq x_j, & if $P\in\cL_1$.
    \end{dcases*}
  \end{align*}
  and note that we trivially have $I(\cH)_P = \cN_P$ a.e.\ for every $P\in\cL$, hence $\phi_\cH^I = \phi$.

  For every $K\in\cK_k[\Th(\phi)]$, let $p_{P_K}\df\lambda(\cH_{P_K}) = \phi(\langle K\rangle) > 0$ and note
  that the definition of $\Theta$ implies that $p$ is $\Theta$-invariant and $\sum_{K\in\cK_k[\Th(\phi)]}
  p_{P_K} = 1$. To conclude the proof, we will show that $\phi_\cH = \psi_{\Theta,p}$. To do so, for every
  $K\in\cK_k[\Th(\phi)]$, let $K_K\in\cK_k[T_\Theta]$ be the unique model such that $\id_k\in R_{P_K}(K_K)$
  and note that the axioms of $T_\Theta$ imply that $W^{K_K}_\cH$ is a.e.\ equal to the $(k-1)$-flattening
  $W^{k-1}_{\cH_{P_K}}$ of the peon $\cH_{P_K}$, which in turn is a.e.\ equal to $W^K_\cN$. But then from
  Lemma~\ref{lem:UCoupleWKcNaeconstant}, it follows that $W^{K_K}_\cH = \phi(\langle K\rangle) = p_{P_K}$
  a.e. Since the $T_\Theta$-on $\cN^Z$ of Definition~\ref{def:actiontheories} and
  Proposition~\ref{prop:ThetapQR} also clearly satisfies $W^{K_K}_{\cN^Z} = W^{k-1}_{\cN^Z_{P_K}} = p_{P_K}$
  a.e., from Lemma~\ref{lem:densityfromWKcN}, it follows that $\phi_\cH = \phi_{\cN^Z} = \psi_{\Theta,p}$.
\end{proofof}

\section{Conclusion and open problems}
\label{sec:concl}

In this paper we have attempted to build a general theory of quasi-randomness
that is uniformly applicable to arbitrary combinatorial structures and is invariant
under their ``natural transformations''. While our basic definitions deliberately
avoided mentioning specific densities,
it turned out, in the vein of the previous research in the area, that our quasi-random
properties can be characterized in several equivalent ways, {\em including} such
densities. We have shown how to arrange these properties into a hierarchy and,
with one or two notable exceptions, have been able to prove that this hierarchy is
proper. Finally, we have compared our quasi-random properties to what has been
studied before for hypergraphs (with the focus on specific densities) and have
found that these two frameworks are essentially incomparable.

\medskip

One topic that we touched tangentially in the proof of Theorem~\ref{thm:UCouple}, more specifically with
Example~\ref{ex:nonclosed} and Lemma~\ref{lem:ucL1closed}, is the closedness of our properties with respect to
both the density topology and $L_1$-topology (Definition~\ref{def:L1}). The aforementioned example and lemma show
that in general unique coupleability with a particular collection of limit objects is closed in
$L_1$-topology but not necessarily closed in the density topology. On the other hand, alternative
syntactic descriptions of $\UCouple[\ell]$ and $\UInduce[\ell]$ (as $\ell$-locality and
symmetric $\ell$-locality, respectively) imply that these classes are closed even in the
density topology. So in a sense we have a satisfactory overall picture for the classes
based on the ``extrinsic'' notion of coupleability.

Remarkably, we do not know the answer for the class $\Independence[\ell]$, even if it has a
very clean and natural ``intrinsic'' definition. This is the first question we
would like to ask: is $\Independence[\ell]$ closed in the density, or at least
$L_1$-topology? One sensible approach to this question might consist in developing
an alternative, and perhaps more concrete, characterization of this class that
might be interesting in its own right.

\medskip

If $\phi_1$ and $\phi_2$ are uniquely coupleable with {\em all} theons of rank $\leq\ell$,
then the same is true for $\phi_1\otimes\phi_2$ (Theorem~\ref{thm:naturalityindcoup}~\ref{thm:naturalityindcoup:UCouple}).
We do not know if the same remains true after replacing this class of tests with
individual tests, and when we needed this in one of our proofs, we had to take
a considerable detour (see item~\ref{it:quasirandomcolhyp} in our program at the beginning of Section~\ref{sec:UCouple}). Thus comes our second open question: assume that $\phi_1$
and $\psi$, as well as $\phi_2$ and $\psi$ are uniquely coupleable. Does it imply that
$\phi_1\otimes \phi_2$ is also uniquely coupleable with $\psi$?

Under the additional assumption that $\phi_1,\phi_2 $ are themselves uniquely
coupleable, the question takes a particularly nice and symmetric form: assume
that $\phi_1$, $\phi_2$ and $\phi_3 (=\psi)$ are pairwise uniquely coupleable.
Does it imply that $\phi_1,\phi_2,\phi_3$ are (mutually) uniquely coupleable?
While the analogy with independence for random variables is now visible, it is
not immediately clear how useful it might turn out here.

\medskip

Another interesting question is whether unique coupleability establishes a
Galois correspondence between
$\UCouple[\ell]$ and limit objects of rank at most $\ell$. In other words,
is it true that if $\phi$ is uniquely
coupleable with every $\psi\in\UCouple[\ell]$, then $\rk(\phi)\leq\ell$?

\medskip

As we mentioned before, the results of Theorems~\ref{thm:anti-monotone}, \ref{thm:inter-properties},
\ref{thm:separationupward}, \ref{thm:separationUCoupleIndependence} and~\ref{thm:separationUInduceUCouple}
almost complete the Hasse diagram of implications between the families $\Independence$, $\UCouple$ and
$\UInduce$. The only missing separations are the ones of the form
$\UCouple[\ell]\nRightarrow\Independence[\ell']$ when $\ell' < \ell$, and this
is our fourth question. Let us remark that the somewhat subtle theons we
introduced in Section~\ref{sec:useful} do not seem to work for this purpose.

\medskip
Recall that Theorem~\ref{thm:UCouple}\ref{thm:UCouple:UCouple}$\equiv$\ref{thm:UCouple:orderUInduce} says that
$\phi\in\UCouple[\ell]$ is equivalent to $\phi\otimes\psi\in\UInduce[\ell]$, where $\psi\in\HomT{\TLinOrder}$
is the linear order. Let us now draw attention to three interesting open problems that can be extracted from
this equivalence.

The first is whether a ``converse'' of this is true in the spirit of
 Theorems~\ref{thm:fullIndependence} and~\ref{thm:fullUCouple}: can every
$\phi\in\UInduce[\ell]$ be written as $\phi = (\widehat{\phi}\otimes\psi)^I$ for some
$\widehat{\phi}\in\UCouple[\ell]$ and some open interpretation $I\interpret{T}{T'\cup\TLinOrder}$?

The second problem is an analogue of Theorems~\ref{thm:fullIndependence} and~\ref{thm:fullUCouple} themselves
in the context of unique inducibility. We conjecture that if all arities are at most $k$, then $\phi\in\UInduce[k-1]$ should
be equivalent to $\phi = (\phi_{\Theta,p}\otimes\psi)^I$ for some action
$\Theta\function{S_k\times\cL'}{\cL'}$ on a language $\cL'$, some open interpretation
$I\interpret{T}{T_\Theta\cup\TLinOrder}$ and some $\Theta$-invariant $p\in[0,1]^{\cL'}$ (of course, this would follow from a positive answer to the previous problem).

\smallskip

The third question is more open-ended and concerns the quasirandom permuton (i.e., $\psi\otimes\psi$ for the unique $\psi\in\HomT{\TLinOrder}$). It does
not even satisfy the weakest of our properties $\UInduce[1]$ (see remark after
Theorem~\ref{thm:naturalityindcoup}), so we can rightfully wonder: why do we
think about this object as quasi-random at all, to start with?
 A simple but unsatisfactory answer in the vein of (again) Theorems~\ref{thm:fullIndependence} and~\ref{thm:fullUCouple} is that it can be written as $\phi=(\widehat{\phi}\otimes\psi)^I$ for some
$\widehat{\phi}\in\UInduce[\ell]$ and some $I\interpret{T}{T'\cup\TLinOrder}$, and
we could form the class of all objects representable this way. Are there any
``reasonable'' descriptions of this class, extrinsic or intrinsic? Note that if
the conjectures from the previous two paragraphs are true, this would also form another
interesting hierarchy in an orthogonal direction: we can get progressively weaker families of natural
quasirandomness properties by taking independent coupling with the linear order.

Another possible approach is to start from the known permutation quasirandomness
theory~\cite{Coo04,Coo05,KP13,CKNPSV20}. However, in comparison to their (hyper)graph and tournament
counterparts, the theory of permutation quasirandomness provides a much smaller variety of quasirandomness
formulations as candidates for natural generalizations, essentially boiling down to only three types: explicit
density notions, discrepancy notions based on intervals and spectral notions. Let us also note that
there is still a whole host of properties~\cite{DEG14,CD17} that random permutations satisfy and
that have not yet been fully explored in the quasirandom setting. In fact, some
of these properties are so fine-grained that it is not even clear if they can be
encoded by subpermutation densities.

\medskip

The notions of rank and $\Independence$ have the following generalization: for $B\subseteq\NN_+$, let us say
that a peon $\cN$ over $\Omega=(X,\cA,\mu)$ is \emph{$B$-compatible} if it only depends on coordinates that
are indexed by sets $A$ with $\lvert A\rvert\in B$, that is, it can be written as $\cN = \cG\times
X^{\bigcup_{b\in[k(P)]\setminus B}\binom{[k(P)]}{b}}$ for some $\cG\subseteq X^{\bigcup_{b\in
    B}\binom{[k(P)]}{b}}$. Let us say that an Euclidean structure is \emph{$B$-compatible} if all its peons
are so and let us say that $\phi\in\HomT{T}$ is \emph{$B$-compatible} if it has a $T$-on representation that
is $B$-compatible. Then rank at most $k$ amounts to $[k]$-compatibility and $\ell$-independence amounts to
$(\NN_+\setminus[\ell])$-compatibility. We believe that with a careful inductive application of the theon
uniqueness theorems~\cite[Theorems~3.9 and~3.11, Proposition~7.7]{CR19}, one could generalize the proof of weak
independence to show that if $\phi_1$ and $\phi_2$ are $B_1$-compatible and $B_2$-compatible, respectively and
$B_1\cap B_2 = \emptyset$, then $\phi_1$ and $\phi_2$ are uniquely coupleable. However, we know that
$\UCouple[\ell]$, i.e., unique coupleability with all $[\ell]$-compatible limit objects, is strictly weaker
than $\Independence[\ell]$, so it is natural to ask if the weak independence analogue of $(\NN_+\setminus
B)$-compatibility (i.e., asking the exchangeable array $\rn{K}$ to be independent from $(\rn{\theta}_A \mid
\lvert A\rvert\in B)$ as a random variable) also yields a strictly weaker property than $(\NN_+\setminus
B)$-compatibility when $B$ is not of the form $[k]$ for some $k\in\NN$. In particular, this involves studying
unique coupleability with all $\psi\in\Independence[\ell]$ as well. Building on that, it is also natural to
ask if there are examples of uniquely coupleable $\phi_1$ and $\phi_2$ that do not fall in this
$B$-compatibility setting or in its weak independence analogue.

\bibliographystyle{alpha}
\bibliography{refs}

\appendix

\section{The $L_1$-topology}
\label{sec:L1topology}

\begin{lemma}\label{lem:L1topology}
  The $L_1$-distance $\delta_1$ is a metric on $\HomT{T}$ and generates a finer topology than the density
  topology.
\end{lemma}

\begin{proof}
  Let us first check the triangle inequality. Let $\xi$ be a coupling of $\phi_1$
  and $\phi_2$ and $\zeta$ be a coupling of $\phi_2$ and $\phi_3$ attaining the $L_1$-distances
  in~\eqref{eq:d1coupling}. Let also $J_i\interpret{T}{T\cup T}$ be the structure-erasing interpretation
  corresponding to coordinate $i$ and $I_{ij}\interpret{T\cup T}{T\cup T\cup T}$ be the structure-erasing
  interpretation corresponding to coordinates $i$ and $j$. Since $\xi$ is a coupling of $\phi_1$ and $\phi_2 =
  \zeta^{J_1}$, Proposition~\ref{prop:lifting} gives us a coupling $\widehat{\xi}$ of $\phi_1$ and $\zeta$
  such that $\widehat{\xi}^{\id_T\cup J_1} = \xi$. Since $\id_T\cup J_1 = I_{12}$, we get that $\widehat{\xi}$
  is a coupling of $\phi_1$, $\phi_2$ and $\phi_3$ such that $\widehat{\xi}^{I_{12}} = \xi$ and
  $\widehat{\xi}^{I_{23}} = \zeta$. But $\widehat{\xi}^{I_{13}}$ is a coupling of $\phi_1$ and $\phi_3$ and
  for each $P\in\cL$ we have
  \begin{align*}
    \widehat{\xi}^{I_{13}}(d_P)
    & \leq
    \widehat{\xi}^{I_{12}}(d_P) + \widehat{\xi}^{I_{23}}(d_P),
  \end{align*}
  hence by~\eqref{eq:d1coupling} we get $\delta_1(\phi_1,\phi_3)\leq \delta_1(\phi_1,\phi_2) +
  \delta_1(\phi_2,\phi_3)$.

  \medskip

  Finally, note that by~\eqref{eq:d1} we have
  \begin{align*}
    \lvert\phi_1(\langle M\rangle) - \phi_2(\langle M\rangle)\rvert
    & \leq
    \delta_1(\phi_1,\phi_2)\sum_{P\in\cL} (\lvert M\rvert)_{k(P)},
  \end{align*}
  for every $M\in\cM[T]$. This implies both $\delta_1(\phi_1,\phi_2) = 0 \implies \phi_1=\phi_2$
  and that the $L_1$-topology is finer than the density topology.
\end{proof}

\end{document}